\documentclass[onecolumn,final,a4paper]{elsart3p}

\def\ElsevierStyle

\usepackage{float}
\usepackage{comment}
\usepackage{amsmath}       
\usepackage{amssymb}       
\usepackage{amscd}

\usepackage{chapterbib}    
\usepackage{color}
\usepackage{comment}
\usepackage{bm,bbm}
\usepackage{overpic}

\usepackage{times}

\usepackage{epsfig}
\usepackage{float}
\usepackage{graphicx,graphics,rotating}
\usepackage{subfigure}
\usepackage{multicol}
\usepackage{multirow}

\usepackage{booktabs,siunitx}

\newcommand{\citeappx}[1]{}

\newtheorem{proof}{Proof}
\newtheorem{theorem}{Theorem}[section]
\newtheorem{lemma}[theorem]{Lemma}
\newtheorem{corollary}[theorem]{Corollary}

\newtheorem{remark}[theorem]{Remark}

\newcommand{\huvhI}{\huvh^I}
\newcommand{\uvhI}{\uvh^I}
\newcommand{\vvhI}{\vvh^I}
\newcommand{\BvhI}{\Bvh^I}
\newcommand{\CvhI}{\Cvh^I}
\newcommand{\hEshI}{\hEsh^I}

\newcommand{\DshI}{\Dsh^I}
\newcommand{\pshI}{\psh^I}
\newcommand{\qshI}{\qsh^I}
\newcommand{\xivh}{\bm{\xi}_{\hh}}
\newcommand{\etavh}{\bm{\eta}_{\hh}}

\newcommand{\DIVh}{\DIV_{\hh}}
\newcommand{\ROTvh}{\ROTv_{\hh}}

\newcommand{\EOD}{\end{document}}
\newcommand{\TERM}[1]{\textbf{(#1)}}

\newcommand{\DOFS} [2]{$\mathbf{(#1_{#2})}$}

\newcommand{\DG} {DG}

\newcommand{\REAL}  {\mathbbm{R}}

\newcommand{\bv}{\bm{b}}
\renewcommand{\cv}{\bm{c}}

\newcommand{\fv}{\bm{f}}
\newcommand{\gv}{\bm{g}}

\newcommand{\nv}{\bm{n}}

\newcommand{\qv}{\bm{q}}

\newcommand{\tv}{\bm{t}}
\newcommand{\uv}{\bm{u}}
\newcommand{\vv}{\bm{v}}
\newcommand{\wv}{\bm{w}}
\newcommand{\xv}{\bm{x}}
\newcommand{\yv}{\bm{y}}

\newcommand{\Bv}{\bm{B}}
\newcommand{\Cv}{\bm{C}}
\newcommand{\Dv}{\bm{D}}
\newcommand{\Ev}{\bm{E}}

\newcommand{\Jv}{\bm{J}}

\newcommand{\as}{a}
\newcommand{\bs}{b}
\newcommand{\cs}{c}

\newcommand{\fs}{f}
\newcommand{\gs}{g}

\newcommand{\ms}{m}

\newcommand{\ps}{p}
\newcommand{\qs}{q}

\renewcommand{\ss}{s}
\newcommand{\ts}{t}
\newcommand{\us}{u}
\newcommand{\vs}{v}
\newcommand{\ws}{w}
\newcommand{\xs}{x}
\newcommand{\ys}{y}

\newcommand{\Bs}{B} 

\newcommand{\Cs}{C}
\newcommand{\Ds}{D}
\newcommand{\Es}{E}

\newcommand{\Gs}{G}

\newcommand{\Js}{J}

\newcommand{\Ms}{M}

\newcommand{\Ps}{P}

\newcommand{\Ts}{T}

\newcommand{\xsE}{x_{\E}}
\newcommand{\ysE}{y_{\E}}

\newcommand{\xsP}{x_{\P}}
\newcommand{\ysP}{y_{\P}}

\newcommand{\matA}{\mathbbm{A}}

\newcommand{\calD}{\mathcal{D}}
\newcommand{\calE}{\mathcal{E}}

\newcommand{\calG}{\mathcal{G}}
\newcommand{\calH}{\mathcal{H}}
\newcommand{\calI}{\mathcal{I}}

\newcommand{\calK}{\mathcal{K}}
\newcommand{\calL}{\mathcal{L}}

\newcommand{\calP}{\mathcal{P}}

\newcommand{\calR}{\mathcal{R}}
\newcommand{\calS}{\mathcal{S}}
\newcommand{\calT}{\mathcal{T}}

\newcommand{\calV}{\mathcal{V}}
\newcommand{\calX}{\mathcal{X}}

\newcommand{\calEP}{\calE_{\P}}

\newcommand{\xiv}{{\bm\xi}}

\newcommand{\KER}{\textrm{ker}}

\newcommand{\restrict}[2]{{#1}_{|#2}}

\newcommand{\ABS}    [1]{|#1|}

\newcommand{\scal}   [2]{\big (#1,#2\big)}
\newcommand{\Scal}   [2]{\Big (#1,#2\Big)}

\newcommand{\scalP}  [2]{\big (#1,#2\big)_{\P}}
\newcommand{\ScalP}  [2]{\Big (#1,#2\Big)_{\P}}

\newcommand{\scalVh}  [2]{\big (#1,#2\big)_{\Vh}}
\newcommand{\ScalVh}  [2]{\Big (#1,#2\Big)_{\Vh}}

\newcommand{\scalEh}  [2]{\big (#1,#2\big)_{\Eh}}
\newcommand{\ScalEh}  [2]{\Big (#1,#2\Big)_{\Eh}}

\newcommand{\scalTVh}   [2]{\big(#1,#2\big)_{\TVh}}
\newcommand{\scalTVhP}  [2]{\big(#1,#2\big)_{\TVh(\P)}}
\newcommand{\scalTVhNa} [2]{\big[#1,#2\big]_{\TVh}}
\newcommand{\scalTVhPNa}[2]{\big[#1,#2\big]_{\TVh(\P)} }

\newcommand{\ScalTVh}   [2]{\Big(#1,#2\Big)_{\TVh}}

\newcommand{\ScalTVhNa} [2]{\Big[#1,#2\Big]_{\TVh}}

\newcommand{\scalPh}  [2]{\big (#1,#2\big)_{\Ph}}
\newcommand{\ScalPh}  [2]{\Big (#1,#2\Big)_{\Ph}}

\newcommand{\scalVhP} [2]{\big (#1,#2\big)_{\Vh(\P)}}

\newcommand{\scalEhP}  [2]{\big (#1,#2\big)_{\Eh(\P)}}

\newcommand{\nlen}{\hspace{-0.2mm}}

\newcommand{\snorm}  [2]{\big\vert#1\big\vert_{#2}}

\newcommand{\NORM}   [2]{\|#1\|_{#2}}

\newcommand{\norm}   [2]{\big\|#1\big\|_{#2}}

\newcommand{\TNORM}  [2]{\vert\nlen\vert\nlen\vert#1\vert\nlen\vert\nlen\vert_{{}_{#2}}}

\newcommand{\Dt}{\Delta\ts}

\def\trait #1 #2 #3 {\vrule width #1pt height #2pt depth #3pt}
\def\fin{\hfill
        \trait .3 5 0
        \trait 5 .3 0
        \kern-5pt
        \trait 5 5 -4.7
        \trait 0.3 5 0
\medskip}
\newcommand{\ENDPROOF}{\fin}
\newcommand{\BEGINPROOF}{\emph{Proof}.~~}

\usepackage{scalerel,stackengine}
\stackMath
\newcommand\reallywidehat[1]{%
\savestack{\tmpbox}{\stretchto{%
  \scaleto{%
    \scalerel*[\widthof{\ensuremath{#1}}]{\kern-.6pt\bigwedge\kern-.6pt}%
    {\rule[-\textheight/2]{1ex}{\textheight}}
  }{\textheight}%
}{0.5ex}}%
\stackon[1pt]{#1}{\tmpbox}%
}

\newcommand{\DIV} {\textrm{div}\,}

\newcommand{\ROT} {\textrm{rot}\,}

\newcommand{\ROTv}{\mathbf{rot}\,}

\newcommand{\HONE}  {H^1}
\newcommand{\HONEzr}{H^1_0}

\newcommand{\LTWO}  {L^2}
\newcommand{\LTWOzr}{L^2_0}

\newcommand{\PS}[1] {\mathbbm{P}_{#1}}

\newcommand{\PSv}[1] {\mathbf{P}_{#1}}

\newcommand{\RT} [1]{\textrm{RT}_{#1}}
\newcommand{\RTz}   {\textrm{RT}_{0}}

\newcommand{\CS}[1] {C^{#1}}

\newcommand{\HDIV}   [1]{H  (\textrm{div}; #1)}
\newcommand{\HROT}   [1]{H  (\textrm{rot}; #1)}
\newcommand{\HROTv}  [1]{H  (\textbf{rot}; #1)}

\newcommand{\HROTvzr}[1]{H_0(\textbf{rot}; #1)}

\newcommand{\Honev}  [1]{\left[H^1  (#1)\right]^2}
\newcommand{\Honevzr}[1]{\left[H_0^1  (#1)\right]^2}

\newcommand{\Gperp}  [2]{\calG_{#1}^{\perp}\left(#2\right)}
\newcommand{\Bspace} [1]{\mathbbm{B}\left(#1\right)}
\newcommand{\Gspace} [2]{\calG_{#1}\left(#2\right)}

\newcommand{\Vvh} {\textbf{V}_{\hh}}
\newcommand{\Uvh} {\textbf{U}_{\hh}}

\newcommand{\PirP}{\Pi^{\ROTv}_{\P}}

\newcommand{\PinP}{\Pi^{\nabla}_{\P}}

\newcommand{\PizP}{\Pi^{0}_{\P}}

\renewcommand{\P} {\textsf{P}}            
\newcommand  {\E} {\textsf{e}}            
\newcommand  {\V} {\textsf{v}}            
\newcommand  {\T} {\textsf{T}}            

\newcommand{\Vp}  {\V_{1}}
\newcommand{\Vpp} {\V_{2}}

\newcommand{\hh}{h}
\newcommand{\Th}{\Omega_{\hh}}

\newcommand{\xvP}{\xv_{\P}}        
\newcommand{\xvE}{\xv_{\E}}        
\newcommand{\xvV}{\xv_{\V}}        

\newcommand{\hP}{\hh_{\P}}

\newcommand{\hE}{\hh_{\E}}

\newcommand{\mP}{\ABS{\P}}

\newcommand{\mE}{\ABS{\E}}

\newcommand{\nor}  {\mathbf{n}}

\newcommand{\fsh}{\fs_{\hh}}

\newcommand{\psh}{\ps_{\hh}}

\newcommand{\qsh}{\qs_{\hh}}

\newcommand{\Esh}{\Es_{\hh}}

\newcommand{\hEs}{\widehat{\Es}}
\newcommand{\hEsh}{\widehat{\Es}_{\hh}}
\newcommand{\Dsh}{\Ds_{\hh}}

\newcommand{\gsh}{\gs_{\hh}}

\newcommand{\gvp}{\gv^\perp}

\newcommand{\hJsh}{\widehat{\Js}_{\hh}}
\newcommand{\Jhb}{\Js_{\hh,b}}
\newcommand{\Jsh}{\Js_{\hh}}

\newcommand{\hJs}{\widehat{\Js}}
\newcommand{\zrv}{{\bm 0}}
\newcommand{\uvh}{\uv_{\hh}}
\newcommand{\huvh}{\widehat{\uv}_{\hh}}

\newcommand{\huv}{\widehat{\uv}}
\newcommand{\vvh}{\vv_{\hh}}

\newcommand{\xvh}{\xv_{\hh}}
\newcommand{\yvh}{\yv_{\hh}}

\newcommand{\Bvh}{\Bv_{\hh}}

\newcommand{\Cvh}{\Cv_{\hh}}

\newcommand{\Dvh}{\Dv_{\hh}}

\newcommand{\fvh}{\fv_{\hh}}

\newcommand{\ash}{\as_{\hh}}
\newcommand{\bsh}{\bs_{\hh}}

\newcommand{\ashzr}{\as_{\hh,0}}

\newcommand{\Vh}{\calV_{\hh}}
\newcommand{\Eh}{\calE_{\hh}}
\newcommand{\Ph}{\calP_{\hh}}
\newcommand{\TVh}{\mathcal{TV}_{\hh}}
\newcommand{\TVhzr}{\mathcal{TV}_{\hh,0}}
\newcommand{\Phzr}{\mathcal{P}_{\hh,0}}

\newcommand{\Vhzr}{\calV_{\hh,0}}

\newcommand{\IVh}{\calI^{\Vh}}
\newcommand{\IEh}{\calI^{\Eh}}
\newcommand{\IPh}{\calI^{\Ph}}

\newcommand{\ITVh}{\calI^{\TVh}}
\newcommand{\ITVhP}{\calI^{\TVh}_\P}

\newcommand{\IVhP}{\calI^{\Vh}_{\P}}
\newcommand{\IEhP}{\calI^{\Eh}_{\P}}

\newcommand{\PiVhPLS}{\Pi^{LS}_{\P}}
\newcommand{\PiVhPpw}{\Pi^{pw}_{\P}}

\newcommand{\Pih}{\Pi_h}
\newcommand{\PiEhRT}{\Pi^{RT}}
\newcommand{\PiEhRTP}{\Pi^{RT}_{\P}}

\newcommand{\Xh}{\calX_{\hh}}

\newcommand{\Xhzr}{\calX_{\hh,0}}
\newcommand{\GXh}{\mathfrak{X}_{\hh}}
\newcommand{\GXhzr}{\mathfrak{X}_{\hh,0}}
\newcommand{\CurlNorm}[1]{\TNORM{#1}{\Dt,\ROTv}}
\newcommand{\DivNorm} [1]{\TNORM{#1}{\Dt,\DIV}}
\newcommand{\GradNorm} [1]{\TNORM{#1}{\Dt,\nabla}}

\begin{document}


\begin{frontmatter}
  
  \title{The virtual element method for the coupled system of
    magneto-hydrodynamics}
  
  \author[SNA]{S.~Naranjo~Alvarez}
  \author[VB] {V.~Bokil}
  \author[VG] {V.~Gyrya}
  \author[GM] {and G.~Manzini}
  
  \address[SNA]{
    Department of Mathematics,
    Oregon State University,
    Corvallis, OR 97331
    USA,
    \emph{e-mail: naranjos@math.oregonstate.edu}
  }
  \address[VB]{
    Department of Mathematics,
    Oregon State University,
    Corvallis, OR 97331
    USA,
    \emph{e-mail: bokilv@math.oregonstate.edu}
  }
  \address[VG]{
    Group T-5,
    Theoretical Division,
    Los Alamos National Laboratory,
    Los Alamos, 87545 NM,
    USA;\\
    \emph{e-mail: vitaliy\_gyrya@lanl.gov}
  }
  \address[GM]{
    Group T-5,
    Theoretical Division,
    Los Alamos National Laboratory,
    Los Alamos, 87545 NM,
    USA;\\
    \emph{e-mail: gmanzini@lanl.gov}
  }
  
  \begin{abstract}
    In this work, we review the framework of the Virtual Element
    Method (VEM) for a model in magneto-hydrodynamics (MHD), that
    incorporates a coupling between electromagnetics and fluid flow,
    and allows us to construct novel discretizations for simulating
    realistic phenomenon in MHD.
    First, we study two chains of spaces approximating the
    electromagnetic and fluid flow components of the model.
    Then, we show that this VEM approximation will yield divergence
    free discrete magnetic fields, an important property in any
    simulation in MHD.
    We present a linearization strategy to solve the VEM approximation
    which respects the divergence free condition on the magnetic
    field.
    This linearization will require that, at each non-linear
    iteration, a linear system be solved.
    We study these linear systems and show that they represent
    well-posed saddle point problems.
    We conclude by presenting numerical experiments exploring the
    performance of the VEM applied to the subsystem describing the
    electromagnetics.
    The first set of experiments provide evidence regarding the speed
    of convergence of the method as well as the divergence-free
    condition on the magnetic field.
    In the second set we present a model for magnetic reconnection in
    a mesh that includes a series of hanging nodes, which we use to
    calibrate the resolution of the method.
    The magnetic reconnection phenomenon happens near the center of
    the domain where the mesh resolution is finer and high resolution
    is achieved.
  \end{abstract}
  
\end{frontmatter}



\section{Introduction}\label{Sec:Intro}

The number of applications involving electrically charged and
magnetized fluids, for example plasmas, has ``skyrocketed'' in the last
decades and great efforts have been devoted to the development of
predictive mathematical models.
One approach that has withstood the test of time and has become
``standard'' in the area of plasma physics is the area called
Magneto-HydroDynamics (MHD), which studies the behavior and the magnetic properties of electrically conducting fluids.  
The system of equations that describe MHD are a coupling between an electromagnetic submodel 
and a fluid flow submodel.
The electromagnetic submodel in MHD is normally based on
Maxwell's equations while the fluid flow submodel relies on conservation
principles such as mass and momentum conservation.
These two submodels are nonlinearly coupled.
Indeed, mass density and momentum distribution in a plasma are
determined by the Lorentz force, which, in turn, is generated by 
the same plasma particles moving in the self-consistent
electromagnetic field.
The details of the MHD model, its derivation and properties are
nowadays well-understood and explained in many textbooks and review
papers, e.g.,
\cite{Davidson:2002,Moreau:2013}.

The topic of this chapter is the review of a novel discretization
method for an MHD model, in the framework of the Virtual
Element Method (VEM), that has been recently proposed ~\cite{paper1}.
In the development of this method we will fix the approximation degree.

VEM was originally proposed as a variational reformulation of the
the \emph{nodal} mimetic finite difference (MFD) 
method~\cite{Brezzi-Buffa-Lipnikov:2009,BeiraodaVeiga-Lipnikov-Manzini:2011}
for solving diffusion problems on unstructured polygonal meshes in a
finite element setting.
A survey of the MFD method can be found in the review
paper~\cite{Lipnikov-Manzini-Shashkov:2014} and the research
monograph~\cite{BeiraodaVeiga-Lipnikov-Manzini:2014}.


Solving partial differential equations (PDEs) on polygonal and polyhedral
meshes has become a central and important issue in the last decades.
In fact, the generality of the admissible meshes makes the VEM highly
versatile and very useful when the mesh must be adapted to the
characteristics of the problem.
For example, we can mention problems where the domain boundary deforms
in time, or there are oddly shaped material interfaces to which the
mesh must be conformal, or the mesh needs to be locally refined in
those parts of the domain requiring greater accuracy as in adaptive
mesh refinement strategies.
In all such situations, the mesh refinement process may result in
highly skewed meshes or meshes with highly irregular structures and a
numerical method must be capable of handling these traits in order to
be robust and provide an accurate approximation to the solution of a
partial differential equation.

The VEM inherits the great mesh flexibility of the MFD in a setting
similar to the finite element method (FEM), so that results and
techniques from FEM can be imported over to VEM.
Moreover, VEM makes possible to formulate numerical
approximations of arbitrary order and arbitrary regularity to PDEs in
two and three dimensions on meshes that other methods often consider as pathological. 
Because of its origins, VEM is intimately connected with other finite
element approaches and the fact that VEM is a FEM implies some
important advantages over other discretization methods such as the finite
volume methods and the finite difference methods.
The connection between VEM and finite elements on
polygonal/polyhedral meshes was thoroughly investigated
in~\cite{Manzini-Russo-Sukumar:2014,Cangiani-Manzini-Russo-Sukumar:2015,DiPietro-Droniou-Manzini:2018},
between VEM and discontinuous skeletal gradient discretizations
in~\cite{DiPietro-Droniou-Manzini:2018}, and between VEM and
BEM-based FEM method
in~\cite{Cangiani-Gyrya-Manzini-Sutton:2017:GBC:chbook}.

The main difference between VEM and FEM is that VEM does
not require an explicit knowledge of basis functions that
generates the finite element approximation space.
The formulation of the method and its practical implementations are
based on suitable polynomial projections that are always computable
from a careful choice of the degrees of freedom.
VEM was first proposed for the Poisson
equation~\cite{BeiraodaVeiga-Brezzi-Cangiani-Manzini-Marini-Russo:2013}
and, then, extended to convection-reaction-diffusion problems with
variable coefficients
in~\cite{BeiraodaVeiga-Brezzi-Marini-Russo:2016b}.
The effectiveness of the virtual element approach is reflected in the
many significant applications that have been developed in less than a
decade see, for example,~\cite{%
BeiraodaVeiga-Manzini:2015, Berrone-Pieraccini-Scialo-Vicini:2015,%
Mora-Rivera-Rodriguez:2015,%
Paulino-Gain:2015,%
Antonietti-BeiraodaVeiga-Scacchi-Verani:2016,%
BeiraodaVeiga-Chernov-Mascotto-Russo:2016,%
BeiraodaVeiga-Brezzi-Marini-Russo:2016b,%
Cangiani-Georgoulis-Pryer-Sutton:2016,%
Perugia-Pietra-Russo:2016,%
Wriggers-Rust-Reddy:2016,%
Manzini-Lipnikov-Moulton-Shashkov:2017,%
Certik-Gardini-Manzini-Vacca:2018:ApplMath:journal,%
Dassi-Mascotto:2018,%
Benvenuti-Chiozzi-Manzini-Sukumar:2019:CMAME:journal,%
Antonietti-Manzini-Verani:2019:CAMWA:journal,%
Certik-Gardini-Manzini-Mascotto-Vacca:2020}.
Numerical dispersion can also be greatly reduced on carefully selected
polygonal meshes, see \cite{Joaquim-Scheer:2016,Ding-Yang:2017}.
In these works, the Finite Difference Time Domain (FDTD) method is
applied to a grid of hexagonal prisms and yields much less numerical
dispersion and anisotropy than on using regular hexahedral grids where
such method is normally considered.

Finally, the divergence of the magnetic field is zero in the Maxwell
equations thus reflecting the absence of magnetic monopoles.
Classical numerical discretizations fail to capture this property when the
discrete versions of the divergence and rotational operators do not
annihilate each other at the level of the zero machine precision, thus
leaving a remainder that can significantly be compounded during a
simulation.
The consequence of the violation of this divergence-free constraint
has thoroughly been investigated in the literature and it was seen
that the numerical simulations are prone to significant
errors~\cite{Brackbill-Barnes:1980,Brackbill:1985,Dai-Woodward:1998,Toth:2000},
as fictitious forces and an unphysical behavior may
appear~\cite{Dai-Woodward:1998}.
Efforts have been devoted to the development of \emph{divergence-free} techniques.
For example, in \cite{Dedner-Kemm-Kroner-Munz-Schnitzer-Wesenberg:2002} 
the divergence equation
$\nabla\cdot\Bv=0$ is taken into account through a Lagrange multiplier
that is additionally introduced in the set of the unknowns;
in~\cite{Kuzmin-Klyushnev:2020}, the divergence-free condition relies on
special flux limiters; in~\cite{Jiang:1998} a special energy
functional is minimized by a least squares finite element method.
Instead, the VEM considered in this chapter provides a numerical
approximation of the magnetic field that is intrinsically divergence
free as a consequence of a de~Rham inequality chain.
The VEM described in the papers of References~\cite{%
  BeiraodaVeiga-Brezzi-Dassi-Marini-Russo:2017,
  BeiraodaVeiga-Brezzi-Dassi-Marini-Russo:2018-CMAME,
  BeiraodaVeiga-Brezzi-Dassi-Marini-Russo:2018-SINUM} 
are also pertinent to this issue.

This chapter is structured as follows.
In Section~\ref{sec:MathForms}, we present the system of equations of
the continuous MHD model and introduce its discrete virtual element
approximation.
In Section~\ref{sec:VirtualElements}, we review the formal definition
and properties of the finite dimensional functional spaces of the
formulation of VEM.
Here, we also discuss the computability of the orthogonal projection
operators and the possibility of using oblique projection operators,
which are orthogonal with respect to a different inner product.
In Section~\ref{sec:EnergyEstimate}, we present a number of energy
estimates that provide evidence of the stability of the method.
In Section~\ref{sec:Linearization}, we review a possible linearization
strategy for solving the nonlinear system that results from
virtual element approximation of the MHD model and prove that the
approximate magnetic field is divergence free.
In Section~\ref{sec:WellPosedness}, we discuss the well-posedness of
the linear solver in the setting of saddle-point problems.
In Section~\ref{sec:NumericalExperiments}, we assess the convergence behavior of the method and show an application of the VEM to the numerical modeling of a magnetic reconnection phenomenon.
Finally, in Section~\ref{sec:Conclusions}, we give the full picture about the proposed method we outline.
\section{Mathematical Formulation}
\label{sec:MathForms}
Let the computational domain $\calD$ be an open, bounded, polygonal
subset of $\REAL^3$.
Further assume that there is a magnetized fluid contained in this domain.
We denote by $\overrightarrow{\us}$, $\overrightarrow{\Bs}$, $\overrightarrow{\Es}$ and $\ps$ 
the velocity, magnetic and electric fields and the pressure of such a fluid.
The evolution of these quantities is governed by the following system of differential equations:
\begin{subequations}
  \begin{align}
    \mbox{Incompressibility of the fluid}:     \quad & \nabla\cdot \overrightarrow{\us} = 0,                                                 \\[0.5em]
    \mbox{Conservation of momentum}: \quad & \frac{\partial}{\partial t}\overrightarrow{\us}-R_e^{-1}\Delta\overrightarrow{\us}-\overrightarrow{\Js}\times\overrightarrow{\Bs}+\nabla\ps =\overrightarrow{\fs},  \\[0.5em]
    \mbox{Faraday's Law}:               \quad & \frac{\partial}{\partial t}\overrightarrow{\Bs}+\nabla\times\overrightarrow{\Es} =\overrightarrow{ 0},                              \\[0.5em]
    \mbox{Ohm's Law}:                   \quad & \overrightarrow{\Js} = \overrightarrow{\Es} + \overrightarrow{\us}\times\overrightarrow{\Bs},                                                      \\[0.5em]
    \mbox{Amp{\`e}re's Law}:                \quad & \overrightarrow{\Js} - R_{m}^{-1}\nabla\times\overrightarrow{\Bs} = \overrightarrow{0},                                            
  \end{align}
\end{subequations}
We denote the two dimensional vector whose components are the $x,y$ components of a three dimensional vector using bold.
Thus we denote
\begin{equation*}
    \overrightarrow{\Bs} = \begin{pmatrix}
      \Bv\\ \Bs_z
    \end{pmatrix},\quad 
    \overrightarrow{\Es} = \begin{pmatrix}
      \Ev\\ \Es_z
    \end{pmatrix},\quad
    \overrightarrow{\us} = \begin{pmatrix}
      \uv\\ \us_z
    \end{pmatrix},\quad
    \overrightarrow{\Js} = \begin{pmatrix}
      \Jv\\ \Js_z
    \end{pmatrix},\quad
    \overrightarrow{\fs} = \begin{pmatrix}
      \fv\\ \fs_z
    \end{pmatrix}.
\end{equation*}
In this chapter we will consider that the $z-$component of the magnetic  field is exactly zero.
Moreover, we will also consider that the $x$ and $y$ components of the magnetic and electric fields and the $z-$component of the velocity field do not vary in the $z-$ direction.
In summary, we are assuming that
\begin{equation*}
    \Bs_z = \frac{\partial}{\partial z}\Bv= \frac{\partial}{\partial z}\Ev =\frac{\partial}{\partial z}\us_z= 0.
\end{equation*}
The consequence is that the dynamics only occurs in two dimensions effectively reducing the dimensionality of the problem.
Our goal will be to attain approximations to the $x$ and $y$ components of the electric and magnetic fields.

Consider an arbitrary, non-empty, cross-section parallel to the $x,y-$plane of $\calD$ denoted by $\Omega$. 
The domain $\Omega$ is embedded in $\REAL^2$ so each point $p\in\Omega$ has a fixed $z-$value.
Applying the aforementioned set of assumptions allows us to predict the dynamics in $\Omega$ as being ruled by:
\begin{subequations}\label{eq:MHD all equations}
  \begin{align}
    \mbox{Incompressibility of the fluid}:     \quad & \DIV\uv = 0,             \label{eq:flowStrongConsMass} \\[0.5em]
    \mbox{Conservation of momentum}: \quad & \frac{\partial}{\partial t}\uv-R_e^{-1}\Delta\uv-\Js\times\Bv+\nabla\ps =\fv, \label{eq:flowConsMomentum}   \\[0.5em]
    \mbox{Faraday's Law}:               \quad & \frac{\partial}{\partial t}\Bv+\ROTv\Es ={\bm 0},                            \label{eq:flowStrongFaraday}  \\[0.5em]
    \mbox{Ohm's Law}:                   \quad & \Js = \Es + \uv\times\Bv,                                                    \label{eq:flowStrongOhmsLaw}  \\[0.5em]
    \mbox{Amp{\`e}re's Law}:                \quad & \Js - R_{m}^{-1}\ROT\Bv = {\bm 0},                                             \label{eq:flowStrongAmperesLaw}
  \end{align}
\end{subequations}
Consider a scalar function $\fs:\Omega\to\REAL$ and 
vector functions 
\begin{align*}
    \gv,\fv:\Omega\to\REAL^2, \qquad \text{where} \quad
    \fv = \begin{pmatrix}f_x\\f_y\end{pmatrix}, \quad 
    \gv = \begin{pmatrix}g_x\\g_y\end{pmatrix},
\end{align*}
with the sub-indices denoting components of the vector values functions $\fv$ and $\gv$, rather than differentiation.
We define two versions of a cross product:
\begin{equation}
  \fs\times\gv = \begin{pmatrix}-\fs g_y\\ \ \ \fs g_x\end{pmatrix},\quad 
  \fv\times\gv = f_xg_y-f_yg_x.
\end{equation}
The two dimensional curl and divergence operators are defined as:
\begin{align*}
  \ROTv\Es =
  \begin{pmatrix}
    \displaystyle
    \frac{\partial}{\partial y}\Es\\[1em]
    \displaystyle
    -\frac{\partial}{\partial x}\Es
  \end{pmatrix},
  \quad
  \ROT\Bv
  = \frac{\partial}{\partial x}\Bs_y
  - \frac{\partial}{\partial y}\Bs_x
  \quad\textrm{and}\quad
  \DIV\uv = \frac{\partial}{\partial x}u_x+\frac{\partial}{\partial y} u_y.
\end{align*}
The initial conditions we prescribe onto the system are:
\begin{equation}
  \uv(0) = \uv_0
  \quad\mbox{and}\quad
  \Bv(0) = \Bv_0.
\end{equation}
The initial field $\Bv_0$ must be divergence free, as
Gauss's Law for the magnetic field requires that $\Bv$ remain divergence free throughout its evolution.
Gauss's law, i.e., the divergence free nature of $\Bv$, is not
explicitly stated in the system of MHD equations.
This is due to the fact that, under the assumption that $\Bv_0$ is divergence free, $\DIV\Bv=0$ is a consequence of Faraday's Law, implying that $\Bv$ is solenoidal for all time. We have
\begin{equation}
  \frac{\partial}{\partial t}\left( \DIV\Bv \right)
  = \DIV\left( \frac{\partial}{\partial t}\Bv \right)
  = \DIV\left( -\ROTv\Es \right) = 0.
\end{equation}
Hence, 
\begin{equation}
    \DIV\Bv=\DIV\Bv_0=0.
\end{equation}
Such condition is a further
evidence of the fact that the
magnetic field is divergence
free and the violation of this
condition will lead to a
nonphysical description of
a MHD phenomenon..
We close the MHD system by adding the boundary conditions
\begin{equation}
  \uv = \uv_b
  \quad\mbox{and}\quad
  \Es = \Es_b
  \quad\mbox{on~}\partial\Omega.
  \label{eq:flowbc}
\end{equation}
On using the divergence theorem and the incompressibility condition, we find
that
\begin{equation}
  \int_{\partial\Omega}\uv\cdot\nv\ d\ell = \int_{\Omega}\DIV\uv\ dA = 0,
  \label{eq:flowcondonuv}
\end{equation}
which implies the consistency condition
\begin{equation}
  \int_{\partial\Omega}\uv_b\cdot\nv\ d\ell=0
\end{equation}
on the boundary velocity field $\uv_b$.
\subsection{Weak formulation}

\medskip
In this section, we present a weak formulation of
problem~\eqref{eq:MHD all equations}.
Such a formulation requires the 
definition of the following inner products and norms.

We use standard notation that, for the sake of completeness, we will describe in what follows. For a pair of sufficiently regular real-valued functions $\fs,\gs:\Omega\to\REAL$ 
or vector valued-functions $\fv,\gv:\Omega\to\REAL^2$ 
we define
\begin{equation}
  \scal{\fs}{\gs} = \int_{\Omega} \fs\gs\ dx,\quad
  \scal{\fv}{\gv} = \int_{\Omega} \fv\cdot\gv\ dx.
\end{equation}
We will denote the $\LTWO-$norms by
\begin{subequations}
  \begin{align}
    & \NORM{\fs}{0,\Omega}:=\left(\int_\Omega |\fs|^2\ dx\right)^{1/2},\\
    &\NORM{\fv}{0,\Omega}:=\left(\int_{\Omega}|\fv|^2\ dx\right)^{1/2}.
  \end{align}
\end{subequations}
The spaces $\LTWO(\Omega)$ and $[\LTWO(\Omega)]^2$ will consist of all those scalar and vector functions, respectively, 
that have finite $\LTWO-$norms. 
Likewise, the $H^1-$norm sand the corresponding spaces $H^1(\Omega)$ and $[H^1(\Omega)]^2$
are defined below. We have
\begin{subequations}
  \begin{align}
    &\NORM{\fs}{1,\Omega}:=\left(\NORM{\fs}{0,\Omega}^2+\NORM{\nabla\fs}{0,\Omega}^2\right)^{1/2},\\
    &\NORM{\fv}{1,\Omega}:=\left(\NORM{\fv}{0,\Omega}^2+\NORM{\nabla\fv}{0,\Omega}^2\right)^{1/2}.
  \end{align}
\end{subequations}
The setting, in space, will require the following functional spaces:
\begin{subequations}\label{eq:def:space spaces}
  \begin{align}
    \HONE(\Omega) &= \left\{\vs\in\LTWO(\Omega):\nabla\vs\in[\LTWO(\Omega)]^2\right\},
    \label{eq:H1:def}
    \\[0.5em]
    \HROTv{\Omega} &= \left\{\Ds\in\LTWO(\Omega):
    \ROTv\Ds \in \left[\LTWO(\Omega)\right]^2\right\},
    \label{eq:Hrotv:def}
    \\[0.5em]
    \HDIV{\Omega} &= \left\{
    \Cv\in\left[\LTWO(\Omega)\right]^2:
    \DIV\Cv\in\LTWO(\Omega)
    \right\},
    \label{eq:Hdiv:def}
    \\[0.5em]
    \LTWOzr(\Omega) &= \left\{
    \qs\in\LTWO(\Omega):\int_\Omega \qs\ dA =0
    \right\},
    \label{eq:L20:def}
    \\[0.5em]
    \HONEzr(\Omega) &= 
    \left\{
    \vs\in\HONE:\restrict{\vs}{\partial\Omega} = 0
    \right\},
    \label{eq:H10:def}
    \\[0.5em]
    \HROTvzr{\Omega} &= \left\{\Ds\in\HROTv{\Omega}:
    \restrict{\Ds}{\partial\Omega}=0
    \right\},
    \label{eq:Hrot0:def}
  \end{align}
\end{subequations}
Each of the function spaces in \eqref{eq:def:space spaces} are endowed with its natural norm.
Let us refer to a generic version of the space from \eqref{eq:def:space spaces}
as $S(\Omega)$ and to its natural norm as $\|\cdot\|_S$.
We say that a function  
$f:[0,T]\to S(\Omega)$
is continuous in time if it is continuous with respect to the natural norm $\|\cdot\|_S$.
The space of all time continuous functions $f:[0,T]\to S(\Omega)$ is denoted as 
$C(0,T; S(\Omega))$.
Thus,
%
%
%
%
\begin{subequations}
  \begin{align}   
    &\CS{1}\left(0,T;\Honev{\Omega}\right)=
    \Big\{\vv:[0,T]\to\Honev{\Omega}:
    \vv\mbox{~and~}\frac{\partial\vv}{\partial t}\mbox{~are~continuous}
    \Big\},
    \label{eq:C1-H1v:def}
    \\[0.5em]
    &\CS{}\left(0,T;\LTWOzr(\Omega)\right)=
    \left\{\qs:[0,T]\to\LTWOzr(\Omega):
    \qs\mbox{ is continuous}
    \right\},
    \label{eq:C0-L20:def}
    \\[0.5em]
    & \CS{}\left(0,\Ts,\HROTv{\Omega}\right)
    :=\left\{ \Es:[0,T]\to\HROTv{\Omega}:\Es\mbox{ is continuous}\right\},
    \label{eq:C0-Hrotv:def}
    \\[0.5em]
    &\nonumber\CS{1}\left(0,T;\HDIV{\Omega}\right) = 
    \Big\{ \Bv:[0,T]\to\HDIV{\Omega}:\\
    &\hspace{6cm}\Bv\mbox{ and }\frac{\partial}{\partial t}\Bv\mbox{ are continuous}\Big\}.
    \label{eq:C0-Hdiv:def}
  \end{align}
\end{subequations}
\medskip
\noindent

The weak form of
problem~\eqref{eq:flowStrongConsMass}-\eqref{eq:flowStrongAmperesLaw}
reads as: \\
\emph{Find} \\
\begin{tabular}{l l}
  \hspace{1cm}
  $\uv \in \CS{1}\left(0,T;\Honev{\Omega}\right)$, 
  \hspace{1cm} & 
  $\Bv \in \CS{1}\left(0,T;\HDIV{\Omega}\right)$,
  \\ \hspace{1cm}
  $\Es \in \CS{ }\left(0,T;\HROTvzr{\Omega}\right)$, 
  &
  $\ps \in \CS{ }\left(0,T;\LTWOzr(\Omega)\right)$,
\end{tabular}\\
\emph{such that}
\begin{subequations}
   \label{eq:flowContVarForm}
  \begin{align}
    &\Scal{\frac{\partial}{\partial t}\uv}{\vv}
    +R_e^{-1}\Scal{\nabla\uv}{\nabla\vv}
    -\Scal{\Js\times\Bv}{\vv}
    -\Scal{p}{\DIV\vv} = \Scal{\fv}{\vv},
    \label{eq:flowContVarFormStokes}
    \\[0.5em]
    &\Scal{\DIV\uv}{q} = 0,
    \label{eq:flowContVarFormDivFreeu}
    \\[0.5em]
    &\Scal{\frac{\partial }{\partial t}\Bv}{\Cv}
    +\Scal{\ROTv \Es}{\Cv} = 0,
    \label{eq:flowContVarFormFaraday}
    \\[0.5em]
    &\Scal{\Js}{\Ds}-R_m^{-1}\Scal{\Bv}{\ROTv \Ds} = 
    0,
    \label{eq:ContVarFormAmpereOhm}
    \\[0.5em]
    &
    \Js = \Es+\uv\times\Bv,\quad
    \uv(\cdot,0)=\uv_0,\quad\Bv(\cdot,0) = \Bv_0\mbox{ with }\DIV\Bv_0=0,
  \end{align}
\end{subequations}
\emph{for any}
$\vv \in \Honevzr{\Omega}$, 
$\Cv \in \HDIV{\Omega}$, 
$\Ds \in \HROTvzr{\Omega}$, and 
$\qs \in \LTWOzr(\Omega)$.
In this formulation we are making implicit that 
\begin{equation}
    \uv  = \uv_b\quad\mbox{and}\quad
    \Es = \Es_b
    \quad\mbox{along}\quad 
    \partial\Omega.
\end{equation}\section{The Virtual Element Method}
\label{sec:VirtualElements}

In the VEM formulation, that will be presented, we will define conforming finite-dimensional subspaces of the spaces 
in \eqref{eq:def:space spaces}.
To do this we introduce a mesh of the domain $\Omega$ with size $h>0$.
Then, we define mesh-dependent finite dimensional virtual element spaces:
\begin{subequations}
  \begin{align}
    &\Ph\subset \LTWO(\Omega),\quad \Vh\subset\HROTv{\Omega},\quad
    \Eh\subset\HDIV{\Omega}\quad   \TVh\subset\Honev{\Omega}.\\[0.5em]
    &\Phzr\subset\LTWOzr(\Omega),\quad
    \Vhzr\subset\HROTvzr{\Omega},\quad
    \TVhzr\subset\Honevzr{\Omega},
  \end{align}
\end{subequations}
with the obvious inclusions
\begin{equation}
  \TVhzr\subset\TVh,\quad
  \Vhzr\subset\Vh,\quad
  \Phzr\subset\Ph.
\end{equation}
These spaces will be formally defined in
Subsections~\ref{subsubsec:NodalSpace},\ref{subsec:EdgeSpace},\ref{subsec:CellSpace}
and \ref{subsec:ElementsFluidFlow}, respectively.
We endow these finite dimensional spaces with the inner products
\begin{subequations}
  \begin{align}
    &\scalTVh{\uvh}{\vvh}\approx\scal{\uvh}{\vvh},\quad
    \scalTVhNa{\uvh}{\vvh} \approx
    \scal{\nabla\uvh}{\nabla\vvh},\\[0.5em]
    &\scalEh{\Bvh}{\Cvh}\approx\scal{\Bvh}{\Cvh},\quad\scalVh{\Esh}{\Dsh}\approx\scal{\Esh}{\Dsh},\\[0.5em]
    &\scalPh{\psh}{\qsh}\approx\scal{\psh}{\qsh},
  \end{align}
\end{subequations}
which approximate the corresponding inner products in $\LTWO(\Omega)$
and $\HONE(\Omega)$ and their vector variants.
In the formulation of the method, we will find it convenient to use a
set of local and global interpolation operators embedding the
continuous spaces into their discrete versions.
We denoted the local operators referring to a specific mesh element
$\P$ as $\ITVhP,\IEhP,\IVhP$ and $\IPh$ and the corresponding global
ones $\ITVh,\IEh,\IVh$ and $\IPh$.

\medskip
Having defined the functional spaces the functional spaces , we now
turn our attention to the time variable.
Te begin we introduce a time-step $\Dt>0$ and the time staggering
parameter $0 \le\theta\le 1$.
The approximate solutions will be considered at time steps given by
\begin{equation*}
  t^n = n\Dt,\quad t^{n+\theta}= (n+\theta)\Dt.
\end{equation*}
These approximations are abbreviated by
$\uvh^{n}=\uvh(\cdot,\ts^{n})$,
$\psh^{n+\theta}=\psh(\cdot,\ts^{n+\theta})$,
$\Bvh^{n}=\Bvh(\cdot,\ts^{n})$, and
$\Esh^{n+\theta}=\Esh(\cdot,\ts^{n+\theta})$ for the time-dependent
vector fields $\uvh$ and $\Bvh$ and the scalar fields $\Esh$ and
$\psh$, respectively, approximating $\uv$, $\Bv$, $\Es$ and $\ps$.
These vector fields are the solution of the virtual element method,
which reads as:\\
\emph{ Find $\big\{(\uvh^n,\Bvh^n)\big\}_{n=0}^N\subset\TVh\times\Eh$}
\emph{and
$\big\{(\Esh^{n+\theta},\psh^{n+\theta}))\big\}_{n=0}^{N-1}\subset
\Vh\times\Phzr$}, \emph{such that}
\begin{subequations}
  \label{eq:flowVEMVarForm}
  \begin{align}
    \ScalTVh{\frac{\uvh^{n+1}-\uvh^n}{\Dt}}{\vvh}+
    R_e^{-1}
    \ScalTVhNa{\uvh^{n+\theta}}{\vvh}
    &+\underbrace{\ScalVh{\Jsh^{n+\theta}}{\IVh(\vvh\times\PiEhRT\Bvh^{n+\theta})}}_{\textbf{(*)}}
    \nonumber\\
    \hspace{3cm}
    -\ScalPh{\psh^{n+\theta}}{\DIV\vvh}&=\ScalTVh{\fvh}{\vvh},
    \label{eq:flowVEMVarFormConsMom} 
    \\
    \ScalPh{\DIV\uvh^{n+\theta}}{\qsh}&=0,
    \label{eq:flowVEMVarFormConsMass}
    \\[0.5em]
    \ScalEh{\frac{\Bvh^{n+1}-\Bvh^{n}}{\Dt}}{\Cvh}+
    \ScalEh{\ROTv\Esh^{n+\theta}}{\Cvh}&=0,
    \label{eq:flowVEMVarFormFaraday}
    \\[0.5em]
    \scalVh{\Jsh^{n+\theta}}{\Dsh}
    -
    R_m^{-1}\scalEh{\Bvh^{n+\theta}}{\ROTv\Dsh}
    &=0,
    \label{eq:flowVEMVarFormOhmAmpere}
  \end{align}
\end{subequations}
\emph{for all
$\vvh\in\TVhzr,\Cvh\in\Eh,\Dsh\in\Vhzr$ and $\qsh\in\Phzr$.} 
We define
\begin{align}
  \Jsh^{n+\theta} &:= \Esh^{n+\theta} +\IVh(\uvh^{n+\theta}\times\PiEhRT\Bvh^{n+\theta}),
  \label{eq:flowDiscCurrentDensity}
\end{align}
and the fractional step quantities $\uvh^{n+\theta}$,
$\Bvh^{n+\theta}$ 
through linear (in time) interpolations
\begin{subequations}
  \begin{align}
    \uvh^{n+\theta} &:= (1-\theta)\uvh^n+\theta\uvh^{n+1}, \\
    \Bvh^{n+\theta} &:= (1-\theta)\Bvh^n+\theta\Bvh^{n+1}. 
  \end{align}
\end{subequations}
The initial conditions are given as
\begin{align}
  \uvh^0=\ITVh(\uv_0),\quad\Bvh^0 &= \IEh(\Bv_0)\quad \mbox{ with }\quad \DIV\Bv_0=0.
\end{align}

\medskip
Here, we implicitly assume that for all $t\in [0,T]$:
\begin{align*}
  \quad \Esh(t) = \IVh(\Es_b(t))\quad\mbox{and}\quad
  \uvh(t)= \ITVh(\uv_b(t))\quad\mbox{along}\quad\partial\Omega.
\end{align*}
Here we use $\Es_b$ and $\uv_b$ as the extensions to the continuous boundary conditions to the interior of $\Omega$.

The term labeled as (*) in \eqref{eq:flowVEMVarFormConsMom} is
produced by the approximation
\begin{align*}
  -\Scal{\Js\times\Bv}{\vv}
  =       \Scal{\Js}{\vv\times\Bv}
  \approx \ScalVh{\Jsh}{\IVh(\vvh\times\Bvh)}.
\end{align*}
The reason why we use this discretization will be made clear in
Section~\ref{sec:EnergyEstimate}, where we present the stability
estimates in the $\LTWO(\Omega)$ norm.
It is important to note that when $\theta=0$ the scheme, in time, is a
forward Euler step.
Whereas if $\theta=1$ the scheme is a Backward Euler step.

\subsection{Mesh notation and regularity assumptions}
\label{subsec:MeshAssumtions}
In this subsection, we present the main notation and the regularity
assumptions that we will make on the mesh.

For ease of exposition, we assume that the computational domain
$\Omega$ be an open, bounded, connected subset of $\REAL^{2}$ with
polygonal boundary $\Gamma$.
We consider the family of domain partitionings
$\mathcal{T}=\{\Th\}_{\hh\in\calH}$.
Every partition $\Th$, the \emph{mesh}, is a finite collection of
polygonal elements $\P$, which are such that
$\overline{\Omega}=\cup_{\P\in\Th}\overline{\P}$.

\medskip
For a polygonal element $\P\in\Th$, we denote the boundary of $\P$ by
$\partial\P$, the outward unit normal to the boundary by $\nor_\P$,
its diameter by $\hP=\max_{\xv,\yv\in\P}\ABS{\xv-\yv}$, and its area by $\mP$.
Each elemental boundary $\partial\P$ is formed by a sequence of
one-dimensional non-intersecting straight edges $\E$ with length $\hE$
and midpoint $\xvE=(\xsE,\ysE)^T$.

%
To enforce mesh regularity
we require that there exists $\rho\ge 0$ independent of the mesh size
$h>0$, such that
\begin{itemize}
\item[]\textbf{(M1)}\,: every polygonal cell $\P\in \Th$ is
  star-shaped with respect to every point of some disk of radius
  $\rho\hP$;
\item[]\textbf{(M2)}\,: every edge $\E\in\partial\P$ of cell
  $\P\in\Th$ satisfies $\hE\geq\rho\hP$.
\end{itemize}

The regularity assumptions \textbf{(M1)}-\textbf{(M2)} allow us to use
meshes with cells having quite general geometric shapes.
For example, non-convex cells or cells with hanging nodes on their
edges are admissible.
Nonetheless, these assumptions have some important implications such
as: $(\mathrm{i})$ every polygonal element is \textit{simply
  connected}; $(\mathrm{ii})$ the number of edges of each polygonal
cell in the mesh family $\{\Th\}_{\hh}$ is uniformly bounded;
$(\mathrm{iii})$ a polygonal element cannot have \emph{arbitrarily
small} edges with respect to its diameter $\hP\leq\hh$ for $\hh\to0$
and inequality $\hP^2\leq\Cs(\rho)\mP\hP^2$ holds, with the obvious
dependence of constant $\Cs(\rho)$ on the mesh regularity factor
$\rho$.

\begin{remark}
It is worth mentioning that virtual element methods on polygonal
meshes possibly containing ``small edges'' have been considered
in~\cite{Brenner-Sung:2018} for the numerical approximation of the
Poisson problem.
The work in~\cite{Brenner-Sung:2018} extends the results
in~\cite{BeiraodaVeiga-Lovadina-Russo:2017}
for the original two-dimensional virtual element method to the version
of the virtual element method in
\cite{Ahmad-Alsaedi-Brezzi-Marini-Russo:2013}
that can also be applied to problems in three dimensions, see
\cite{da2017high}.
Finally, we note that assumptions \textbf{(M1)}-\textbf{(M2)} above
also imply that the classical polynomial approximation theory in
Sobolev spaces holds~\cite{Brenner-Scott:2008}.
While these assumptions are the minimal necessary to develop
theoretical analysis, in practice they can be significantly weakened.
\end{remark}
\subsection{The Nodal Space}
\label{subsubsec:NodalSpace}
Consider the cell $\P$ of the polygonal mesh $\Th$.
The formal definition of the nodal elemental space is
\begin{align}
  \Vh(\P):=\Big\{~
  \Dsh\in\HROTv{\P}:
  &
  \restrict{\Dsh}{\partial\P}\in\CS{}(\partial\P),\,
  \Dsh\in\PS{1}(\E)\,\,
  \forall\E\in\partial\P,\,
  \\
  &
  \ROT\ROTv\Dsh=0\textrm{~~in}~\P 
  ~\Big\}.
  \label{eq:Vh:def}
\end{align}
Every function $\Dsh\in\Vh(\P)$ is uniquely determined by the set of
degrees of freedom:
\begin{itemize}
\item[]\TERM{V}{} the vertex values $\Dsh(\V)$ at the nodes $\V$ of
  cell $\P$.
\end{itemize}
These are represented by blue disks centered at the nodes, see the
sample picture in Figure~\ref{fig:VertexDOF}.

\begin{figure}[t]
  \centering
  \includegraphics[scale=0.5]{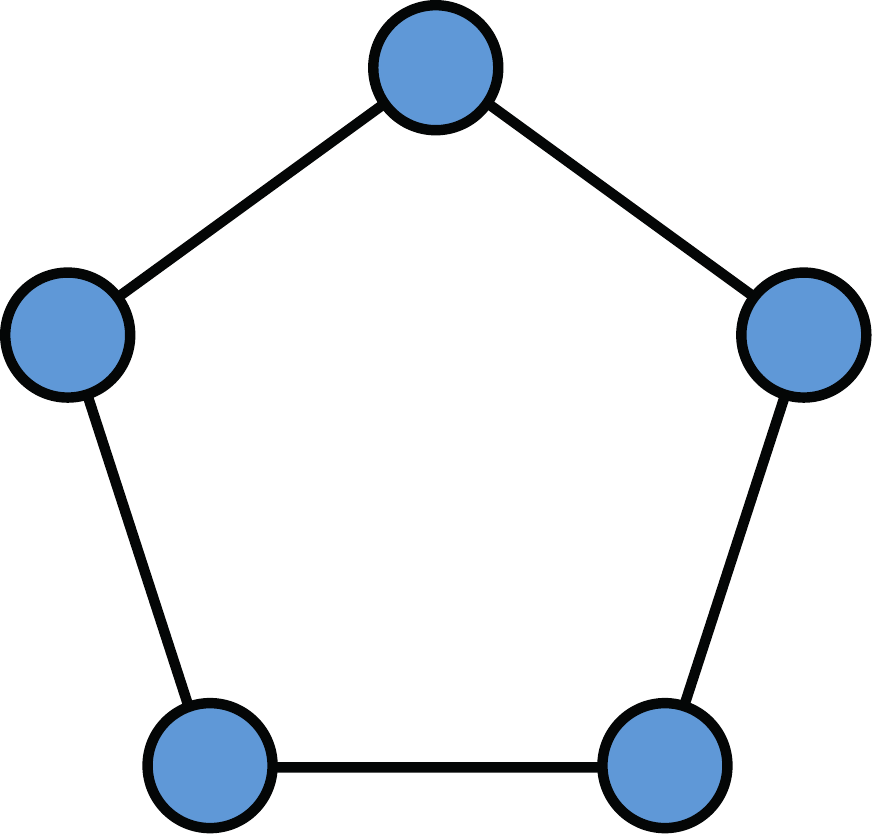}
  \caption{Representation of the degrees of freedom of functions in $\Vh(\P)$.}
  \label{fig:VertexDOF}
\end{figure}
As in the classic finite element method, this property is referred to
as the \emph{unisolvency} of the degrees of freedom \TERM{V}{}, see
\cite{monk2003finite}.
Accordingly, every function in $\Vh(\P)$ corresponds to one and only
one set of degrees of freedom and, conversely, every set of degrees of
freedom corresponds to one and only one function in $\Vh(\P)$.
To formally state this property we define the operators
$\calR_\V:\CS{\infty}(\Omega)\to\REAL^N$ such that for any
$\Ds\in\CS{\infty}(\Omega)$ the image $\calR_\V(\Ds)$ is the array of
degrees of freedom of $\Ds$.
This function can be continuously extended to the full space
$\HROTv{\P}$.
The next theorem states the unisolvency of the finite element previously described.

\medskip
\begin{theorem}\label{Thm:VhUnisolvency}
Define $\calK_\V:\Vh(\P)\to\REAL^N$ as the restriction of $\calR_\V$ to $\Vh(\P)$.
Then, $\calK_\V$ is bijective.
\end{theorem}

\BEGINPROOF
Proof of the above theorem is provided in \cite{BeiraodaVeiga-Brezzi-Marini-Russo:2016a}.
\ENDPROOF

The result of Theorem~\ref{Thm:VhUnisolvency}
Allows us to define the mapping  $\IVh_\P:\HROTv{\P}\to\Vh(\P)$ given by $\IVh_\P = \calK_\V^{-1}\circ\calR_\V$.

We endow the space $\Vh(\P)$ with an $\LTWO$-like inner product.
The usual strategy for the construction of an inner product in a
finite dimensional space relies on the $\LTWO$-orthogonal projection
onto linear polynomials on $\P$, denoted by
$\PizP{}:\Vh(\P)\to\PS{1}(\P)$.
Unfortunately, this projector is not computable in $\Vh(\P)$.
To have a computable orthogonal projection operator we could change
the definition of the space as in the so called 
``enhancement approach''~\cite{BeiraodaVeiga-Brezzi-Marini-Russo:2014,Ahmad-Alsaedi-Brezzi-Marini-Russo:2013,BeiraodaVeiga-Brezzi-Marini-Russo:2016a}.
This strategy will, effectively, change the definition of the space $\Vh(\P)$ calling into question whether or not an important De-Rham complex hold, see Subsection~\ref{subsec:ElectroMagDeRham}.
The reality is that such a diagram holds even in the enhanced scenario, see~\cite{da2018lowest}.
However, we were not aware of this enhanced diagram, instead, we follow a different strategy through a special polynomial
reconstruction operator $\Pi_\P:\Vh(\P)\to\PS{1}(\P)$ satisfying the following
three properties:

\begin{itemize}
\item[]\TERM{P1}{} 
  $\Pi_\P\Dsh$ is computable only from the degrees of freedom of $\Dsh\in \Vh(\P)$;

  \smallskip
\item[]\TERM{P2}{} 
  $\Pi_\P$ preserves all linear polynomials, 
  i.e., for any $\Dsh\in\PS{1}(\P)$, $\Pi_\P\Dsh=\Dsh$;

  \smallskip
\item[]\TERM{P3}{} 
  $\Pi_\P$ is a bounded operator with respect to $L^2$ norm
  with the upper bound constant $\Cs_{\Pi}$ independent of the mesh resolution $\hh$, 
  i.e., for any $\Dsh\in \Vh(\P)$
  \begin{equation}
    \NORM{\Pi_\P\Dsh}{0,\P} \le \Cs_{\Pi}\NORM{\Dsh}{0,\P}.
  \end{equation}
\end{itemize}
We can use this projector to define an inner product in the space $\Vh$ that will allow us to approximate $\LTWO-$inner product as they appear in \eqref{eq:flowVEMVarForm}. In Subsection~\ref{subsubsec:ObliqueProj},  we discuss three possible
implementations of the polynomial reconstruction operator.

We define
\begin{equation}
  \scalVhP{\Esh}{\Dsh} = 
  \scal{\Pi_\P\Esh}{\Pi_\P\Dsh} + 
  \calS^\V\big( (1-\Pi_\P)\Esh, (1-\Pi_\P)\Dsh \big),
  \label{eq:Vh(P)InProdDef}
\end{equation}
where $\calS^\V$ is the stabilization bilinear form.
According to the standard VEM construction, $\calS^\V$ can be
\emph{any} bilinear form for which there exist two real constants
$s_*$ and $s^*$ independent of $\hh$ such that 
\begin{align}
  s_*\NORM{\Dsh}{0,\P}^2
  \leq\calS^\V(\Dsh,\Dsh)\leq
  s^*\NORM{\Dsh}{0,\P}^2
  \quad\forall\Dsh\in\ker\Pi\cap\Vh(\P).
  \label{eq:Vhsstardefs}
\end{align}
In practice, we can design the stabilization as
in~\cite{Mascotto:2018,Dassi-Mascotto:2018}.
This inner product defines the norm in $\Vh(\P)$ given by
$\TNORM{\Dsh}{\Vh(\P)}^2=\scalVhP{\Dsh}{\Dsh}$.
When the projection operator $\Pi$ 
satisfies properties \TERM{P1}{}-\TERM{P3}{}, 
the inner product~\eqref{eq:Vh(P)InProdDef} satisfies two fundamental
properties summarized in the following theorem.

\begin{theorem}
  \label{Thm:EquivalentInProdsVh(P)}
  The inner product $\scalVhP{\,\cdot\,}{\,\cdot\,}$ defined
  in~\eqref{eq:Vh(P)InProdDef} satisfies
  \begin{itemize}
  \item\textbf{Linear consistency}:
    \begin{align}
      \scalVhP{\ps}{\qs} = \scal{\ps}{\qs}
      \quad\forall\ps,\qs\in\PS{1}(\P)\subset\Vh(\P).
      \label{eq:Vh(P)accuracy}
    \end{align}
  \item\textbf{Stability}: there exists two real constants $\alpha_*$
    and $\alpha^*>0$ independent of $\hh$ and $\P$ such that
    \begin{align}
      \alpha_*\NORM{\Dsh}{0,\P}^2
      \leq\TNORM{\Dsh}{\Vh(\P)}^2\leq
      \alpha^*\NORM{\Dsh}{0,\P}^2.
      \quad\forall \Dsh\in\Vh(\P)^2.
      \label{eq:Vh(P)stability}
    \end{align}
  \end{itemize}
\end{theorem}
\BEGINPROOF
  To prove the \emph{linear consistency}~\eqref{eq:Vh(P)accuracy},
  take two polynomial functions $\ps,\qs\in\PS{1}(\P)$.
  Property \TERM{P2}{} implies that $\Pi\ps=\ps$ and $\Pi\qs=\qs$.
  So, the stabilization term in~\eqref{eq:Vh(P)InProdDef} is zero, and
  we find that
  \begin{align*}
    \scalVhP{\ps}{\qs} = \scalP{\Pi_\P\ps}{\Pi_\P\qs} = \scalP{\ps}{\qs}.
  \end{align*}
  
  To prove the lower bound of the \emph{stability}
  condition~\eqref{eq:Vh(P)stability}, we add and subtract
  $\Pi_\P\Dsh$, apply the triangular inequality, the left-most
  inequality in~\eqref{eq:Vhsstardefs} and note that 
  \begin{align*}
    \NORM{\Dsh}{0,\P}^2
    &\leq \big(\NORM{\Pi_\P\Dsh}{0,\P}+\NORM{(1-\Pi_\P)\Dsh}{0,\P}\big)^2
    \leq 2\big(\NORM{\Pi_\P\vvh}{0,\P}^2 + \NORM{(1-\Pi_\P)\Dsh}{0,\P}^2 \big)
    \\[0.5em]
    &
    \leq 2\max(1,s^*)\Big( \scalP{\Pi_\P\Dsh}{\Pi_\P\Dsh} + \calS^{\V}\big( \big(1-\Pi_\P\big)\Dsh, \big(1-\Pi_\P\big)\Dsh \big) \Big)
    \\[0.5em]
    &
    = (\alpha_*)^{-1}\scalVhP{\Dsh}{\Dsh}
    = (\alpha_*)^{-1}\TNORM{\Dsh}{\Vh(P)}^2,
  \end{align*}
  where the lower bound $\alpha_*$ in \eqref{eq:Vh(P)stability} is given by $\alpha_*^{-1} = 2\max(1,s^*)$.
  
  To obtain the upper bound in \eqref{eq:Vh(P)stability}, we first
  note that property \TERM{P3}{} implies that:
  \begin{align*}
    &\calS^\V\big( (1-\Pi_\P)\Dsh, (1-\Pi_\P)\Dsh \big)
    \leq s^*\NORM{(1-\Pi_\P)\Dsh}{0,\P}^2
    \leq s^*\big( \NORM{\Dsh}{0,\P} + \NORM{\Pi_\P\Dsh}{0,\P} \big)^2
    \\[0.5em]
    &\qquad
    \leq s^*\big( \NORM{\Dsh}{0,\P} + \Cs_{\Pi}\NORM{\Dsh}{0,\P} \big)^2
    \leq s^*(1+\Cs_{\Pi})^2\NORM{\Dsh}{0,\P}^2.
  \end{align*}
To conclude this theorem we use
  \begin{align*}
    \TNORM{\Dsh}{\Vh(\P)}^2
    &=    \scalP{\Pi_\P\Dsh}{\Pi_\P\Dsh} + \calS^\V\big( (1-\Pi_\P)\Dsh, (1-\Pi_\P)\Dsh \big)
    \\[0.5em]
    &
    \leq \Cs_{\Pi}^{2}\NORM{\Dsh}{0,\P}^2 + s^*(1+\Cs_{\Pi})^2\NORM{\Dsh}{0,\P}^2
    \leq \alpha^*\NORM{\Dsh}{0,\P}^2,
  \end{align*}
  where the upper bound $\alpha^*$ in \eqref{eq:Vh(P)stability} is given by $\alpha^*=\max(\Cs_{\Pi}^2,s^*(1+\Cs_{\Pi})^2)$.
  
  Note that both lower and upper bounds $\alpha_*$ and $\alpha^*$ 
  depend only on the upper bound $\Cs_{\Pi}$ of the projection operator $\Pi$ and are independent of mesh resolution $\hh$.
 \ENDPROOF

The global space $\Vh$ is the subset of functions in $\HROTv{\Omega}$
whose restriction to any element $\P\in\Th$ belongs to $\Vh(\P)$.
Formally, we write that 
\begin{equation}
  \Vh =
  \left\{ 
  \Dsh\in\HROTv{\Omega}:
  \forall\P\in\Omega_h\quad
  \restrict{\Dsh}{\P} \in\Vh(\P)
  \right\}.
\end{equation}
We endow the global space $\Vh$ with
the global inner product
\begin{equation}
  \scalVh{\Esh}{\Dsh}
  = \sum_{\P\in\Omega_h}\scalVhP{\restrict{\Esh}{\P}}{\restrict{\Dsh}{\P}}
  \quad\forall\Esh,\Dsh\in\Vh,
  \label{eq:VhGlobalInProd}
\end{equation}
and the global norm $\NORM{\Dsh}{\Vh}^2=\scal{\Dsh}{\Dsh}$
The global inner product inherits the properties of accuracy and
stability from the elemental inner product that are stated in
Theorem~\ref{Thm:EquivalentInProdsVh(P)}.
\begin{corollary}
  \label{Coro:EquivalentInProdsVh(P)}
  The inner product $\scalVh{\,\cdot\,}{\,\cdot\,}$ defined
  in~\eqref{eq:VhGlobalInProd} has the two properties:
  \begin{itemize}
  \item\textbf{Linear consistency}:
    \begin{align}
      \scalVh{\ps}{\qs} = \scal{\ps}{\qs}
      \quad\forall\ps,\qs\in\PS{1}(\Th)\subset\Vh(\P).
      \label{eq:Vh(P)accuracy global}
    \end{align}
  \item\textbf{Stability}: there exists two real constants $\alpha_*$
    and $\alpha^*>0$ independent of $\hh$ and $\P$ such that
    \begin{align}
      \alpha_*\NORM{\Dsh}{0}^2
      \leq\TNORM{\Dsh}{\Vh(\P)}^2\leq
      \alpha^*\NORM{\Dsh}{0}^2
      \quad\forall \Dsh\in\Vh(\P)^2,
      \label{eq:Vh(P)stability global}
    \end{align}
  \end{itemize}
where the lower and upper bound constants $\alpha_*$ and $\alpha^*>0$ 
are the same constants introduced in~Theorem~\ref{Thm:EquivalentInProdsVh(P)}
and $\PS{1}(\Th)$ is the space of piecewise linear polynomials built on the mesh $\Th$.
\end{corollary}
\BEGINPROOF
Proof of this Corollary follows immediately from the results of Theorem~\\ref{Thm:EquivalentInProdsVh(P)}.
\ENDPROOF

Finally, we introduce the \emph{global interpolation operator}
$\IVh:\HROTv{\Omega}\to\Vh$, whose restriction to any cell coincides
with the elemental interpolation operator:
\begin{equation}
  \restrict{\IVh(\Dsh)}{\P} =
  \IVhP(\restrict{\Dsh}{\P})
  \quad
  \forall\Dsh\in\Vh\,\P\in\Th.
\end{equation}

\subsubsection{The polynomial reconstruction operators}
\label{subsubsec:ObliqueProj}

We discuss here three alternative choices for the oblique projections that can be used to approximate inner products in the space $\Vh(\P)$.

\medskip
\noindent
\textbf{(I)~Elliptic Projection operator.}
We denote this projection operator as 
\begin{equation}
\PirP{}:\Vh(\P)\to\PS{1}(\Omega).
\end{equation}
For $\Dsh\in\Vh(\P)$ the elliptic projection operator is the solution
to the variational problem
\begin{subequations}
  \begin{align}
    \int_{\P}\ROTv\left(\Dsh-\PirP\Dsh\right)\cdot\ROTv\qs\ dA &= 0
    \quad\forall\qs\in\PS{1}(\P),
    \label{eq:VarFormPirP:A}
    \\[0.5em]
    \Ps_0\big(\Dsh-\PirP\Dsh\big) &= 0,    
    \label{eq:VarFormPirP:B}
  \end{align}
\end{subequations}
where we use the additional projector
$\Ps_0(\Dsh)=\sum_{\V\in\partial\P}\Dsh(\V)$ on $\PS{0}(\P)$ to remove
the kernel of operator $\ROTv$.
The linear polynomial $\PirP{}\Dsh$ is computable because the integral
quantities
\begin{align}
  \quad\int_\P \ROTv\Dsh\cdot\ROTv\qs\ dA
  \quad\forall\qs\in\PS{1}(\P),
\end{align}
are computable from the degrees of freedom \TERM{V} of $\Dsh$.
To prove this statement we use the Green's theorem, note that
$\ROT\ROTv\qs=0$, since $\qs\in\PS{1}(\P)$, and split the integral on
$\partial\P$ in the summation of edge integrals to obtain
\begin{align*}
  \int_{\P}\ROTv\Dsh\cdot\ROTv\qs\ dA
  &= \int_{\P}\Dsh\ROT\ROTv\qs\ dA
  + \int_{\partial\P}\Dsh\tv\cdot\ROTv\qs\ d\ell
  \\
  &= \sum_{\E\in\partial\P}\int_{\E}\Dsh\tv\cdot\ROTv\qs\ d\ell,
\end{align*}
where $\tv$ is the unit tangent vector parallel to $\E$.
The edge integrals are computable because $\tv\cdot\ROTv\qs$ is a
known function in $\PS{0}(\E)$ and we can interpolate the trace
$\restrict{\Dsh}{\E}\in\PS{1}(\E)$ using the evaluation of $\Dsh$ at
the vertices of edge $\E$, which are known from the degrees of freedom
\TERM{V}{}.

\medskip
\textbf{(II)~Least Squares polynomial reconstruction operator.}
The second reconstruction operator that we consider is denoted as
\begin{equation}
    \PiVhPLS:\Vh(\P)\to\PS{1}(\Omega).
\end{equation}
For a function $\Dsh\in\Vh(\P)$, the linear polynomial $\PiVhPLS\Dsh$
is the solution of the Least Squares problem
\begin{align*}
  \PiVhPLS\Dsh(\xv) := \textrm{argmin}_{\qs\in\PS{1}(\P)} \sum_{\V\in\partial\P}\ABS{\Dsh(\xvV)-\qs(\xvV)}^2.
\end{align*}
The solution to this problem has a closed form that can be easily
written as follows.
Let $\big\{\ms_1,\ms_2,\ms_3\big\}$ be the scaled monomial basis of
$\PS{1}(\P)$, which is given by:
\begin{align*}
  \ms_1(\xs,\ys) = 1,\quad
  \ms_2(\xs,\ys) = \frac{\xs-\xsP}{\hP},
  \quad\mbox{and}\quad
  \ms_3(\xs,\ys) = \frac{\ys-\ysP}{\hP},
\end{align*}
where $\xvP=(\xsP,\ysP)^T$ is the position vector of the barycenter of
$\P$.
Let 
\begin{align*}
  \PiVhPLS\Dsh(\xs,\ys) = \as\ms_1(\xs,\ys) + \bs\ms_2(\xs,\ys) + \cs\ms_3(\xs,\ys).
\end{align*}
We denote the position vector of the $i$-th vertex $\V_i$ by $\xv_i$,
for $i=1,\ldots,N$, where $N$ is the number of vertices of $\P$.
Then, the coefficient vector $\xiv=(\as,\bs,\cs)^T$ is the solution of
the system
\begin{equation}
  \matA\xiv=\bv
  \textrm{~~with~~}
  \matA
  = \begin{pmatrix}
    \ms_1(\xv_1) & \ms_2(\xv_1) & \ms_3(\xv_1) \\
    \ms_1(\xv_2) & \ms_2(\xv_2) & \ms_3(\xv_2) \\
    \vdots      & \vdots       & \vdots      \\
    \ms_1(\xv_N) & \ms_2(\xv_N) & \ms_3(\xv_N) \\
  \end{pmatrix},
  \quad
  \bv =
  \begin{pmatrix}
    \Dsh(\xv_1)\\ \Dsh(\xv_2)\\ \vdots \\ \Dsh(\xv_N)
  \end{pmatrix}.
\end{equation}
Since $\matA$ is a maximum rank matrix, the array of the solution
coefficients is given by $\xiv=(\matA^T\matA)^{-1}\matA^T\bv$.

\medskip
\noindent
\textbf{(III)~Galerkin Interpolation operator.}
The final projector that we consider in this chapter is denoted as
$\PiVhPpw$ and is the piecewise linear Galerkin interpolation on a
triangular partition of $\P$.
If $\P$ is a convex polygon, we can easily build such triangular
partition by connecting its vertices and the barycenter given by the
convex linear combination
\begin{align*}
  \xvV^* = \sum_{\V\in\partial\P} \alpha_{\V}\xvV,
\end{align*}
for some suitable choices of the coefficients $\alpha_{\V}$ that are such
that $0\leq\alpha_{\V}\leq1$ for every $\V$ and
$\sum_{\V\in\partial\P}\alpha_{\V}=1$.
If $\P$ is only star-shaped but not necessarily convex, we can still
define an inner point $\V^*$ by a different choice of the coefficients
$\alpha_{\V}$.
For a given function $\Dsh\in\Vh(\P)$, we assume that
\begin{align*}
  \PiVhPpw\Dsh(\xvV^*) = \sum_{\V\in\partial\P}\alpha_{\V}\Dsh(\V)
  \quad\textrm{and}\quad
  \PiVhPpw\Dsh(\xvV) = \Dsh(\xvV)\quad\forall\V\in\partial\P.
\end{align*}
Then, in every triangle $\T$ with vertices $\Vp$, $\Vpp$ and $\V^*$,
we define $\PiVhPpw\Dsh(\xv)$ as the linear interpolant of the values
$\Dsh(\Vp)$, $\Dsh(\Vpp)$, and $\Dsh(\V^*)$.

\subsection{The Edge Space}
\label{subsec:EdgeSpace}

The next virtual element space that we consider is the finite
dimensional counterpart of $\HDIV{\Omega}$.
This space was introduced
in~\cite{BeiraodaVeiga-Brezzi-Marini-Russo:2016a}.
Like before, we begin by defining a local space over a cell $\P$.
The formal definition reads as
\begin{align}
  \Eh(\P) :=
  \Big\{
  \Cvh\in\HDIV{\P}\cap\HROT{\P}:\,
  & \DIV\Cvh\in\PS{0}(\P),\;\;\ROT\Cvh = 0,
  \nonumber\\
  & \restrict{\Cvh}{\E}\cdot\nv\in\PS{0}(\E)
  \,\,\forall\E\in\partial\P
  \Big\}.
  \label{eq:Eh:def}
\end{align}
Every virtual element function $\Cvh\in\Eh(\P)$ is characterized by
the following set of degrees of freedom
\begin{itemize}
\item[]\TERM{E} the average of the normal flux on each edge:
  \begin{align*}
    \forall\E\in\partial\P:\quad
    \frac{1}{|\E|}\int_\E\Cvh\cdot\nv\ d\ell.
  \end{align*}
\end{itemize}
These are represented by red arrows pointing out the edges of the cell
$\P$, see the sample picture in Figure~\ref{fig:EdgesDOF}.

\begin{figure}[t]
  \centering
  \includegraphics[scale=0.5]{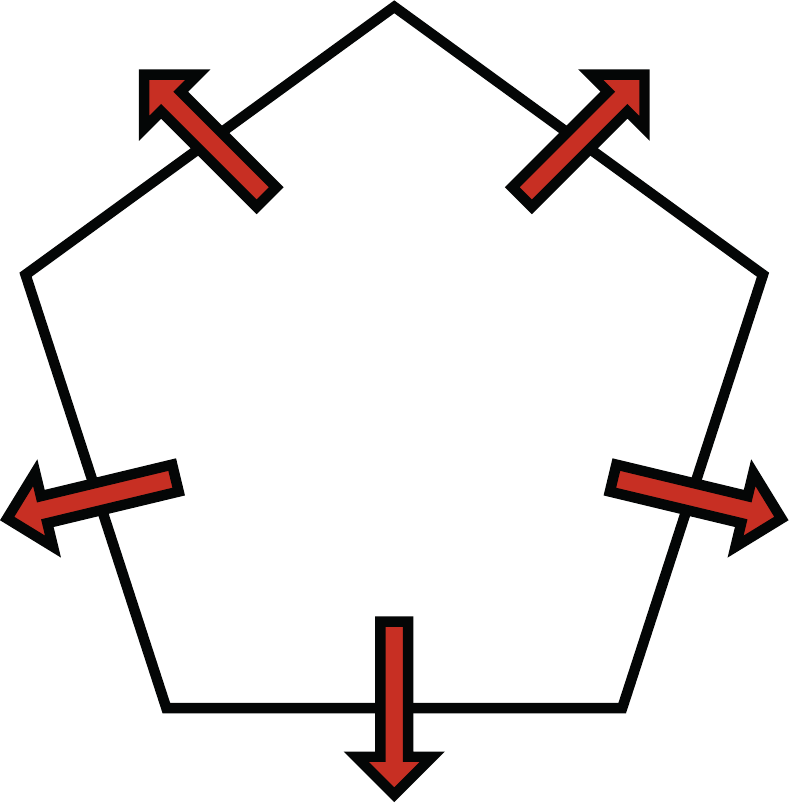}
  \caption{Representation of the degrees of freedom of functions in $\Eh(\P)$.}
  \label{fig:EdgesDOF}
\end{figure}
In order to properly state the property of unisolvency we introduce
$\calR_\E:\HDIV{\P}\to\REAL^N$ such that for any $\Cv\in\HDIV{\P}$ the
array $\calR_\E(\Cv)$ is the array of degrees of freedom of $\Cv$.
This result is stated below
\begin{theorem}\label{Thm:EhUnisolvency}
Let $\calK_\E:\Eh(\P)\to\REAL^N$ be the restriction of $\calR_\E$.
Then, $\calK_\E$ is bijective.
\end{theorem}
In view of the unisolvency of $\Eh(\P)$, we define the interpolation
operator $\IEhP = \calK_\E^{-1}\circ\calR_\E$.

Next, we define an important projector in the space $\Eh(\P)$, namely
the orthogonal projections $\PizP{}:\Eh(\P)\to\PS{0}(\P)$ whose image
are the solution to the variational problem
\begin{equation}
  \scalP{\Cvh-\PizP{}\Cvh}{\qv} = 0
  \qquad\text{for all}\quad \qv\in\big[\PS{0}(\P)\big]^2,
  \label{eq:EhProjDefs:A}
\end{equation}
for every $\Cvh\in\Eh(\P)$, and where $\RT{0}(\P)$ is the space of
vector-valued functions over $\P$ defined as
This projector are computable using the degrees of freedom \TERM{E}.
Pick $\qv\in[\PS{0}(\P)]^2$, $\ps\in\PS{1}(\P)$ where the scalar polynomial $\ps$ is
chosen so that $\qv=\nabla\ps$.
Then, we apply the Green theorem and we find that
\begin{equation}
  \int_{\P}\Cvh\cdot\qv\ dA =
  \int_{\P}\Cvh\cdot\nabla\ps\ dA =
  -\int_{\P} (\DIV\Cvh)\ps\ dA
  +\int_{\partial\P} \ps\Cvh\cdot\nv\ d\ell
  \label{eq:EhComputeProjs}
\end{equation}
for all $\Cvh\in\Eh(\P)$.
We split the integral on $\partial\P$ in the summation of line
integrals
\begin{equation}\label{eq:SplitLineInt}
  \int_{\partial\P} \Cvh\cdot\nv\ps\ d\ell = 
  \sum_{\E\in\partial\P}\restrict{(\Cvh\cdot\nv)}{\E}\int_{\E}\ps\ d\ell,
\end{equation}
and we note that $\restrict{\Cvh\cdot\nv}{\E}$ is constant on each
edge $\E\in\partial\P$, cf. space definition~\eqref{eq:Eh:def}, and
coincides with the evaluation in \TERM{E}{}.
In turn, we compute $\DIV\Cvh\in\PS{0}(\P)$ by applying the divergence
theorem:
\begin{align}
  \DIV\Cvh =\frac{1}{|\P|}\int_{\partial\P}\Cvh\cdot\nv\ d\ell
  = \frac{1}{\mP}\sum_{\E\in\partial\P}\mE\left(\frac{1}{\mE}\int_{\E} \Cvh\cdot\nv\ d\ell\right).
\end{align}
Note that the polynomial $\ps$ is determined by the relation
$\nabla\ps=\qv$ and is defined up to an additive constant factor.
If we choose this constant factor equal to the elemental average of
$\ps$ on $\P$, so that $\int_\P\ps\ dA=0$, we can make the area
integral vanish since
\begin{equation}
  \int_\P(\DIV\Cvh)\ps\ dA
  = \restrict{(\DIV\Cvh)}{\P}\int_{\P}\ps\ dA = 0.
\end{equation}
In conclusion, the information about a virtual element function $\Cvh$
in the space $\Eh(\P)$ that we need to compute the projection $\PizP{}\Cvh$ can be read off the degrees of
freedom of $\Cvh$.

We can use $\PizP{}$ to define the inner product:
\begin{align}
  \scalEhP{\Bvh}{\Cvh}
  = \scal{\PizP{}\Bvh}{\PizP{}\Cvh}
  + \calS^\E((\calI-\PizP{})\Bvh,(\calI-\PizP{})\Cvh)
  \label{eq:EhInProdDef}
\end{align}
for every possible pair of virtual element functions
$\Bvh,\Cvh\in\Eh(\P)$.
As before, the stabilization form $\calS^\E$ can be \emph{any}
continuous bilinear form for which there exists two strictly positive
constants $s_*$ and $s^*$ independent of $\hh$ such that
\begin{align*}
  s_*\NORM{\Cvh}{0,\P}^2
  \leq\calS^\E(\Cvh,\Cvh)\le
  s^*\NORM{\Cvh}{0,\P}^2
  \quad\forall\Cvh\in\Eh(\P)\cap\ker\PizP{}\cap\Eh(\P).
\end{align*}
Practical implementations of $\calS^\E$can be designed according
with~\cite{Mascotto:2018,Dassi-Mascotto:2018} for more examples.
The constants $s_*$ and $s^*$ are different from those in equation
\eqref{eq:Vhsstardefs}.
%
This inner product defines the norm 
\begin{align}
  \TNORM{\Cvh}{\Eh(\P)} = \scalEhP{\Cvh}{\Cvh}^{1/2}
  \quad\forall\Cvh\in\Eh(\P),
  \label{eq:EhPNorm}
\end{align}
and the two fundamental properties of $\PS{0}$-consistency an
stability hold as stated in the following theorem.

\medskip
\begin{theorem}
  \label{Thm:EhPInProdProperties}
  The inner product $\scalEhP{\,\cdot\,}{\,\cdot\,}$ defined
  in~\eqref{eq:EhInProdDef} has the two properties:
  \begin{itemize}
  \item\textbf{$\PS{0}$-consistency}:
    \begin{align}
      \scalEhP{\Cvh}{\qv} = \scal{\Cvh}{\qv}
      \quad\forall\Cvh\in\Eh(\P),\,\qv\in[\PS{0}(\P)]^2
      \label{eq:Eh(P):accuracy}
    \end{align}
  \item\textbf{Stability}: there exists two real constants $\beta_*$
    and $\beta^*>0$ independent of $\hh$ and $\P$ such that
    \begin{align}
      \beta_*\NORM{\Cvh}{0,\P}^2
      \leq\TNORM{\Cvh}{\Eh(\P)}^2\leq
      \beta^*\NORM{\Cvh}{0,\P}^2
      \quad\forall\Cvh\in\Eh(\P).
      \label{eq:Eh(P):stability}
    \end{align}
  \end{itemize}
\end{theorem}
\begin{proof}
  We omit the proof of this Theorem since it is
  essentially the same as the one presented for
  Theorem~\ref{Thm:EquivalentInProdsVh(P)}.
  We just note that here the orthogonality of the projector $\PizP{}P{}$
  makes a more general result possible as the consistency condition is
  verified if at least one and not necessarily both of the entries of
  $\scalEhP{\Cvh}{\qv}$ is a (vector-valued) polynomial field.
  In fact, if $\qv\in\big[\PS{2}(\P)\big]^2$ we have that
  $\calS^\E((\calI-\PizP{})\Cvh,(\calI-\PizP{})\qv)=0$ because
  $\PizP{}\qv=\qv$.
  Then, the definition of the orthogonal projection $\PizP{}$, which
  is also polynomial-preserving, implies that
  \begin{align}
    \scalEhP{\Bvh}{\Cvh}
    = \scalP{\PizP{}\Cvh}{\PizP{}\qv}
    = \scalP{\Cvh}{\PizP{}\qv}
    = \scalP{\Cvh}{\qv}
  \end{align}
  for all $\Cvh\in\Eh(\P)$ and $\qv\in\PS{0}(\P)$.
\end{proof}

\medskip
We introduce the global virtual element space $\Eh$ built on
the mesh $\Th$ by pasting together the elemental spaces $\Eh(\P)$
built on all cells $\P$:
\begin{align*}
  \Eh = \big\{
  \Cvh\in\HDIV{\Omega}:\,
  \restrict{\Cvh}{\P}\in\Eh(\P)
  \,\,\forall\P\in\Th
  \big\}.
\end{align*}
We endow this space with the inner product
\begin{align}
  \scalEh{\Bvh}{\Cvh}
  = \sum_{\P\in\Th}\scalEhP{\restrict{\Bvh}{\P}}{\restrict{\Cvh}{\P}}
  \quad\forall\Bvh,\Cvh\in\Eh,
  \label{eq:EhInProdGlobal}
\end{align}
and the induced norm
\begin{align}
  \TNORM{\Cvh}{\Eh}^2 = \scalEh{\Cvh}{\Cvh}
  \quad\forall\Cvh\in\Eh.
  \label{eq:EhNormGlobal}
\end{align}
As for the nodal space, this global inner product and associated norm
satisfy the fundamental properties of \emph{$\PS{0}$-consistency} and
\emph{stability}, which we state in the next corollary.
These properties imply the exactness of the inner product defined in
\eqref{eq:EhInProdGlobal} on the piecewise constant functions and that
the norm defined in~\eqref{eq:EhInProdDef} is equivalent to the
$\LTWO$ norm.
We omit the proof since these properties are an immediate consequence
of Theorem~\ref{Thm:EhPInProdProperties}.

\medskip
\begin{corollary}\label{Cor:EhEquivNormsGlobal}
  The inner product $\scalEh{\,\cdot\,}{\,\cdot\,}$ defined
  in~\eqref{eq:EhInProdGlobal} has the two properties:
  \begin{itemize}
  \item\textbf{Linear consistency}:
    \begin{align}
      \scalEhP{\Cvh}{\qv} = \scal{\Cvh}{\qv}
      \quad\forall\Cvh\in\Eh,\,\qv\in[\PS{0}(\Th)]^2.
      \label{eq:Eh:accuracy}
    \end{align}
  \item\textbf{Stability}: there exists two real constants $\beta_*$
    and $\beta^*>0$ independent of $\hh$ such that
    \begin{align}
      \beta_*\NORM{\Cvh}{0,\P}^2
      \leq\TNORM{\Cvh}{\Eh(\P)}\leq
      \beta^*\NORM{\Cvh}{0,\P}^2
      \quad\forall\Cvh\in\Eh(\P).
      \label{eq:Eh:stability}
    \end{align}
  \end{itemize}
\end{corollary}

Next, we introduce the \emph{global interpolation operator}
$\IEhP:\HDIV{\Omega}\to\Vh$.
This operator is defined by gluing together its respective
elemental definitions, so that 
\begin{align*}
  \restrict{\IEh(\Cvh)}{\P}=\IEhP(\restrict{\Cvh}{\P})
\end{align*}

To end this subsection we will define a second orthogonal projection that we will use to approximate a term unique to MHD.
Consider a cell $\P$ and define $\PiEhRTP:\Eh(\P)\to\RT{0}(\P)$.
Given $\Cvh\in\Eh(\P)$ the image $\PiEhRTP\Cvh$ is the solution to the variational formulation
\begin{align}
\scalP{\Cvh-\PiEhRTP\Cvh}{\qv} &= 0
  \qquad\text{for all}\quad \qv\in\RT{0}(\P),
  \label{eq:EhProjDefs:B}
\end{align}
where
\begin{align*}
  \RT{0}(\P) = \left\{
   \as\left(\begin{array}{c}1\\0\end{array}\right)
  +\bs\left(\begin{array}{c}0\\1\end{array}\right)
  +\cs\left(\begin{array}{c}x\\y\end{array}\right)
  :\,
  \as,\bs,\cs\in\REAL
  \right\}.
\end{align*}
This projector is also computable using only the degrees of freedom in $\Eh(\P)$.
The strategy is the same as the one presented for $\PizP$. 
Consider $\qv\in\RT{0}(\P)$ and $\ps\in\PS{2}(\P)$ with $\int_{\P} p dA = 0$ such that $\nabla\ps = \qv$.
The terms in Green's Theorem~\eqref{eq:EhComputeProjs} can be computed as before with the only difference being that the quadrature rule used in \eqref{eq:SplitLineInt} needs to be exact for quadratic polynomials.

The \emph{global orthogonal projector}
$\PiEhRT:\Eh\to\RTz{}(\Th)$, where
$\RTz{}(\Th)=\big\{\qv\in\HDIV{\Omega}:\restrict{\qv}{\P}\in\RTz{}(\P)\big\}$ is defined as
\begin{equation}
    \restrict{(\PiEhRT\Cvh)}{\P}=\PiEhRTP(\restrict{\Cvh}{\P})
  \quad\forall\P\in\Th.
\end{equation}

We use $\PiEhRT$ to approximate the term "$\uv\times\Bv$" as can be
evidenced in the MHD variational formulation~\eqref{eq:flowVEMVarForm}.
The main issue with the aforementioned term is that we only have
access to the fluxes of the magnetic field across the edges while the
inner product in the variational formulation requires nodal
evaluations.
We amend this inconsistency by projecting the magnetic field onto the
space of vector polynomial fields $\RTz{}(\P)$ and extract the
necessary evaluations from this projection.
We note that we could use $\PizP$ to extract these vertex evaluations.
However, more complex MHD models have terms of the form
\begin{equation}
  \scal{(\ROTv\Bv)\times\Bv}{\Ds}.
\end{equation}
Such a quantity cannot be estimated using $\PizP{}$ since the codomain of
this projector is the space of constants and their curl is zero.
In this case using the projector $\PiEhRT$ is ideal for low order
approximations.

\subsection{The cell space}
\label{subsec:CellSpace}

The final space that we need to define for the electromagnetic part is
the space of piecewise constant functions on $\Th$, i.e., the space of
constant polynomials in every element $\P$:
\begin{align}
  \Ph = \big\{
  \qsh\in\LTWO(\Omega):\,
  \restrict{\qsh}{\P}\in\PS{0}(\P)\,\,
  \forall\P\in\Th
  \big\}.
  \label{eq:Ph:def}
\end{align}
The degrees of freedom of a function $\qsh\in\Ph$ are given by 
\begin{itemize}
\item[]\TERM{D} the elemental averages of $\qsh$ over every cell
  $\P\in\Th$
  \begin{align}
    \frac{1}{\mP}\int_{\P}\qsh\ dA.
  \end{align}
\end{itemize}
These are represented by red disks in the interior of the cell $\P$,
see the sample picture in Figure~\ref{fig:CellDOF}.
\begin{figure}[t]
  \centering
  \includegraphics[scale=0.5]{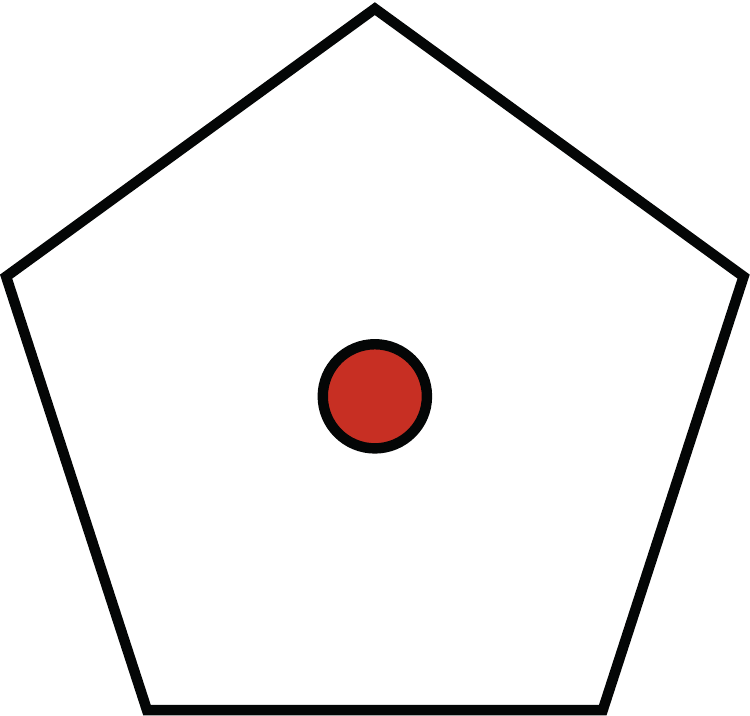}
  \caption{Representation of the degrees of freedom of functions in $\Ph(\P)$.}
  \label{fig:CellDOF}
\end{figure}
It is straightforward
to see that such
degrees of freedom are
unisolvent in $\Ph$.
In fact, the constant value given by restricting a function $\qsh$ to
a cell is precisely the degree of freedom of $\qsh$ associated with
that cell.
We endow the elemental space $\Ph$ with the inner product
\begin{align*}
  \scalPh{\psh}{\qsh}
  = \sum_{\mP}\mP\restrict{\psh}{\P}\restrict{\qsh}{\P}
  \quad\forall\psh,\qsh\in\Ph,
\end{align*}
which is the $\LTWO(\Omega)$ inner product of two piecewise constant
functions.
This inner product induces the norm
\begin{align*}
  \TNORM{\qsh}{\Ph}^2 = \scalPh{\qsh}{\qsh},
  \quad\forall\qsh\in\Ph,
\end{align*}
which is the $\LTWO(\Omega)$-norm restricted to the functions of
$\Ph$, so that
\begin{align*}
  \TNORM{\qsh}{\Ph}=\NORM{\qsh}{0,\Omega}
  \quad\forall\qsh\in\Ph.
\end{align*}
Finally, we define the \emph{global interpolation operator}
$\IPh:\LTWO(\Omega)\to\Ph$ such that for every $\qs\in\LTWO(\Omega)$
we have:
\begin{align}
  \restrict{(\IPh\qs)}{\P} = \frac{1}{|\P|}\int_\P\qsh\ dA
  \quad\forall\P\in\Th.
\end{align}

\subsection{The de~Rham complex.}
\label{subsec:ElectroMagDeRham}

In the previous sections we introduced and discussed the virtual
element spaces $\Vh$, $\Eh$ and $\Ph$.
It is well-known that the spaces $\HROTv{\Omega},\HDIV{\Omega}$ and
$\LTWO(\Omega)$ form the de~Rham chain
\begin{align}
  \begin{CD}
    \HROTv{\Omega} @> \ROTv >> \HDIV{\Omega} @> \DIV >> \LTWO(\Omega).
  \end{CD}
  \label{eq:HROTHDIVL2Chain}
\end{align}
If $\Omega$ is simply connected, the chain is exact, see \cite{Munkres:2018}.
Equivalently, we can say that
\begin{align*}
  \ROTv\HROT{\Omega} = 
  \big\{
  \Cv\in\HDIV{\Omega}:\DIV\Cv = 0
  \big\}.
\end{align*}
In the spirit of constructing a discrete version of the continuous
problem, the spaces $\Vh,\Eh$ and $\Ph$ also form a similar exact
de~Rham chain
\begin{align}\label{eq:VhEhPhChain}
  \begin{CD}
    \Vh @> \ROTv >> \Eh @> \DIV >> \Ph.
  \end{CD}
\end{align}
This chain was first introduced in
\cite{BeiraodaVeiga-Brezzi-Marini-Russo:2016}, and explored in more
details and generality in
\cite{BeiraodaVeiga-Brezzi-Marini-Russo:2016a}.
It reveals that the set of degrees of freedom are transformed in
accordance with the following diagram:\\

\medskip
\begin{minipage}{0.9\textwidth}
  \hspace{1.1cm}
  \begin{picture}(0,100)
    \put(0,10){\includegraphics[width=.2\textwidth]{fig00}}
    \put(100,40){$\xrightarrow{\;\;\;\;\;\;\ROTv\;\;\;\;\;\;\;}$}
    %
    \put(160,10){\includegraphics[width=.2\textwidth]{fig01}}
    \put(250,40){$\xrightarrow{\;\;\;\;\;\;\DIV\;\;\;\;\;\;\;}$}
    %
    \put(300,10){\includegraphics[width=.2\textwidth]{fig02}}
  \end{picture}
\end{minipage}

\medskip
First we want to show that the chain in \eqref{eq:VhEhPhChain} is
well-defined.
This is to say that two important inclusions hold.
The first is presented in the following lemma.
\begin{lemma}
  \label{lem:ElectroMagDeRhamImageOfROT}
  Let $\Vh$ and $\Eh$ be the virtual element spaces
  defined in~\eqref{eq:Vh:def} and \eqref{eq:Eh:def} respectively.
  Then, it holds that
  \begin{align}
    \ROTv\Vh\subset\Eh.
    \label{eq:lemImROTeq}
  \end{align}
\end{lemma}
\BEGINPROOF
Let $\P$ be a mesh cell of $\Th$ and take $\Dsh\in\Vh(\P)$.
  In view of the definition of $\Vh(\P)$, we have that
  $\ROT\ROTv\Dsh=0$ in $\P$ and, clearly,
  $\DIV\ROTv\Dsh=0\in\PS{0}(\P)$.
  Moreover, for every edge $\E\in\partial\P$, we find that
  $\restrict{(\ROT\Dsh\cdot\nv)}{\E}=\restrict{(\nabla\Dsh\cdot\tv)}{\E}\in\PS{0}(\E)$.
  Consequently, $\restrict{(\ROTv\Dsh)}{\P}\in\Eh(\P)$, and, thus,
  $\ROTv\Dsh\in\Eh$ for every $\Dsh\in\Eh$ proving the inclusion
  relation in \eqref{eq:lemImROTeq}.
\ENDPROOF

From Lemma~\ref{lem:ElectroMagDeRhamImageOfROT}, we know that
$\ROTv\Dsh\in\Eh$ if $\Dsh\in\Vh$.
Moreover, we can compute the degrees of freedom of $\ROTv\Dsh$ in
$\Eh$ from the degrees of freedom of $\Dsh$ in $\Vh$.
In fact, by applying the fundamental theorem of line integrals,
we find that
\begin{align}
  \frac{1}{\mE}\int_{\E}\ROTv\Dsh\cdot\nv\ d\ell =
  \frac{1}{\mE}\int_{\E}\nabla\Dsh\cdot\tv\ d\ell=
  \frac{\Dsh(\Vpp)-\Dsh(\Vp)}{\mE},
  \label{eq:RotofVh}
\end{align}
for every edge $\E$ of the polygonal boundary $\partial\P$ with
endpoints $\Vp$ and $\Vpp$ (oriented from $\Vp$ to $\Vpp$), where
again we used the identity $\nv\cdot\ROTv(\Dsh)=\tv\cdot\nabla(\Dsh)$.
In view of equation~\eqref{eq:RotofVh}, we can read the necessary
information to identify the image of the rotational of $\Vh$ as a
subset of $\Eh$ by using the degrees of freedom defined for $\Vh$.

The second inclusion in the chain \eqref{eq:VhEhPhChain} is the conclusion of the following lemma.
\begin{lemma}
  \label{lem:ElectroMagDeRhamImageOfDIV}
  Let $\Eh$ and $\Ph$ be the virtual element spaces
  defined in~\eqref{eq:Eh:def}, and
  \eqref{eq:Ph:def}.
  Then, it holds that
  \begin{align}
    \DIV\Eh\subset\Ph.
    \label{eq:lemImDIVeq}
  \end{align}
\end{lemma}
\BEGINPROOF
  To verify this inclusion, we only need to note that any
  $\Cvh\in\Eh(\P)$ is such that $\DIV\Cvh\in\PS{0}(\P)$ from the
  definition of $\Eh(\P)$.
  It follows that $\DIV\Cvh\in\PS{0}(\Th)=\Ph$ for every $\Cvh\in\Eh$, which is
  the second inclusion relation in
  \eqref{eq:lemImDIVeq}.
\ENDPROOF

\medskip
Noting that the divergence of a function in $\Eh$ lies in $\Ph$ will help us identify that its divergence can be entirely characterized by its set of degrees of freedom in $\Ph$.
Take $\Cvh\in\Eh$.
From the divergence theorem we have that
\begin{equation}
  \frac{1}{\mP}\int_\P\DIV\Cvh\ dA
  = \frac{1}{\mP}\int_{\partial\P}\DIV\Cvh\ dA
  = \frac{1}{\mP}\sum_{\E\in\partial\P}\mE\left(\frac{1}{\mE}\int_{\E}\Cvh\cdot\nv\ dA\right).
  \label{eq:DivofEh}
\end{equation}
Hence, we can evaluate the divergence of a function $\Cvh\in\Eh$ using
only its degrees of freedom.

The results summarized by equations \eqref{eq:RotofVh} and
\eqref{eq:DivofEh} are essential in order to further study the spaces
$\Vh,\Eh$ and $\Ph$ and their relationship with the larger spaces
$\HROTv{\Omega},\HDIV{\Omega}$ and $\LTWO(\Omega)$.
These spaces form the commutative diagram
\begin{align}\label{eq:ElectroDeRhamComplex}
  \begin{CD}
    \HROTv{\Omega} @> \ROTv >> \HDIV{\Omega} @> \DIV >> \LTWO(\Omega)\\
    @VV \IVh V  @VV \IEh V @VV \IPh V \\
    \Vh  @> \ROTv >> \Eh @> \DIV >> \Ph
  \end{CD}
\end{align}
The proof of this theorem is broken into two lemmas.
The first lemma, presented below, involves the spaces $\HROTv{\Omega}$ and $\HDIV{\Omega}$, and their discrete counterparts $\Vh$ and $\Eh$.
\begin{lemma}\label{lem:ROTDiagramCommutes}
  The following identity holds
  \begin{equation}
      \forall\Ds\in\HROTv{\Omega}:\quad
      \IEh\circ\ROTv(\Ds) = \ROTv\circ\IVh(\Ds),
      \label{eq:ROTCommutes}
  \end{equation}
  i.e., the interpolation and the rotational operators commute. 
\end{lemma}
\BEGINPROOF
  Take a scalar function $\Ds\in\HROTv{\Omega}$.
  By definition, the degrees of freedom of $\IEh\circ\ROTv(\Ds)$ in
  $\Eh$ are the same of $\ROTv\Ds$.
  So, if $\E$ is a mesh edge oriented from endpoint $\Vp$ to $\Vpp$,
  the theorem of line integral yields:
  \begin{align}
    \frac{1}{\mE}\int_{\E} \nv\cdot\ROTv\Ds\ d\ell
    = \frac{1}{\mE}\int_{\E} \tv\cdot\nabla\Ds\ d\ell
    = \frac{\Ds(\Vpp)-\Ds(\Vp)}{\mE}.
    \label{eq:Commutelemeq1}
  \end{align}
  In turn, the degrees of freedom of $\ROTv\circ\IVh(\Ds)$ are given
  by
  \begin{align}
    \frac{1}{\mE}\int_{\E} \nv\cdot\ROTv\big(\IVh(\Ds)\big)\ d\ell
    = \frac{1}{\mE}\int_{\E} \tv\cdot\nabla\IVh(\Ds)\ d\ell
    = \frac{\IVh\Ds(\Vpp)-\IVh\Ds(\Vp)}{\mE},
    \label{eq:Commutelemeq2}
  \end{align}
  using again the theorem of line integral yields.
  The definition of operator $\IVh$ is such that 
  \begin{align*}
    \IVh\Ds(\Vp) =\Ds(\Vp)
    \quad\mbox{and}\quad
    \IVh\Dsh(\Vpp)=\Ds(\Vpp).
  \end{align*}
  Thus, equations \eqref{eq:Commutelemeq1} and
  \eqref{eq:Commutelemeq2} imply that the functions
  $\IEh\circ\ROTv(\Ds)$ and $\ROTv\circ\IVh(\Ds)$ have the same
  degrees of freedom in $\Eh$ and relation \eqref{eq:ROTCommutes}
  follows from the unisolvence.
\ENDPROOF

The second lemma involves the spaces $\HDIV{\Omega},\LTWO(\Omega),\Eh$ and $\Ph$.
\begin{lemma}\label{lem:DIVDiagramCommutes}
  The following identity holds
  \begin{equation}
      \forall \Cv\in\HDIV{\Omega}:\quad
      \IPh\circ\DIV(\Cv) = \DIV\circ\IEh(\Cv),
      \label{eq:DIVCommutes}
  \end{equation}
  i.e., the interpolation and the divergence operators commute. 
\end{lemma}
\begin{proof}
  We prove \eqref{eq:DIVCommutes} by verifying that the two
  functions in the left and right side share the same degrees of
  freedom in $\Ph$.
  Take a vector-valued field $\Cv\in\HDIV{\Omega}$.
  By definition, the degrees of freedom of $\IPh\circ\DIV(\Cv)$ in
  $\Eh$ are the same of $\DIV(\Cv)$.
  So, if $\P$ is a mesh cell, the divergence theorem yields
  \begin{align}
    \frac{1}{\mP}\int_\P\DIV\Cv\ dA
    = \frac{1}{\mP}\sum_{\E\in\partial\P}\int_{\E}\Cv\cdot\nv\ d\ell.
    \label{eq:Commutelemeq3}
  \end{align}
  In turn, the degrees of freedom of $\DIV\circ\IEh(\Cv)$ are given by
  \begin{align}
    \frac{1}{\mP}\int_\P\DIV\IEh(\Cv)\ dA
    = \frac{1}{\mP}\sum_{\E\in\partial\P}\int_{\E}\IEh(\Cv)\cdot\nv\ d\ell.
    \label{eq:Commutelemeq4}
  \end{align}
  The definition of operator $\IEh$ is such that 
  \begin{align*}
    \forall\E\in\partial\P:\quad\int_{\E} \IEh\Cv\cdot\nv\ d\ell
    =\int_{\E} \Cv\cdot\nv\ d\ell.
  \end{align*}
  Thus, equations \eqref{eq:Commutelemeq3} and
  \eqref{eq:Commutelemeq4} imply that the two functions
  $\IPh\circ\DIV(\Cv)$ and $\DIV\circ\IEh(\Cv)$ have the same degrees
  of freedom in $\Ph$ and relation \eqref{eq:DIVCommutes} follows from
  the unisolvence.
\end{proof}

We summarize our findings in the following theorem
\begin{theorem}
  \label{Thm:ElectroMagDeRhamComplex}
  The chain in \eqref{eq:VhEhPhChain} is well-defined and exact, and
  diagram \eqref{eq:ElectroDeRhamComplex} is commutative.
\end{theorem}
\begin{proof}
  Lemmas~\ref{lem:ElectroMagDeRhamImageOfROT} and \ref{lem:ElectroMagDeRhamImageOfDIV} prove that 
  \eqref{eq:VhEhPhChain} is well-defined.
  Lemmas~\ref{lem:ROTDiagramCommutes} and \ref{lem:DIVDiagramCommutes} prove that diagram
  \eqref{eq:ElectroDeRhamComplex} is commutative.
  Hence, we are only left to prove that the de~Rham chain
  \eqref{eq:VhEhPhChain} is exact, or, equivalently
  that
  \begin{align}
    \ROTv\Vh
    = \KER\big(\DIV\Eh\big)
    = \big\{ \Cvh\in\Eh:\DIV\Cvh=0 \big\}.
    \label{eq:RotVheqkerDiv}
  \end{align}
  Take $\Dsh\in\Vh$.
  Lemma~\ref{lem:ElectroMagDeRhamImageOfROT} implies that
  $\ROTv\Dsh\in\Eh$, and, obviously, $\DIV\ROTv\Dsh=0$, so that
  $\ROTv\Dsh\in\KER\big(\DIV\Eh\big)$ as defined in
  \eqref{eq:RotVheqkerDiv}, which implies that
  $\ROTv\Vh\subseteq\KER\big(\DIV\Eh\big)$.
  Next, consider $\Cvh\in\Eh$ with $\DIV\Cvh=0$.
  Since $\Eh\subset\HDIV{\Omega}$, then $\Cvh\in\HDIV{\Omega}$ and the
  exactness of chain \eqref{eq:HROTHDIVL2Chain} implies the existence
  of a scalar function $\Ds\in\HROTv$ such that $\Cvh=\ROTv\Ds$.
  Moreover, $\IVh\Ds\in\Vh$ must verify
  \begin{equation}
    \ROTv\circ\IVh\Ds
    = \IEh\circ\ROTv\Ds
    = \IEh\Cvh
    = \Cvh,
  \end{equation}
  which implies that $\Cvh$ is the rotational of a function of $\Vh$
  and, thus, $\KER\big(\DIV\Eh\big)\subset\ROTv\Vh$.
\end{proof}

\medskip

\subsection{Fluid Flow}
\label{subsec:ElementsFluidFlow}

In this section, we briefly review the virtual element spaces for the
discretization of the fluid-flow equations in the MHD model.
These spaces were originally proposed
in~\cite{Vacca:2018,BeiraodaVeiga-Lovadina-Vacca:2017,BeiraodaVeiga-Lovadina-Vacca:2018}.

The first virtual element space is used to discretize the pressure.
We consider a subspace of $\Ph$ as defined in Section~\ref{subsec:CellSpace}.
This subspace is given by
\begin{align}
  \Phzr = \left\{ \qsh\in\Ph:\int_\Omega\qsh\ dA = 0\right\}.
  \label{eq:PhzrDOF}
\end{align}
The degrees of freedom of a function $\qsh\in\Phzr$ are given by
\begin{itemize}
\item[]\TERM{P'}{} $\qquad\displaystyle\int_\P\qsh\ dA\quad$ for every $\P\in\Th$.
\end{itemize}
These degrees of freedom are the same of $\Ph$ up to a multiplicative
scale factor equal to $1\slash{\mP}$.
For $\Phzr$ we prefer this definition because the
integral over $\Phzr$ of a function $\qsh\in\Phzr$ is given by
summing the  degrees of freedom of $\qsh$:
\begin{align*}
  \int_\Omega\qsh\ dA = \sum_{i=0}^N\mbox{dof}_i(\qsh).
\end{align*}
\medskip
If we enumerate the cells in the mesh $\Omega_h$ as $\{P_i:1\le i\le N\}$ then 
the functions $\mbox{dof}_i:\Phzr\to\REAL$ map each function in $\Phzr$ to the degree of freedom associated with $P_i$.
\\

The virtual element space for the velocity approximation reads as
\begin{subequations}
  \begin{align}
    \Vvh(\P)
    &=\Big\{
    \vvh\in\Honev{\P}\,:\,\restrict{\vvh}{\partial\P}\in\big[\Bspace{\partial\P}\!\big]^2,\,
    \DIV\vvh\in\PS{0}(\P),
    \nonumber\\
    &\hspace{3cm}
    -\Delta\vvh-\nabla\ss=\zrv\,\,\textrm{for~some~}
    \ss\in\LTWOzr(\P)
    \Big\},
    \intertext{where}
    \Bspace{\partial\P}
    &=\Big\{\vs\in\CS{0}(\partial\P):\,
    \restrict{\vs}{\E}\in\PS{2}(\E)
    \,\,\forall\E\in\partial\P
    \Big\}.
  \end{align}
\end{subequations}

A function $\vvh\in\Vvh(\P)$ is uniquely characterized by the
following degrees of freedom:

\begin{itemize}
\item[]\DOFS{\Dv1}{} pointwise evaluations of $\vvh$ at the vertices of
  $\P$;

  \smallskip
\item[]\DOFS{\Dv2}{} pointwise evaluations at $\vvh$ at the midpoint of
  the edges of $\partial\P$.
\end{itemize}
These are represented by blue disks centered at the nodes and the
mid-point of edges of the cell $\P$, see the sample picture in
Figure~\ref{fig:TVhDOF}.
\begin{figure}[t]
  \centering
  \includegraphics[scale=0.5]{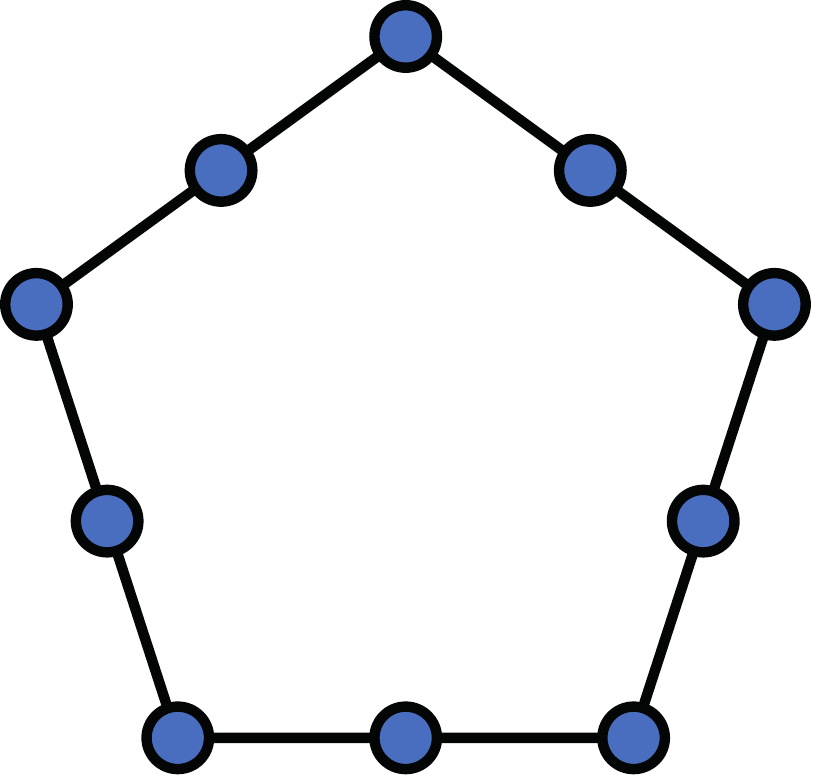}
  \caption{Representation of the degrees of freedom of functions in
    $\Vvh(\P)$.}
  \label{fig:TVhDOF}
\end{figure}
\begin{theorem}
  Define the map $\calK_{\Vvh}:\Vvh(\P)\to\REAL^N$ such that for any
  $\vvh\in\Vvh(\P)$ the array $\calK_{\Vvh}\vvh$ is given by the
  degrees of freedom pf $\vvh$.
  Then, $\calK_{\Vvh}$ is bijective.
\end{theorem}
\BEGINPROOF
The proof is omitted and can be found 
in~\cite{BeiraodaVeiga-Lovadina-Vacca:2017}.
\ENDPROOF

The largest polynomial space that is contained in $\Vvh(\P)$ is the
space of divergence-free, quadratic polynomial vectors, which is
formally written as:
\begin{align*}
  \PSv{}(\P) = \left\{
  \qv\in\big[\PS{2}(\P)\big]^2:\DIV\qv\in\PS{0}(\P) \right\}.
\end{align*}
Let $\calEP$ denote the set of vertices and midpoints of the edges
forming the polygonal boundary $\partial\P$, and consider the
projector $\PinP:\Vvh(\P)\to\PSv{}(\P)$ such that the vector
polynomial $\PinP\vvh$ for $\vvh\in\Vvh(\P)$ is the solution to the
following variational problem
\begin{subequations}
  \begin{align}
    \int_{\P}\nabla\PinP\vvh:\nabla\qv\ dA &= \int_{\P}\nabla\vvh:\nabla\qv\ dA
    \quad\forall\qv\in\PSv{}(\P),
    \label{Eq:EllipticProjTVh:a}
    \\[0.5em]
    P_0 \left(\PinP\vvh \right) &= \P_0(\vvh),
    \label{Eq:EllipticProjTVh:b}
  \end{align}
\end{subequations}
where
\begin{equation}
    P_0(\vvh)=\sum_{\V\in\calEP}\vvh(\V).
\end{equation}
We recall that $\nabla\vvh$, $\nabla\PinP\vvh$ and $\nabla\qv$
are $2\times2$-sized tensors and ``:'' is the usual euclidean scalar
product saturating both indices of such tensors, so that
\begin{equation}
    \nabla\vv:\nabla\wv=\sum_{i,j}(\partial\vs_i\slash\partial\xs_j)(\partial\ws_i\slash\partial\xs_j).
\end{equation}
To prove that this projection operator is computable, we need to show
that the right-hand side of \eqref{Eq:EllipticProjTVh:a} is computable
for every vector-valued field $\vvh\in\Vvh(\P)$ and $\qv\in\PSv{}(\P)$
using only the degrees of freedom of $\vvh$.
To this end, we first apply the Green theorem to find that
\begin{align}
  \int_{\P}\nabla\vvh:\nabla\qv\ dA
  = \int_{\partial\P}\vvh\nabla\qv\cdot\nv\ d\ell
  - \int_P\vvh\cdot\Delta\qv\ dA
  \label{eq:TVhProjComputing1}
\end{align}
Then, we note that $\Delta\qv\in\big[\PS{0}(\P)\big]^2$ and the scalar
polynomial $\gs$ satisfying that
\begin{equation}
    \gs=\Delta\qv\cdot(\xv-\xvP)\quad\mbox{with}\quad\
    \nabla\gs=\Delta\qv\quad\mbox{and}\quad
    \int_{\P}\gs(\xv)\ dA=0.
\end{equation}

Using this identity in~\eqref{eq:TVhProjComputing1} we see that
\begin{align*}
  \int_{\P}\nabla\vvh:\nabla\qv\ dA
  &
  = \int_{\partial\P}\vvh\cdot\nabla\qv\cdot\nv\ d\ell
  - \int_P\vvh\cdot\Delta\qv\ dA
  \\[0.5em]
  &
  = \int_{\partial\P}\vvh\cdot\nabla\qv\cdot\nv\ d\ell
  + \int_{\P}(\DIV\vvh)\gs\ dA
  - \int_{\partial\P}\gs\vvh\cdot\nv\ d\ell
  \\[0.5em]
  &= \TERM{T1}{} + \TERM{T2}{} + \TERM{T3}{}.
\end{align*}
The boundary integrals \TERM{T1}{} and \TERM{T3}{} are computable
since the trace of $\vvh$ on every edge of $\partial\P$ can be
interpolated from its degrees of freedom, while the cell integral
\TERM{T2}{} is zero because $\DIV\vvh\in\PS{0}(\P)$ and we have that:
\begin{align*}
  \TERM{T2}{}
  = \int_{\P}(\DIV\vvh)\gs\ dA
  = \restrict{(\DIV\vvh)}{\P}\int_{\P}\gs\ dA
  = 0.
\end{align*}
Moreover, we can see that $\DIV\vvh$ is also computable on using the
degrees of freedom of $\vvh$ and the divergence theorem:
\begin{align*}
  &\restrict{(\DIV\vvh)}{\P}\,\mP
  = \int_{\P}\DIV\vvh\ dA
  = \int_{\partial\P}\vvh\cdot\nv\ d\ell
  \intertext{so~that}
  &\restrict{(\DIV\vvh)}{\P} 
  = \frac{1}{\mP}\sum_{\E\in\partial\P}\int_{\E}\vvh\cdot\nv\ d\ell.
\end{align*}
This formula makes it possible to compute the divergence of $\vvh$
using only the boundary information that can be extracted from
\DOFS{\Dv1}{} and \DOFS{\Dv2}{}.

\medskip
To approximate the terms of the MHD variational formulation that
depends on the time derivative of the velocity field, we need the
$\LTWO$-orthogonal projection of the virtual element vector-valued
fields.
However, such projection operator is not directly computable in the
space $\Vvh(\P)$, so we change the definition of the space according
to the enhancement strategy in~\cite{Vacca:2018} that we briefly
review below.
First, we consider the auxiliary finite dimensional functional spaces
\begin{align}
  \Gspace{2}{\P} &= \nabla\PS{3}(\P)
  \intertext{and its orthogonal complement in $[\PS{2}(\P)]^2$}
  \Gperp {2}{\P} &=
  \big\{
  \gvp\in[\PS{2}(\P)]^2:\,\scalP{\gvp}{\gv}=0
  \,\,\forall\gv\in\Gspace{2}{\P}
  \big\}.
\end{align}
Then, we introduce the ``extended'' virtual element space
\begin{align}
\nonumber
  \Uvh(\P)
  =
  \Big\{
  \vvh\in\Honev{\P}:\,
  &
  \restrict{\vvh}{\partial\P}\in
  \left[\Bspace{\partial\P}\right]^2,\,\,\DIV\vvh\in\PS{0}(\P)
  \\[0.5em]
  &\hspace{-0.5cm}
  -\Delta\vvh-\nabla s = \gvp\,\textrm{for~some~}s\in\LTWOzr(\P),\,\gvp\in\Gperp{2}{\P}
  \Big\}.
\end{align}
The projector $\PinP$ can also defined in $\Uvh(\P)$ and is computable
using only the information from the degrees of freedom
\DOFS{\Dv1}{}-\DOFS{\Dv2}{}.
However, the degrees of freedom \DOFS{\Dv1}{}-\DOFS{\Dv2}{} are \emph{not}
unisolvent in $\Uvh(\P)$.
Instead, they are unisolvent in $\TVh(\P)$, the subspace of $\Uvh(\P)$
that is formally defined as
\begin{align}
  \TVh(\P)=
  \Big\{
  \vvh\in\Uvh(\P):\,
  \ScalP{\vvh-\PinP\vvh}{\gvp}=0\,\,
  \forall\gvp\in\Gperp{2}{\P}/\REAL^2
  \Big\}.
\end{align}

Since the degrees of freedom \DOFS{\Dv1}{}-\DOFS{\Dv2}{} are
unisolvent in this space, we can define the elemental interpolation
operator $\ITVhP:\big[\HONE(\P)\big]^2\to\TVh(\P)$ such that
$\uv\in\big[\HONE(\P)\big]^2$ and $\ITVhP(\uv)\in\TVh(\P)$ have the
same degrees of freedom.

\medskip
We define the $\LTWO$-orthogonal projection operator
$\PizP:\TVh(\P)\to\PSv{}(\P)$, so that for every $\vvh\in\TVh(\P)$,
the vector polynomial $\PizP{}\vvh$ is the solution to the variational
problem
\begin{align}
  \scalP{\PizP\vvh-\vvh}{\qv} = 0.
  \quad\forall\qv\in\PSv{}(\P).
\end{align}
We show that $\PizP\vvh$ is computable using only the degrees of
freedom \DOFS{\Dv1}{}-\DOFS{\Dv2}{} of $\vvh$.
Let $\vvh\in\TVh(\P)$ and $\qv\in\PSv{}(\P)$.
We consider the decomposition $\qv=\nabla\gs+\gvp$ where
$\gs\in\PS{3}(\P)$ and apply the Green theorem to obtain:
\begin{align}
  \int_{\P}\vvh\cdot\qv\ dA
  &
  = \int_{\P}\vvh\cdot\nabla\gs\ dA
  + \int_{\P}\vvh\cdot\gvp\ dA.
  \\[0.5em]
  &
  = \int_{\partial\P}\vvh\cdot\nv\gs\ d\ell
  - \int_{\P}\DIV\vvh\gs\ dA
  + \int_{\P}\vvh\cdot\gvp\ dA.
  \label{eq:TVhIntegrals}
\end{align}
The boundary integral on the right is computable as the trace of
$\vvh$ on every edge is a quadratic polynomial that can be
interpolated from the degrees of freedom freedom
\DOFS{\Dv1}{}-\DOFS{\Dv2}{} of $\vvh$.
The second integral in the right-hand side is zero because
$\restrict{\DIV\vvh}{\P}\in\PS{0}(\P)$ and we can always take a
function $\gs$ with zero elemental average (alternatively, it can be
computed by noting that the divergence of $\vvh$ is given by the
divergence theorem).
To compute the second integral in \eqref{eq:TVhIntegrals} we first
write
\begin{align*}
  \gvp = \cv+(\gvp-\cv),\quad
  \cv  = \frac{1}{|\P|}\int_{\P}\gvp\ dA.
\end{align*}
Since $\cv\in\PS{0}(\P)$ we can find $\qs\in\PS{1}(\P)$ such that
$\cv=\nabla\qs$ and see that
\begin{align*}
  \int_{\P}\vvh\cdot\cv\ dA
  = \int_{\partial\P}\qs\vvh\cdot\nv\ d\ell
  - \int_{\P} \DIV\vvh\qs\ dA.
\end{align*}
Finally, we note that $\gvp-\cv\in\Gperp{2}{\P}/\REAL^2$, and the
definition of space $\TVh(\P)$ implies that
\begin{align*}
  \int_{\P} \vvh\cdot (\gvp-\cv)\ dA = \int_{\P}\PinP\vvh\cdot(\gvp-\cv)\ dA.
\end{align*}
The last integral is computable because $\PinP\vvh$ is computable from
the degrees of freedom of $\vvh$ and $\gvp-\cv$ is a known function.

\medskip
We use the projection operators $\PizP$ and $\PinP$ to define the
inner product and semi-inner product in $\TVh(\P)$ as follows
\begin{align}
  \scalTVhP{\uvh}{\vvh}
  &
  = \scalP{\PizP\uvh}{\PizP\vvh}
  + \calS^{\TVh}_{\P}\big( (1-\PizP)\uvh,(1-\PizP)\vvh \big),
  \label{Eq:TVhLocalInProds:A}
  \\[0.5em]
  \scalTVhPNa{\uvh}{\vvh}
  &
  = \scalP{\nabla\PinP\uvh}{\nabla\PinP\vvh}
  + \calT^{\TVh}_{\P}\big( \nabla(1-\PinP)\uvh,\nabla(\calI-\PinP)\vvh \big),
  \label{Eq:TVhLocalInProds:B}
\end{align}
for every $\uvh$, $\vvh$ in $\TVh(\P)$.
Here, $\calS^{\TVh}_\P$ and $\calT^{\TVh}_{\P}$ are the stabilizing
terms, i.e., any continuous bilinear forms for which there exist two
pairs of strictly positive constants $(t_*,t^*)$ and $(s_*,s^*)$,
which are independent of $\hh$, such that
\begin{align}
  s_*\NORM{\vvh}{0,\P}^2
  \leq\calS^{\TVh}_{\P}(\vvh,\vvh)\leq
  s^*\NORM{\vvh}{0,\P}^2
  & \quad\forall\vvh\in\TVh(\P)\cap\KER(\PizP{}),
  \label{eq:TVh:inner:product:stability}
  \\[0.5em]
  t_*\NORM{\nabla\vvh}{0,\P}^2
  \leq\calT^{\TVh}_{\P}(\nabla\vvh,\nabla\vvh)\leq
  t^*\NORM{\nabla\vvh}{0,\P}^2
  & \quad\forall\vvh\in\TVh(\P)\cap\KER(\PinP{}).
  \label{eq:TVh:semi:inner:product:stability}
\end{align}
In practice, we can design such stabilizations as
in~\cite{Mascotto:2018,Dassi-Mascotto:2018}.
The inner and semi-inner products respectively defined in
\eqref{eq:TVh:inner:product:stability} and
\eqref{eq:TVh:semi:inner:product:stability}
satisfy two fundamental properties, e.g., \textbf{Polynomial
  Consistency} and \textbf{Stability}, which are stated in the
following theorem.
\begin{theorem}
  \label{thm:ConsistencyStabilityTVhP}
  The inner product defined in~\eqref{eq:TVh:inner:product:stability} and
  semi-inner products defined in~\eqref{eq:TVh:semi:inner:product:stability}
  have the two properties:
  \begin{itemize}
  \item \textbf{Polynomial Consistency}: for every vector-valued field
    $\vvh\in\TVh(\P)$ and polynomial $\qv\in[\PS{2}(\P)]^2$ it holds
    that:
    \begin{align}
      \scalTVhP{\vvh}{\qv}   = \scalP{\vvh}{\qv}
      \quad\textrm{and}\quad
      \scalTVhPNa{\vvh}{\qv} = \scalP{\nabla\vvh}{\nabla\qv}.
    \end{align}
  \item \textbf{Stability}: there exists a pairs of strictly positive
    constants $(t_*,t^*)$ independent of $\hh$ such that
    \begin{align}
      t_*\NORM{\vvh}{0,\P}^2
      \leq\scalTVhP{\vvh}{\vvh}\leq
      t^*\NORM{\vvh}{0,\P}^2,
      \intertext{and}
      t_*\NORM{\nabla\vvh}{0,\P}^2
      \leq\scalTVhPNa{\vvh}{\vvh}\leq
      t^*\NORM{\vvh}{0,\P}^2.
      \label{Eq:LocalStabTVh}
    \end{align}
    for any $\vvh\in\TVh(\P)$.
  \end{itemize}
\end{theorem}
\BEGINPROOF
The proof of this theorem uses the same argument of the proof of
Theorem~\ref{Thm:EquivalentInProdsVh(P)} and is, thus, omitted.
\ENDPROOF
\medskip
The global space $\TVh$ is given by
\begin{align*}
  \TVh = \big\{
  \vvh\in\Honev{\Omega}:\forall\P\in\Omega_h\quad\restrict{\vvh}{\P}\in\TVh(\P)
  \big\}.
\end{align*}
We will endow this space with an inner product and a semi-inner product beginning with their local definitions.
\begin{align*}
  \scalTVh{\uvh}{\vvh}   &= \sum_{\P\in\Omega_h}\scalTVhP{\uvh}{\vvh}
  \quad\forall\uvh,\vvh\in\TVh,
  \\[0.5em]
  \scalTVhNa{\uvh}{\vvh} &= \sum_{\P\in\Omega}\scalTVhPNa{\uvh}{\vvh}
  \quad\forall\uvh,\vvh\in\TVh.
\end{align*}
These inner product and semi-inner product induce the norms and
semi-norm
\begin{subequations}
  \begin{align}
    \TNORM{\vvh}{\TVh} & = \scalTVh{\vvh}{\vvh}^{1/2},
    \quad\snorm{\vvh}{\TVh} = \scalTVhNa{\vvh}{\vvh}^{1/2},\\[0.5em]
    \TNORM{\vvh}{1,\TVh} & = \left(\TNORM{\vvh}{\TVh}^2 + \snorm{\vvh}{\TVh}^{2}\right)^{1/2}.
  \end{align}
  \label{Eq:GlobalTVhForms}
\end{subequations}
The norm in the topological dual space of $\TVhzr$ denoted by
$\TVhzr^{\,\prime}$ is:
\begin{align*}
  \TNORM{\fvh}{-1,\TVh}
  = \sup_{\vvh\in\TVhzr\setminus\{0\}}
  \frac{\scalTVh{\fvh}{\vvh}}{\snorm{\vvh}{\TVh}}.
\end{align*}
The global inner product and semi-inner product and their induced norm
and semi-norm also satisfy the \textbf{consistency} and
\textbf{stability} properties as stated in the following corollary.

\begin{corollary}
  \label{cor:EquivalentNormsTVh}
  The norms and semi-norm in \eqref{Eq:GlobalTVhForms} are equivalent
  to the $[\LTWO(\Omega)]^2$ and $\Honev{\Omega}$ inner products and
  semi-inner product respectively.
  In other words, there exists $t_*,t^*>0$ independent of the
  mesh characteristics such that for any $\vvh\in\TVh$ it holds that

  \vspace{-\baselineskip}
  \begin{subequations}\label{Eq:GlobalStabilityTVh}
    \begin{align}
      &t_* \NORM{\vvh}{0,\Omega}^2\le 
      \TNORM{\vvh}{\TVh}\le 
      t^*\NORM{\vvh}{0,\Omega}^2,\\[0.5em]
      &t_* \NORM{\nabla\vvh}{0,\Omega}^2\le 
      \TNORM{\vvh}{\TVh}^{2,\nabla}\le 
      t^*\NORM{\nabla\vvh}{0,\Omega}^2,\\[0.5em]
      &t_* \NORM{\vvh}{1,\Omega}^2
      \le 
      \TNORM{\vvh}{1,\TVh}^2
      \le 
      t^*\NORM{\vvh}{1,\Omega}^2.
    \end{align}
  \end{subequations}
\end{corollary}
\BEGINPROOF
This result is an immediate consequence of
Theorem~\ref{thm:ConsistencyStabilityTVhP} and we omit the proof.
\ENDPROOF

We can define the global interpolation projector
$\ITVh:\big[\HONE(\Omega)\big]^2\to\TVh$ that is such that
$\restrict{\ITVh(\vv)}{\P} = \ITVhP\left(\restrict{\vv}{\P}\right)$
for every $\P\in\Th$ and $\vv\in[\HONE(\Omega)]^2$.
We also define the global space $\TVhzr$ of the functions in $\TVh$
with zero trace on the boundary of $\Omega$:
\begin{align*}
  \TVhzr = \left\{ \vvh\in\TVh:\restrict{\vvh}{\partial\Omega}\equiv \zrv\right\}.
\end{align*}

We note that the spaces $\TVh$ and $\Ph$ also form a de~Rham complex of the form
\begin{equation*}
\begin{CD}
    \Honev{\Omega} @>\DIV>> \LTWO(\Omega) \\
    @VV\ITVh V @VV\IPh V \\
    \TVh @>\DIV>> \Ph 
\end{CD}
\end{equation*}
However, unlike the case presented in subsection~\ref{subsec:ElectroMagDeRham} this chain is not commutative.
The set of degrees of freedom are transformed in accordance to the
diagram:

\medskip
\begin{minipage}{0.9\textwidth}
  \begin{picture}(0,100)
    \put(100,10){\includegraphics[scale=0.4]{fig03}}
    \put(200,40){$\xrightarrow{\;\;\;\;\;\;\DIV\;\;\;\;\;\;\;}$}
    %
    \put(250,10){\includegraphics[scale=0.4]{fig02}}
  \end{picture}
\end{minipage}

\medskip
Finally, we present a result regarding the stability of the virtual
element approximations using the spaces $\TVh$ and $\Phzr$, which are
specifically chosen to satisfy an inf-sup condition.
\begin{theorem}
  \label{thm:StableStokesPair}
  There exists a projector $\Pi_h:\Honevzr{\Omega}\to\TVhzr$ such that
  \begin{align*}
    \scalPh{\DIV\Pih\vv}{\qsh}=\scalPh{\DIV\vv}{\qsh}
    \quad\textrm{and}\quad
    \TNORM{\Pih\vv}{1,\TVh}\leq\Cs_\pi\NORM{\vv}{1,\Omega},
  \end{align*}
  for every vector-valued field $\vv\in\Honevzr{\Omega}$ and scalar
  function $\qsh\in\Phzr$.
  Here, $C_\pi$ is a positive, real constant independent of $\hh$.
  The two spaces $\TVh$ and $\Phzr$ form an inf-sup stable pair and
  satisfy
  \begin{align}
    \inf_{\qsh\in\Phzr}\sup_{\vvh\in\TVhzr}
    \frac{\scalPh{\DIV\vvh}{\qsh}}{\TNORM{\vvh}{1,\TVh}\TNORM{\qsh}{\Ph}}>\beta_{\pi}>0
  \end{align}
  for some real and strictly positive constant $\beta_{\pi}$.
\end{theorem}
This theorem provides a major result since a pair of finite element
spaces that do not satisfy such an inf-sup condition will yield
unstable simulations of fluid flow phenomena.
\section{Energy Estimates}
\label{sec:EnergyEstimate}

The conforming nature of our VEM allows us to mimic many properties of
the continuous setting.
Among them, one of the more important is the preservation of certain
types of estimates in the $\LTWO(\Omega)$-norm, which are obtained by
testing the variational formulation against the exact solution and
applying Gronwall's lemma, see~\cite{Emmrich:1999}.
In this section, we present an estimate of this type for the
continuous system
\eqref{eq:flowStrongConsMass}-\eqref{eq:flowStrongAmperesLaw}
and its discrete
counterpart
\eqref{eq:flowVEMVarFormConsMom}-\eqref{eq:flowDiscCurrentDensity}.

We start with the decompositions
\begin{align}
  \uv = \huv + \uv_b
  \quad\textrm{and}\quad
  \Es = \hEs + \Es_b,
  \label{Eq:Decomposition}
\end{align}
where $\uv_b$ and $\Es_b$ are extensions of the boundary conditions inside the domain and 
$\huv\in\Honevzr{\Omega}$ and $\hEs\in\HROTvzr{\Omega}$ are functions to be found.
The extension to the boundary condition of the velocity field is such
that
\begin{align*}
  \DIV\uv_b = 0 \mbox{~~~~in~~~}\Omega
  \quad\textrm{and}\quad
  \uv_{b}(\xv) = \zrv
  \mbox{~~~~if~~~}d(\xv,\partial\Omega)\geq\epsilon
\end{align*}
for $h>\epsilon>0$ for a given threshold $\epsilon$,
$d(\xv,\partial\Omega)$ being the distance between $\xv$ and the
boundary $\partial\P$.
We can construct such an extension $\uv_{b}$ by defining the domain 
$\Omega_\epsilon = \big\{
\xv\in\Omega:d(\xv,\partial\Omega)<\epsilon \big\}$ 
and taking $\uv_{b}$ to be the solution to the problem:
\begin{subequations}
  \label{eq:bnd:pblm}
  \begin{align}
    -\Delta\huv_b + \nabla s &= \zrv\phantom{0\uv_b}\quad\mbox{in~}\Omega_\epsilon,\\
    \DIV\huv_b               &= 0\phantom{\zrv\uv_b}\quad\mbox{in~}\Omega_\epsilon,\\
    \huv_b                   &= \uv_b\phantom{0\zrv}\quad\mbox{on~}\partial\Omega,\\
    \huv_b                   &= \zrv\phantom{0\uv_b}\quad\mbox{on~}\partial(\Omega\setminus\Omega_\epsilon).
  \end{align}
\end{subequations}
Problem~\eqref{eq:bnd:pblm} is well-posed, cf.~\cite{Boffi-Brezzi-Fortin:2013}.
We further decompose the current density into its values along the boundary and the interior by:
\begin{align*}
  \hJs  = \hEs+\huv\times\Bv
  \quad\textrm{and}\quad
  \Js_b = \Es_b + \uv_b\times\Bv.
\end{align*}
The following theorem gives the continuous energy estimate. Similar
estimates are reported in
\cite{%
Liu-Wang:2001,%
Liu-Wang:2004,%
Hu-Ma-Xu:2017}.
\begin{theorem}\label{Thm:FlowContEnerEst}
  Let $(\uv,\Bv,\Es,\ps)$ solve the variational formulation
  \eqref{eq:flowContVarFormStokes}-\eqref{eq:ContVarFormAmpereOhm} in
  the time interval $[0,T]$.
  Then,
  \begin{align}
    &\frac{1}{2}\frac{d}{dt}\norm{\huv}{0,\Omega}^2
    + \frac{1}{2R_m}\frac{d}{dt}\norm{\Bv}{0,\Omega}^2
    + R_e^{-1}\norm{\nabla\huv}{0,\Omega}^2
    + \norm{\hJs}{0,\Omega}^2
    \nonumber\\[0.5em]
    &\qquad
    = \Scal{\fv}{\huv}
    - \Scal{\frac{\partial}{\partial t}\uv_b}{\huv}
    - R_e^{-1}\Scal{\nabla\uv_b}{\nabla\huv}
    - R_m^{-1}\Scal{\ROTv\Es_b}{\Bv}
    - \Scal{\Js_b}{\hJs},
    \label{Eq:ContEnerEst1}
  \end{align}
  and
  \begin{align}
    &\frac{e^{-T}}{2}\norm{\huv(T)}{0,\Omega}^2
    + \frac{e^{-T}}{2R_m}\norm{\Bv(T)}{0,\Omega}^2
    +
    \int_0^T
    \left(
    \frac{e^{-t}}{2R_e}\norm{\nabla\huv}{0,\Omega}^2
    +
    \frac{e^{-t}}{2}
    \norm{\hJs}{0,\Omega}^2
    \right)
    dt
    \nonumber\\[0.5em]
    &\qquad
    \leq
    \frac{e^{-T}}{2}\norm{\huv(0)}{0,\Omega}^2
    +
    \frac{e^{-T}}{2R_m}\norm{\Bv(0)}{0,\Omega}^2
    \nonumber\\[0.5em]
    &\qquad\qquad
    +
    \int_0^T
    \Big(
    e^{-t}R_e\NORM{\fv}{-1,\Omega}^2+
    \frac{e^{-t}}{2}\frac{d}{dt}\norm{\uv_b}{0,\Omega}^2
    +
    R_e^{-1} e^{-t}
    \norm{\nabla\uv_b}{0,\Omega}^2
    \nonumber\\[0.5em]
    &\qquad\qquad
    +
    \frac{e^{-t}}{2R_m}
    \norm{\ROTv\Es_b}{0,\Omega}^2+
    \frac{e^{-t}}{2}
    \norm{\Js_b}{0,\Omega}^2
    \Big)
    dt.
    \label{Eq:ContEnerEst2}
  \end{align}
\end{theorem}

\noindent
For the discrete version of the estimates presented in
Theorem~\ref{Thm:FlowContEnerEst}, for any $n\in [0,\, N-1]$, 
we decompose
\begin{align}
  &\quad\Esh^{n+\theta} = \hEsh^{n+\theta}+\IVh(\Es_b^{n+\theta}),\\
  &\quad\uvh^{n+1} = \huvh^{n+1}+\ITVh(\uv_b^{n+1}),
\end{align}
where $(\hEsh^{n+\theta},\huvh^{n+\theta})\in\Vhzr\times\TVhzr$ and
$\Es_b,\uv_b$ are such that their evaluations in
$\Omega\setminus\Omega_\epsilon$ are identically zero.
The condition on the boundary data is required to guarantee that the
degrees of freedom of these boundary fields all lie along the
boundary.
Next, for any $n$,
$0\le n\le N-1$, 
we define
\begin{align*}
  \hJsh^{n+\theta} &= \hEsh^{n+\theta} + \IVh(\huvh^{n+\theta}\times\PiEhRT\Bvh^{n+\theta}),\\ 
  \Jhb^{n+\theta}  &= \IVh(\Es_b^{n+\theta}) + \IVh(\uv_b^{n+\theta}\times\PiEhRT\Bvh^{n+\theta}).
\end{align*}
The next result is a discrete counterpart of
Theorem~\ref{Thm:FlowContEnerEst}.
\begin{theorem}
  \label{Thm:DiscNonlinearEnEst}
  Let $\big\{(\uvh^n,\Bvh^n)\big\}_{n=0}^N\subset\TVh\times\Eh$ and
  $\big\{(\Esh^{n+\theta},\psh^{n+\theta}))\big\}_{n=0}^{N-1}\subset
  \Vh\times\Phzr$ solve the virtual element formulation
  \eqref{eq:flowVEMVarFormConsMom}-\eqref{eq:flowDiscCurrentDensity}.
  Then, it holds that 
  \begin{equation}
    \TERM{L1} +
    \TERM{L2} =
    \TERM{R},
    \label{Eq:DiscNonlinearEnEst1}
  \end{equation}
  where
  \begin{subequations}
    \begin{align}
      \TERM{L1}
      &=
      \Dt\big(\theta-1/2\big)
      \left(
      \frac{\TNORM{\huvh^{n+1}-\huvh^{n}}{\TVh}^2}{\Dt^2} + 
      \frac{\TNORM{\Bvh^{n+1}-\Bvh^{n}}{\Eh}^2}{\Dt^2 R_m}
      \right)
      \nonumber\\[0.5em]
      &\qquad\qquad
      +\left(
      \frac{\TNORM{\huvh^{n+1}}{\TVh}^2 - \TNORM{\huvh^{n}}{\TVh}^2}{2\Dt} +
      \frac{\TNORM{\Bvh^{n+1}}{\Eh}^2   - \TNORM{\Bvh^{n}}{\Eh}^2}{2\Dt R_m}
      \right),
      \\[0.5em]
      \TERM{L2}
      &=
      R_e^{-1}\snorm{\huvh^{n+\theta}}{\TVh}^2 +
      \TNORM{\hJsh^{n+\theta}}{\Vh}^2 +
      \ScalPh{\DIV\ITVh\uv_b^{n+\theta}}{\psh^{n+\theta}},
      \\[0.5em]
      \TERM{R}
      &=
      \ScalTVh{\fvh}{\huvh^{n+\theta}} - \ScalTVh{ \frac{\ITVh\uv_b^{n+1}-\ITVh\uv_b^n}{\Dt} }{\huvh^{n+\theta}}
      \nonumber\\[0.5em]
      &\qquad\qquad
      - R_e^{-1} \ScalTVhNa{\ITVh\uv_b^{n+\theta}}{\huvh^{n+\theta}}
      - \ScalVh{\Jhb^{n+\theta}}{\hJsh^{n+\theta}}
      \nonumber\\[0.5em]
      &\qquad\qquad
      - R_m^{-1}\ScalEh{\ROTv\IVh\Es_b^{n+\theta}}{\Bvh^{n+\theta}}.
    \end{align}
  \end{subequations}
  If $\theta\in[1/2,1]$, for any $\epsilon>0$ we have that
  \begin{align}
    &
    \alpha^{N}\left(\TNORM{\huvh^N}{\TVh}^2
    + R_m^{-1}\TNORM{\Bvh^N}{\Eh}^2\right)
    \nonumber\\
    &\qquad
    + \sum_{n=0}^{N} \gamma\alpha^n\left( R_e^{-1}\snorm{\huvh^{n+\theta}}{\TVh}^2
    + \TNORM{\hJsh^{n+\theta}}{\Vh}^2
    - 2\epsilon\TNORM{\psh^{n+\theta}}{\Ph}^2\right)\Dt
    \nonumber\\[0.5em]
    &\quad
    \leq
    \left(\TNORM{\ITVh(\uv_0)}{\TVh}^2
    + R_m^{-1}\TNORM{\IEh(\Bv_0)}{\Eh}^2\right)
    \nonumber\\[0.5em]
    &\qquad
    + \sum_{n=0}^N\gamma\alpha^{n}\big(R_e\TNORM{\fvh}{-1,\TVh}^2
    + \Dt^{-1}\TNORM{\ITVh(\uv_b^{n+1}-\uv_b^n)}{\TVh}^2
    \nonumber\\[0.5em]
    &\qquad
    + R_e^{-1}\snorm{\IVh\uv_b^{n+\theta}}{\TVh}^2
    + \frac{\eta^*}{2\epsilon}\left(
    \int_{\partial\Omega}\left\vert\ITVh\uv_b^{n+\theta}\cdot\nv
    \right\vert ds\right)^2
    \nonumber\\[0.5em]
    &\qquad
    + R_m^{-1}\TNORM{\ROTv\IVh\Es_b^{n+\theta}}{\Eh}^2
    + \TNORM{\Jhb^{n+\theta}}{\Vh}^2
    \big)\Dt,
    \label{Eq:DiscNonlinearEnEst2}
  \end{align}
  where $\eta^*>0$ is given in Theorem~\ref{cor:EquivalentNormsTVh}
  and $\alpha={\theta}\slash{(1+\theta)}$ and
  $\gamma={1}\slash{(1+\theta)}$.
  Moreover, if $\uv_b\cdot\nv\equiv 0$ along $\partial\Omega$
  (non-penetrating wall condition) we obtain the final energy
  stability estimate
  \begin{align}
    \label{Eq:DiscNonlinearEnEst3}
    &\alpha^{N}\left(\TNORM{\huvh^N}{\TVh}^2
    + R_m^{-1}\TNORM{\Bvh^N}{\Eh}^2
    \right)
    +\sum_{n=0}^{N} \gamma\alpha^n\left(
    R_e^{-1}\snorm{\huvh^{n+\theta}}{\TVh}^2
    +\TNORM{\hJsh^{n+\theta}}{\Vh}^2
    \right)
    \nonumber\\[0.5em]
    &\quad
    \leq
    \left(
    \TNORM{\ITVh(\uv_0)}{\TVh}^2
    +R_m^{-1}\TNORM{\IEh(\Bv_0)}{\Eh}^2\right)
    +\sum_{n=0}^N\gamma\alpha^{n}\big( R_e\TNORM{\fvh}{-1,\TVh}^2
    \nonumber\\[0.5em]
    &\qquad
    +\Dt^{-1}\TNORM{\ITVh(\uv_b^{n+1}-\uv_b^n)}{\TVh}^2
    +R_e^{-1}\snorm{\IVh\uv_b^{n+\theta}}{\TVh}^2
    \nonumber\\[0.5em]
    &\qquad        
    +R_m^{-1}\TNORM{\ROTv\IVh\Es_b^{n+\theta}}{\Eh}^2
    +\TNORM{\Jhb^{n+\theta}}{\Vh}^2
    \big)\Dt.
  \end{align}
\end{theorem}
\label{sec:Linearization}

In this section, we are mainly concerned with the development of a
solver for the discrete problem
\eqref{eq:flowVEMVarFormConsMom}-\eqref{eq:flowDiscCurrentDensity}
at a given time instant. 
For this reason, we keep $\theta$ and $n$ fixed, and we omit them
from our notation when not strictly necessary. 
This section is based on the reference \cite{Chacon:2008}.

In practice, we manipulate arrays of degrees of freedom of virtual
element scalar and vector functions, which we represent as row vectors
and denote with the superscript $I$, e.g., $\uvhI$ is the row vector
of degrees of freedom of $\uvh$.
We introduce the finite dimensional linear space of
\emph{column} vectors
\begin{align*}
  \Xhzr = 
  \Big\{ 
  (\vvhI,\CvhI,\DshI,\qshI)^T:\,
  (\vvh,\Cvh,\Dsh,\qsh)\in
  \TVhzr\times\Eh\times
  \Vhzr\times\Phzr
  \Big\},
\end{align*}
equipped with the Euclidean ($\ell^2$) inner product.

\medskip
We pose the discrete formulation
\eqref{eq:flowVEMVarFormConsMom}-\eqref{eq:flowDiscCurrentDensity} in
the space $\Xh$.
In order to exploit symmetry in the Jacobian matrix we replace the discrete form of Faraday's Law
  \eqref{eq:flowVEMVarFormFaraday}  given by
  \begin{align}
\ScalEh{\frac{\Bvh-\Bvh^n}{\Dt}}{\Cvh}+\scalEh{\ROTv\Esh}{\Cvh}=0,
\end{align}
with
the equivalent expression
\begin{align}
  \label{eq:NewVEMVarFormFaraday}
  \theta
  R_m^{-1}\ScalEh{\frac{\Bvh-\Bvh^n}{\Dt}}{\Cvh}+\theta R_m^{-1}\scalEh{\ROTv\Esh}{\Cvh}=0,
\end{align}
and add it to \eqref{eq:flowVEMVarFormConsMom},
\eqref{eq:flowVEMVarFormConsMass} and
\eqref{eq:flowVEMVarFormOhmAmpere}.
Then, we define a function $\Gs(\cdot)$ in such a way that
$\Gs(\xvh)\cdot\yvh$ is the left hand side of the resulting
expression, where we assume that $\xvh$ and $\yvh$ are the column vector given by
\begin{align*}
  \xvh = \big(\huvh^{n+1,I},\Bvh^{n+1,I},\hEsh^{n+\theta,I},\psh^{n+\theta,I}\big)^T
  \quad\textrm{and}\quad
  \yvh = \big(\vvhI,\CvhI,\DshI,\qshI\big)^T.
\end{align*}
With these positions, the variational formulation
\eqref{eq:flowVEMVarFormConsMom}-\eqref{eq:flowDiscCurrentDensity} is
equivalent to the problem:\\ 
\emph{Find $\xvh\in\Xh$ such that}
\begin{equation}
  \Gs(\xvh)=\zrv.
  \label{eq:ZeroesProblem}
\end{equation}
Indeed, on testing~\eqref{eq:ZeroesProblem} against $\yvh =
(\vvh,\zrv,0,0)$ we retrieve~\eqref{eq:flowVEMVarFormConsMom}, and the
other three equations can be attained similarly.
This is the set up to apply a Jacobian-free Newton–Krylov method.
This method is highly parallelizable and has optimal speed of
convergence.

The Newton method at every iteration will produce an updated
estimate for the zeroes of $\Gs$ according to
\begin{align}\label{eq:InexactNewton}
  \xvh^0    &= \big(\huvh^{n,I},\Bvh^{n,I},\hEsh^{n-1+\theta,I},\psh^{n-1+\theta,I}\big)^T,\\
  \xvh^{(m+1)} &= \xvh^{(m)}+\delta\xvh^{(m)},
  \qquad \text{where} \quad
  \partial\Gs(\xvh^{(m)})\delta\xvh^{(m)} = -\Gs(\xvh^{m}),
\end{align}
where $\partial\Gs:\Xhzr\to\calL(\Xhzr)$ is the Jacobian of $\Gs$, the
space $\calL(\Xhzr)$ being the collection of bounded linear operators
from $\Xhzr$ to its dual space $\Xhzr'$.
The reason we substitute \eqref{eq:flowVEMVarFormFaraday} with
\eqref{eq:NewVEMVarFormFaraday} is to attain some symmetry in the
Jacobian matrix (which is useful in the well-posedness analysis).
The practical implementation of this method requires to compute and
store the Jacobian matrix, which may take a lot of computational power
and memory.
Instead we propose a Jacobian-Free Krylov method.

At each time step we perform a series on Newton iterations where on each iteration we solve a linear system of the form
\begin{equation}
    \partial\Gs(\xv)\delta\xv = -\Gs(\xv)
\end{equation}
We approximate $\delta\xv$ using a GMRES iteration.
One of the major benefits using GMRES is that we need not know the entries in the Jacobian matrix $\partial\Gs$.
We need only be able to compute the matrx-vector product $\partial\Gs(\xv)\delta\yv$ for any $\delta\yv\in\Xhzr$
We can approximate the action of the Jacobian matrix using the operator $\DG(\xv):\Xhzr\to\Xhzr$ defined using the finite difference approximation:
\begin{align}
  DG(\xvh)\delta\xvh
  = \frac{\Gs(\xvh+\epsilon\delta\xvh)-\Gs(\xvh)}{\epsilon},
  \label{Eq:FFDNewton}
\end{align}
with $\epsilon=10^{-7}$ (see \cite[Page~80]{Kelley:1995}).
We emphasize that $DG(\xvh)$ itself is not computed, only its action $DG(\xvh)\delta\xvh$ is.
Thus, the algorithm updates $\xvh^{(m+1)}$ from $\xvh^{(m)}$, with
$0\leq\ms\leq\Ms-1$, as follows:
\begin{subequations}
   \label{Eq:NewtonStep}
  \begin{align}
    \xvh^{(m+1)}              = \xvh^{(m)}+\delta\xvh^{(m)}, 
    &\qquad \text{where}\quad 
    DG(\xvh^{(m)})\delta \xv^{(m)} = -G(\xvh^{(m)}),\\[0.5em]
    \text{with the initial guess} \quad
    \xvh^{(0)} &=
    \begin{cases}
      \big( \huvh^{n,I}, \Bvh^{n,I}, \hEsh^{n-1+\theta,I}, \psh^{n-1+\theta,I}\big)^T & n>0,\\[0.5em]
      \big( \huvh^{0,I}, \Bvh^{0,I},                 0,                0\big)^T & n=0.
    \end{cases}
  \end{align}
\end{subequations}
We define intermediate approximations at iteration $m$ and the final values 
through the degrees of freedom of $\xvh^{(m)}$ and $\xvh^{(m)}$, respectively:
\begin{align*}
  \left(\huvh^{n+1,(m),I},\Bvh^{n+1,(m),I},\hEsh^{n+\theta,(m),I},\psh^{n+\theta,(m),I}\right)^T &= \xvh^{(m)}, \\    
  \left(\huvh^{n+1},\Bvh^{n+1}, \hEsh^{n+\theta},\psh^{n+\theta}\right) &= \xvh^{(M)}.   
\end{align*}
This Krylov method requires a user-defined input tolerance $\eta_m$
that we fix as follows:
\begin{subequations}
  \begin{align}
    &\NORM{DG(\xvh^{(m)})\delta \xv^{(m)} +G(\xvh^{(m)})}{2}\le
    \eta_m\NORM{\Gs(\xvh^{(m)})}{2},\\[0.5em]
    &\eta_m = \min\left\{
    \eta_{\mbox{max}},
    \max\left(
    \eta_m^B,
    \gamma\frac{\epsilon_t}{\NORM{G(\xvh^{(m)})}{2}}
    \right)
    \right\},\\[0.5em]
    &
    \eta_m^B = 
    \min\left\{
    \eta_{\mbox{max}},
    \max\left( 
    \eta_{m}^A,\gamma\eta_{m-1}^\alpha
    \right)
    \right\},
    \quad
    \eta_m^A =\gamma
    \left(\frac
    {\NORM{G(\xvh^{(m)})}{2}}{\NORM{G(\xvh^{(m-1)})}{2}}
    \right)^\alpha.
  \end{align}
\end{subequations}
with $\alpha = 1.5,\gamma = 0.9,\eta_{\mbox{max}} = 0.8$.
The value of $\epsilon_t$ is chosen to guarantee that the non-linear
convergence has been achieved.
\begin{align}
  \NORM{\Gs(\xvh^{(m)})}{2}&<\epsilon_{a}+\epsilon_{r}\NORM{\Gs (\xvh^{(0)})}{2}=\epsilon_t,\\[0.5em]
  \epsilon_a &=\sqrt{\# \text{dof}}\times 10^{-15,},\quad
  \epsilon_r =10^{-4}.
\end{align}
Where $\#\text{dof}$ is the sum of the number of degrees of freedom in each of our modeling spaces.
The value of these parameters is chosen in accordance
with~\cite{BeiraodaVeiga-Brezzi-Marini-Russo:2016}.
However, this strategy is much more general
\cite{Eisenstat-Walker:1996}.
The guiding philosophy being a desire to guarantee super-linear
convergence while simultaneously not over-solving with unnecessary
GMRES iterations.

The non-linear nature of the inexact Newton steps may shed doubt as to
whether or not this solver preserves the divergence free nature of the
magnetic field.
The following result is a consequence of the Faraday law.
Note that the finite difference approximation to its Jacobian is exact
since the Faraday law is linear.
\begin{theorem}
  \label{Thm:NewtonDivFormula}
  Suppose $\delta\xvh$ is a solution of the linear problem 
  \begin{align}
    \label{Eq:DivNewton1}
    DG(\xvh)\delta\xvh = -G(\xvh).
  \end{align}
  Then we have the following relation for the $\delta\Bvh$ component of $\delta\xvh$
  \begin{align}\label{Eq:DIVNewton2}
    \DIV\delta\Bvh = \DIV(\Bvh^n-\Bvh).
  \end{align}
\end{theorem}
\begin{proof}
  Testing \eqref{Eq:DivNewton1} against $\yvh = (0,\CvhI,0,0)$ yields 
  \begin{align*}
    \Dt^{-1}\ScalEh{\delta\Bvh}{\Cvh}
    + \ScalEh{\ROTv\delta\hEsh}{\Cvh}
    - \ScalEh{\frac{\Bvh-\Bvh^{n}}{\Dt}}{\Cvh}
    - \ScalEh{\ROTv\hEsh}{\Cvh}=0.
  \end{align*}
  Since $\Cvh$ can be selected arbitrarily, the relation above is
  equivalent to
  \begin{align}
    \Dt^{-1} \left[\delta\Bvh+\Bvh-\Bvh^n\right]
    = -\ROTv \left(\delta\hEsh+\hEsh\right).
  \end{align}
  The assertion of the theorem follows by taking the divergence of
  both sides.
\end{proof}
\begin{corollary}
  \label{cor:FlowDIVFreeLinearization}
  If the initial conditions on the magnetic field $\Bv_0$ satisfy that
  $\DIV\Bv_0=0$ then updates defined by \eqref{Eq:NewtonStep} will
  satisfy that
  \begin{align*}
    \DIV\delta\Bvh^{n,(m)}=0
    \quad\forall n\in[0,N],m\in[0,M].
  \end{align*}
  implying that
  \begin{align*}
    \DIV\Bvh^{n}=0
    \quad\forall n\in[0,N].
  \end{align*}
\end{corollary}
\begin{proof}
  The divergence of the initial estimate can be computed using the
  commuting property of the diagram in
  Theorem~\ref{Thm:ElectroMagDeRhamComplex}.
  Indeed,
  \begin{align*}
    \DIV\Bvh^0=\DIV\IEh(\Bv_0)=\IPh(\DIV\Bv_0)=0.
  \end{align*}
  Next, suppose that $\DIVh\Bvh^n=0$.
  Then, by definition we have that $\DIV\Bvh^{n+1,0}=0$.
  For the inductive step we can further assume that
  $\DIV\Bvh^{n+1,m}=0$, so that from
  Theorem~\ref{Thm:NewtonDivFormula} we find that
  \begin{align*}
    \DIV\Bvh^{n+1,(m+1)}
    = \DIV\Bvh^{n+1,(m)} + \DIV\delta\Bvh^{n+1,(m)}=\DIV(2\Bvh^{n+1,(m)}-\Bvh^n)
    = 0,
  \end{align*}
  which implies that assertion of the corollary.
\end{proof}
\section{Well-posedness and stability of the linear solver}
\label{sec:WellPosedness}

The linearization strategy laid out in
subsection~\ref{sec:Linearization} can be summarized as follows.
We are given a set of initial conditions.
Then, at each time step we perform a series of Newton iterations, each
one of these requiring the solution $\delta\xvh\in\Xhzr$ of a linear
system like:
\begin{align}
  \label{eq:flowLinearSystem}
  \partial\Gs(\xvh)\delta\xvh = -G(\xvh).
\end{align}
for any given $\xvh\in\Xhzr$.
To compute the Jacobian $\partial\Gs(\xvh)$ we use the definition:
\begin{equation}
  \big[\partial\Gs(\xvh) \delta\xvh\big]\cdot\yvh
  = \lim_{\epsilon\to 0} 
  \frac{\Gs(\xvh+\epsilon\delta\xvh)\cdot\yvh-\Gs(\xvh)\cdot\yvh}{\epsilon}.
\end{equation}
The limit above yields
\begin{align}
  \left[\partial\Gs(\xvh) \delta\xvh\right]\cdot\yvh =
  \ell_1(\yvh)+
  \ell_2(\yvh)+
  \ell_3(\yvh)+
  \ell_4(\yvh),
  \label{eq:Jacobian}
\end{align}
where 
$\xvh=(\huvhI,\BvhI,\hEshI,\pshI)^T$,
$\delta\xvh=(\delta\huvhI,\delta\BvhI,\delta\hEshI,\delta\pshI)^T$,\\
$\yvh=(\vvhI,\CvhI,\DshI,\qshI)^T$, and
\begin{align*}
  \ell_1(\yvh)
  &= \Dt^{-1}\ScalTVh{\delta\huvh}{\vvh}
  + \theta R_e^{-1}\ScalTVhNa{\delta\huvh}{\vvh}
  \\[0.5em]
  &\quad
  + \theta\ScalVh{\hEsh}{\IVh(\vvh\times\PiEhRT\delta\Bvh)}
  \\[0.5em]  
  &\quad
  + \theta\ScalVh{\delta \hEsh}{\IVh(\vvh\times\PiEhRT\Bvh)}
  -\ScalPh{\DIV\vvh}{\psh},
  \\[0.5em]
  \ell_2(\yvh)
  &= \theta\ScalPh{\DIV\delta\huvh}{\qsh},
  \\[0.5em]
  \ell_3(\yvh)
  & =
  \Dt^{-1}\ScalEh{\delta\Bvh}{\Cvh}+\ScalEh{\ROTv\delta\Esh}{\Cvh},
  \\[0.5em]
  \ell_{4}(\yvh)
  & =
  \ScalVh{
    \delta\hEsh+\theta\IVh(\huvh\times\PiEhRT\delta\Bvh
    +\delta\huvh\times\PiEhRT\Bvh)
  }{\Dvh}
  \\[0.5em]
  &\quad
  +R_m^{-1}\theta\scalEh{\delta\Bvh}{\ROTvh\Dsh}.
\end{align*}
These linear systems are in the form of a saddle-point problem
satisfying the hypothesis of the following theorem, which can be used
to prove the well posedness.
\begin{theorem}\label{thm:Saddle Point Problems}
  Let $U$ and $P$ be Hilbert spaces respectively endowed with the norms
  $\NORM{\cdot}{U}$ and $\NORM{\cdot}{P}$.
  Let $a:U\times U\to\REAL$, $b:U\times P\to\REAL$ be two bounded
  bilinear forms satisfying the inf-sup conditions
  \begin{equation}
    \inf_{u\in U_0}
    \sup_{v\in U_0}
    \frac{a(u,v)}{\NORM{u}{U}\,\NORM{v}{U}}>0,\quad
    \inf_{p\in P}
    \sup_{u\in U}
    \frac{b(u,p)}{\NORM{u}{U}\,\NORM{p}{P}}>0,
  \end{equation}
  where
  \begin{equation}
    U_0 = \big\{u\in U:\forall p\in P\quad b(u,p) =0\big\}.
  \end{equation}
  Then, for every pair of bounded linear functionals $f\in U'$ and
  $g\in P'$
  there exists unique $u\in U$ and $p\in P$ such that for any
  $v\in U$ and $q\in P$ it is the case that
  \begin{align*}
    a(u,v)-b(v,p) &= f(v)\phantom{g(q)}\quad\forall v\in U,\\
    b(u,q)        &= g(q)\phantom{f(v)}\quad\forall q\in P.
  \end{align*}
  Moreover there exists a constant $C>0$ independent of $f$ and $g$
  such that
  \begin{equation}
    \NORM{u}{U} + \NORM{p}{P}
    \leq C\left( \NORM{f}{U'} +\NORM{g}{P'} \right),
  \end{equation}
  with the (standard) definition of the norms in the dual spaces:
  \begin{equation}
    \NORM{f}{U'} = \sup_{u\in U\setminus\{0\}} \frac{|f(u)|}{\NORM{u}{U}},\quad
    \NORM{g}{P'} = \sup_{p\in P\setminus\{0\}} \frac{|g(p)|}{\NORM{p}{P}}.
  \end{equation}
\end{theorem}
\BEGINPROOF
A proof of this theorem can be found
in~\cite{Brezzi:1974,Boffi-Brezzi-Fortin:2013}.
\ENDPROOF

Consider the space:
\begin{equation}
  \GXh = \TVhzr\times\Eh\times\Vhzr,
\end{equation}
and the bilinear form $\ash:\GXh\times\GXh\to\REAL$, whose evaluation at
$\delta\xivh = (\delta\huvh,\delta\Bvh,\delta\hEsh)$,
$\etavh = (\vvh,\Cvh,\Dsh)$
is
given by $\ash(\delta\xivh,\etavh)=\ell_1(\vvh)+\ell_2(\Cvh)+\ell_3(\Dsh)$,
cf. equation~\eqref{eq:Jacobian}.
Here, and for the remainder of the section, we fix the value
$\xvh=(\huvh,\Bvh,\hEsh)$.
We can reformulate problem~\eqref{eq:flowLinearSystem} as:\\
\emph{Find $(\delta\xivh,\delta\psh)\in\GXh\times\Phzr$ such that for
all $(\etavh,\qsh)\in\GXh\times\Phzr$ it holds that}
\begin{subequations}
  \label{Eq:Prob1}
  \vspace{-0.5\baselineskip}
  \begin{align}
    \ash(\delta\xivh,\etavh)-\bsh(\vvh,\delta\psh) &= f(\etavh),\\
    \bsh(\delta\huvh,\qsh)                         &= g(\qsh),
  \end{align}
\end{subequations}
Where $f\in\GXh'$ and $g\in\Phzr'$ are some appropriate bounded linear
functionals and
\begin{equation}
  \bsh(\vvh,\qsh)=\scalPh{\DIV\vvh}{\qsh}.
\end{equation}

\medskip
The strategy we follow to prove the well-posedness of the virtual
element approximation proceeds in three steps:
\begin{itemize}
\item[]$(i)$ we introduce an auxiliary problem;
\item[]$(ii)$ we show that the auxiliary problem and problem
  \eqref{Eq:Prob1} are equivalent;
\item[]$(iii)$ we show that the auxiliary problem is well posed.
\end{itemize}
In the rest of this section we briefly sketch the various steps of this argument, see Chapter 5 in~\cite{Sebas2021} for details.

\medskip
The auxiliary problem is given by: \emph{Find
$(\delta\xivh,\delta\psh)\in\GXh\times\Phzr$ such that for all
$(\etavh,\qsh)\in\GXh\times\Phzr$ it holds that}
\begin{subequations}
  \label{Eq:Prob2}
  \vspace{-\baselineskip}
  \begin{align}
    \ashzr(\delta\xivh,\etavh) - \bsh(\vvh,\delta\psh) &= \fsh(\etavh),\\
    \bsh(\delta\huvh,\qsh)                             &= \gsh(\qsh).
  \end{align}
\end{subequations}
Note that 
\begin{equation}
  \label{eq:defashzr}
  \ashzr(\delta\xivh,\etavh) = \ash(\delta\xivh,\etavh) + \theta R_m^{-1}\ScalPh{\DIV\delta\Bvh}{\DIV\Cvh}.
\end{equation}
Then, to establish the equivalence between \eqref{Eq:Prob1} and
\eqref{Eq:Prob2}, we need to ensure that approximations using the auxiliary problem~\eqref{Eq:Prob2} will have divergence free magnetic fields.
This is settled in the following Theorem:
\begin{theorem}
  \label{Thm:FlowDivFreeAuxiliary}
  Let $\delta\xivh = (\delta\huvh,\delta\Bvh,\delta\hEsh)\in\GXh$ and
  $\psh\in\Phzr$ solve \eqref{Eq:Prob2}.
  If the initial conditions on the magnetic field are divergence free,
  then it holds that $\DIV\delta\Bvh=0$.
\end{theorem}
We can leverage the result of this theorem to show that both problems
\eqref{Eq:Prob1} and \eqref{Eq:Prob2} are equivalent as stated by the
following lemma.
\begin{lemma}
  \label{lem:flowequivalentproblems}
  The problems \eqref{Eq:Prob1} and \eqref{Eq:Prob2} are equivalent.
\end{lemma}
Finally, we present the well-posedness of \eqref{Eq:Prob2}.
Following the framework laid out in~\cite{Hu-Ma-Xu:2017}, we introduce the norm on $\GXhzr$ such that for any
$\xivh = (\uvh,\Bvh,\Esh)\in\GXhzr$ we have
\begin{subequations}
  \begin{align}
    \TNORM{\xivh}{\GXhzr}^2 &:= \GradNorm{\vvh}^2+
    \CurlNorm{\Esh}^2 + \DivNorm{\Bvh}^2,\label{eq:norm-Xh:def}\\
    \GradNorm{\uvh}^2 &:=
    \Dt^{-1}\TNORM{\uvh}{\TVh}^2+
    \snorm{\uvh}{\TVh}^2+
    \Dt^{-1}\TNORM{\DIV\uvh}{\Ph}^2,
    \\[0.25em]
    \DivNorm {\Bvh}^2  &: = \Dt^{-1}\TNORM{\Bvh}{\Eh}^2+\TNORM{\DIV\Bvh}{\Ph}^2,\label{eq:DivNorm}\\[0.25em]
    \CurlNorm{\Esh}^2  &:= \TNORM{\Esh}{\Vh}^2 + \Dt\TNORM{\ROTv\Esh}{\Eh}^2.\label{eq:CurlNorm:def}
  \end{align}
\end{subequations}
Well-posedness relies on Theorem~\ref{thm:Saddle Point Problems}.
The first hypothesis established that the bilinear forms in the formulation of
\eqref{Eq:Prob2} are continuous
\begin{lemma}\label{Thm:ahzrcontinuous}
  Suppose that $\Dt^{1/2}\huvh,\huvh,\Bvh\in [L^{\infty}(\Omega)]^2$
  and $\hEsh\in L^{\infty}(\Omega)$.
  Then, the bilinear form $\ashzr$ is continuous in the norms defined
  in \eqref{eq:norm-Xh:def}.
\end{lemma}
The next lemma guarantees that the bilinear form $\ashzr$ satisfies the inf-sup condition.
\begin{lemma}\label{Thm:ahzrInfSup}
  Let $\theta>0$, and $\huvh,\Bvh\in [L^{\infty}(\Omega)]^2$ and
  $\hEsh\in L^{\infty}(\Omega)$ .
  For a $\Dt$ small enough, we have that
  \begin{align*}
    \inf_{\delta\xivh\in\GXhzr}\sup_{\etavh\in\GXhzr}
    \frac{\ashzr(\delta\xivh,\etavh)}{\TNORM{\delta\xivh}{\GXh}\TNORM{\etavh}{\GXh}}\geq\Cs>0,
  \end{align*}
  where $\GXhzr=\big\{(\vvh,\Bvh,\Esh):\DIV\vvh=0\big\}$ and $C$ is a
  strictly positive, real constant independent of $\hh$ and $\Dt$.
\end{lemma}

\section{Numerical Experiments}
\label{sec:NumericalExperiments}
In this section, we show some numerical results for the approximation
of the subsystem of
\eqref{eq:flowStrongConsMass}-\eqref{eq:flowStrongAmperesLaw} that
describes the electromagnetic part of the MHD model.
\begin{subequations}
  \label{eq:NumericsStrong}
  \vspace{-\baselineskip}
  \begin{align}
    &\label{eq:NumStrongFaraday}
    \frac{\partial}{\partial t}\Bv +\ROTv\Es = \zrv\quad\mbox{in}\quad\Omega,\\[0.5em]
    &\label{eq:NumStrongAmpereOhm}
    \Es+\uv\times\Bv -R_m^{-1}\ROT\Bv = 0\quad\mbox{in}\quad\Omega,\\
    &\Bv(0) = \Bv_0 \quad \mbox{in}\quad\Omega\\
    &\Es\equiv\Es_b\quad\mbox{along}\quad
    \partial\Omega.
  \end{align}
\end{subequations}
We present an experimental study of the convergence properties of the
VEM and show the performance when we apply the VEM to the numerical
modeling of a magnetic reconnection model.

\subsection{Experimental Study of Convergence}
\label{ subsec:ConvPlots}
The first test that we perform regards the convergence rate of the
VEM.
We consider the computational domain $\Omega = [-1,1]\times[-1,1]$
partitioned by three different mesh families, a triangular mesh, a perturbed quadrilateral mesh and a Voronoi Tessellation.
We assume that an external velocity field $\uv=(\us_x,\us_y)^T$ is
imposed, whose components are
\begin{align}
  \us_x(x,y)&=-\frac{(x^2+y^2-1)(\sin(xy)+\cos(xy))-100e^{x}+100e^{y}}{2(50e^x-y\sin(xy)+y\cos(xy))},\\[0.25em]
  \us_y(x,y)&=\frac{(x^2+y^2-1)(\sin(xy)+\cos(xy))-100e^{x}+100e^{y}}{2(50e^y+x\sin(xy)-x\cos(xy))}.
\end{align}
The initial and the boundary conditions are set in accordance with the
electric and magnetic fields, which we assume as the exact solutions.
\begin{align}
  \Bv(x,y,t) &= 
  \begin{pmatrix}
    50e^y+x\sin(xy)-x\cos(xy)\\[0.25em]
    50 e^x-y\sin(xy)+y\cos(xy)
  \end{pmatrix}e^{-t},\\[0.5em]
  \Es(x,y,t) &= -\big( 50(e^x-e^y)+\cos(xy)+\sin(xy) \big) e^{-t}.
\end{align}

The simulation uses the time discretization given by $\theta = 1/2$
and time step $\Delta\ts= 0.05\hh^2$, and we integrate from $t=0$ to
$t=T$, the final time being $T=0.25$.
We measure the relative errors of $\Es$ and $\Bv$ through the mesh
dependent norms of the difference between the exact and numerical
solutions divided by the norm of the exact solution.
The results are shown in Figure~\ref{fig:test1:convergence-curves}.
In each plot, we show three different convergence curves.
These curves refer to the three different possibilities that we
presented in subsubsection~\ref{subsubsec:ObliqueProj} for the
construction of the inner product in the space $\Vh$.
These plots provide evidence that the convergence rate for the electric
field is quadratic while the convergence rate for the magnetic field
is linear.
In the case of Voronoi tessellations the convergence plots associated
with the inner product defined by the Galerkin interpolator (GI) show
some irregular behavior.
These types of meshes may have arbitrarily small edges conflicting with
the criteria normally used in the VEM.
%
Another possible explanation may have to do with the G.I, note that
this irregular behavior does not happen with the other two sample
inner products.
\begin{figure}[htb]
  \begin{center}
    \begin{tabular}{cc}
      \begin{overpic}[width=.475\textwidth]{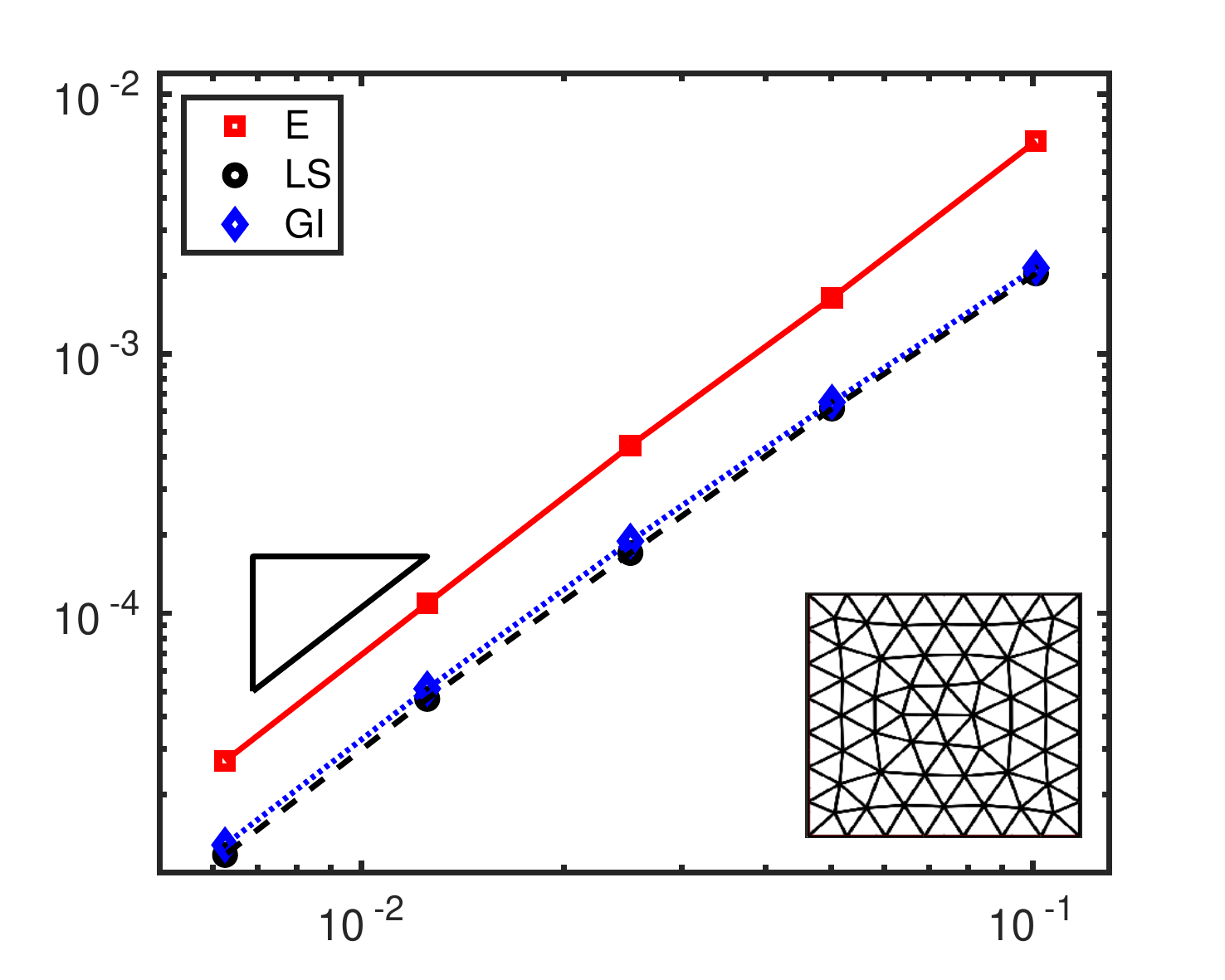} 
       \put( 0,8.5){\begin{sideways}\textbf{Electric field relative error}\end{sideways}}
        \put(40,-2){\textbf{Mesh size $\mathbf{h}$}}
        \put(17,29){\textbf{2}}
        \put(26,36){\textbf{1}}
      \end{overpic}
      & \qquad
      \begin{overpic}[width=.475\textwidth]{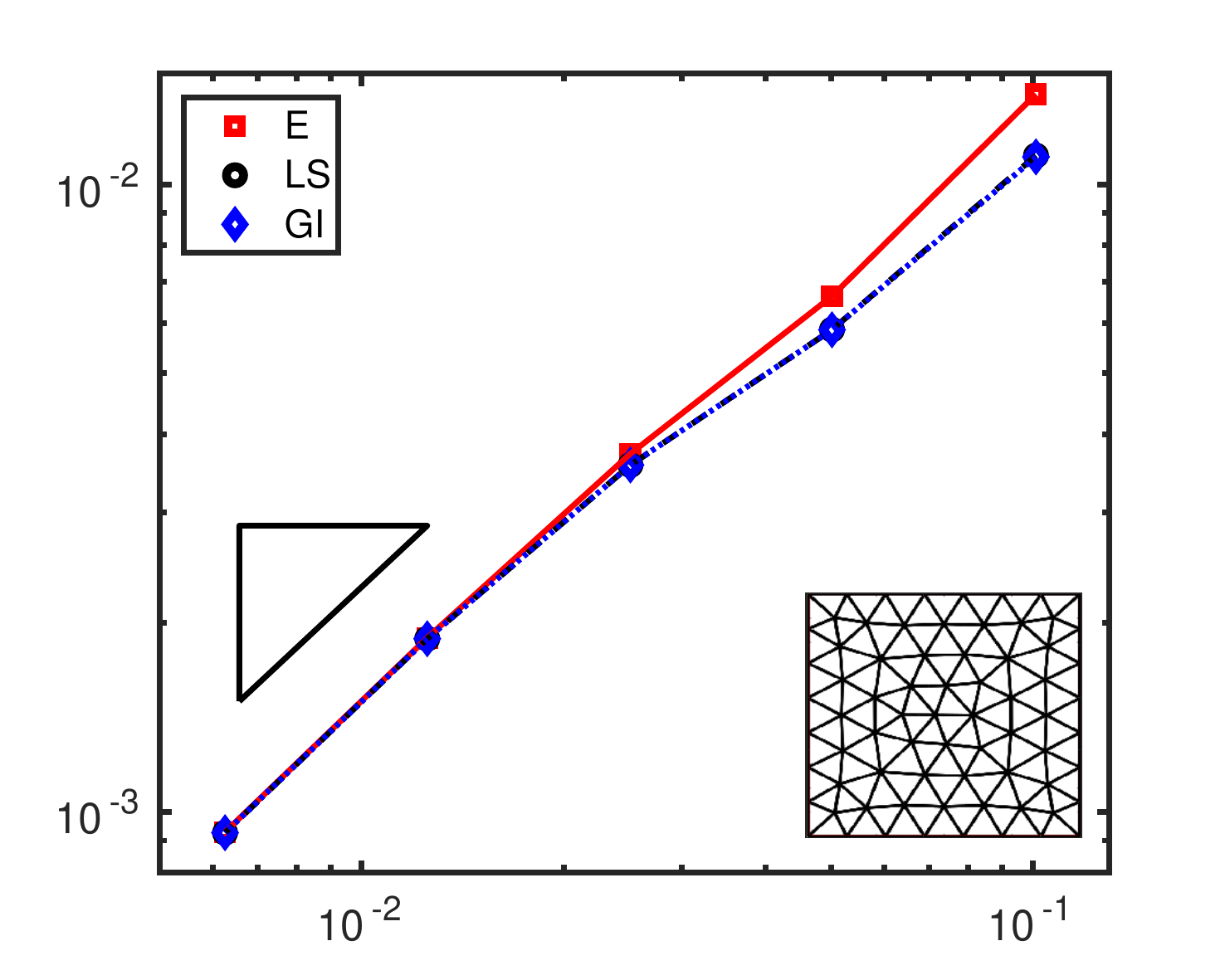} 
        \put( 0,7.5){\begin{sideways}\textbf{Magnetic field relative error}\end{sideways}}
        \put(40,-2) {\textbf{Mesh size $\mathbf{h}$}}
        \put(16,29){\textbf{1}}
        \put(25,38){\textbf{1}}
      \end{overpic}
      \\[1.5em] 
      \begin{overpic}[width=.475\textwidth]{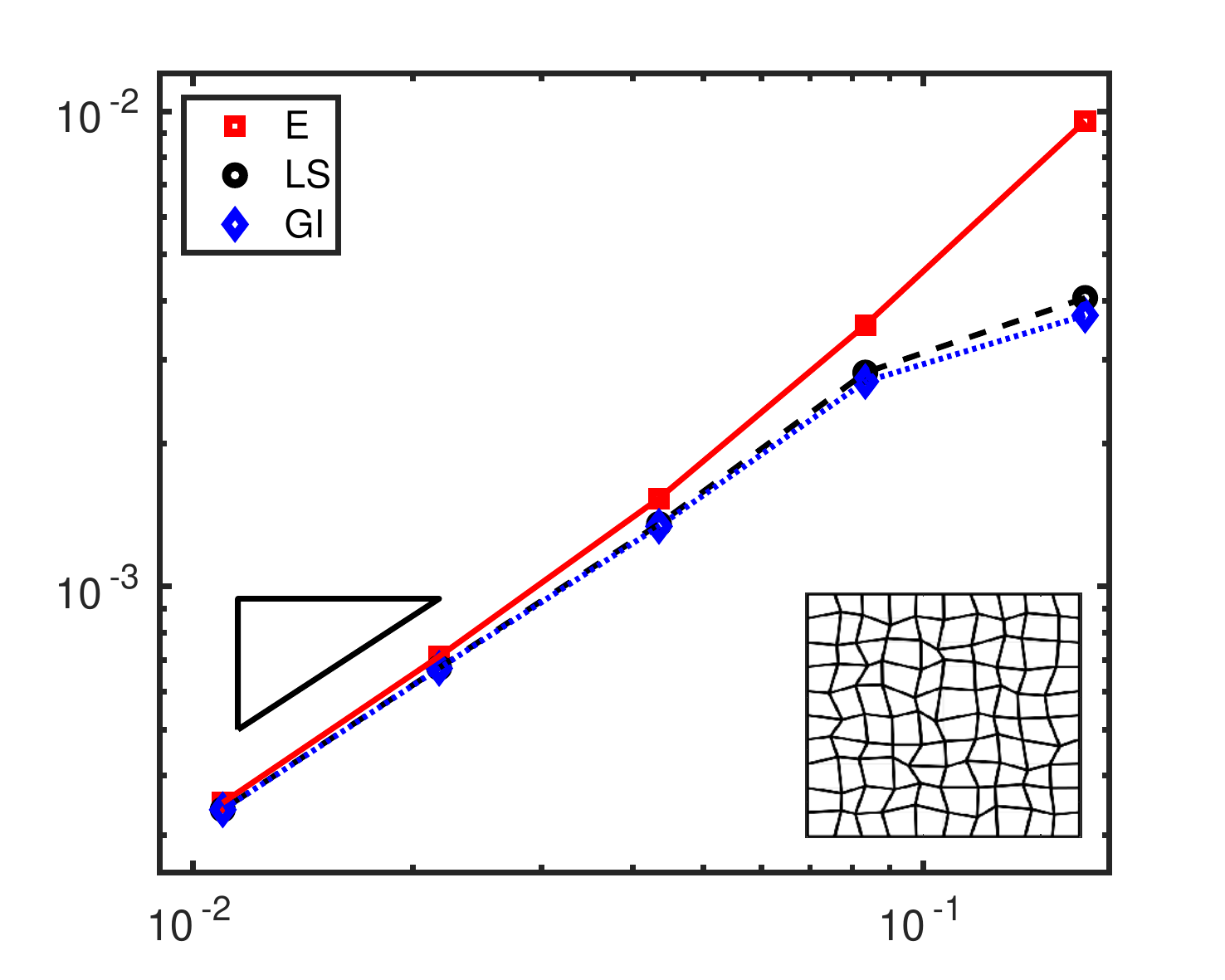} 
       \put( 0,8.5){\begin{sideways}\textbf{Electric field relative error}\end{sideways}}
        \put(40,-2) {\textbf{Mesh size $\mathbf{h}$}}
        \put(16,24){\textbf{2}}
        \put(25,32){\textbf{1}}
      \end{overpic}
      & \qquad
      \begin{overpic}[width=.475\textwidth]{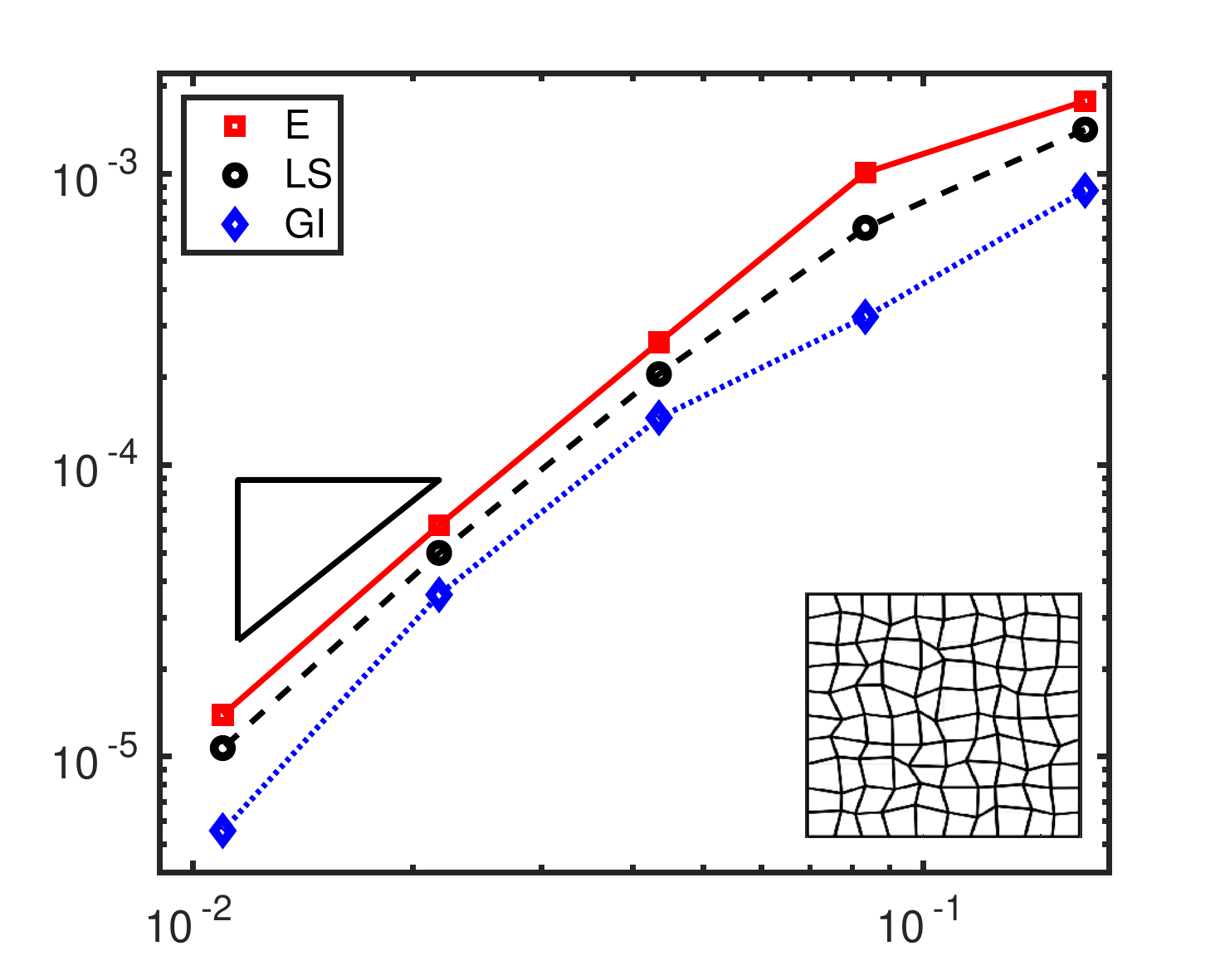} 
       \put( 0,7.5){\begin{sideways}\textbf{Magnetic field relative error}\end{sideways}}
        \put(45,-2) {\textbf{Mesh size $\mathbf{h}$}}
        \put(16,33){\textbf{1}}
        \put(25,43){\textbf{1}}
      \end{overpic}
      \\[1.5em] 
      \begin{overpic}[width=.475\textwidth]{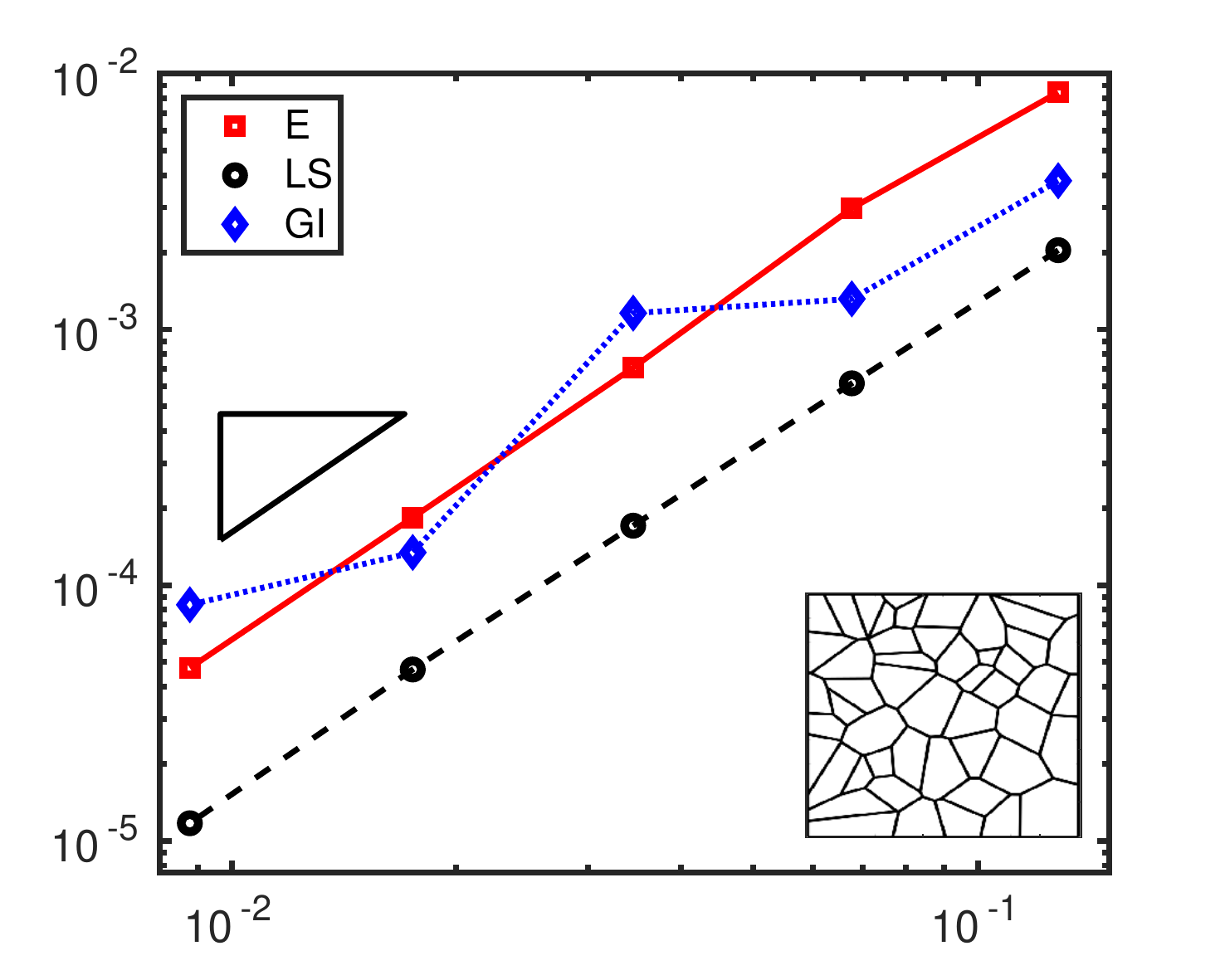} 
       \put( 0,8.5){\begin{sideways}\textbf{Electric field relative error}\end{sideways}}
        \put(40,-2) {\textbf{Mesh size $\mathbf{h}$}}
        \put(15,40){\textbf{2}}
        \put(23,47){\textbf{1}}
      \end{overpic}
      & \qquad
      \begin{overpic}[width=.475\textwidth]{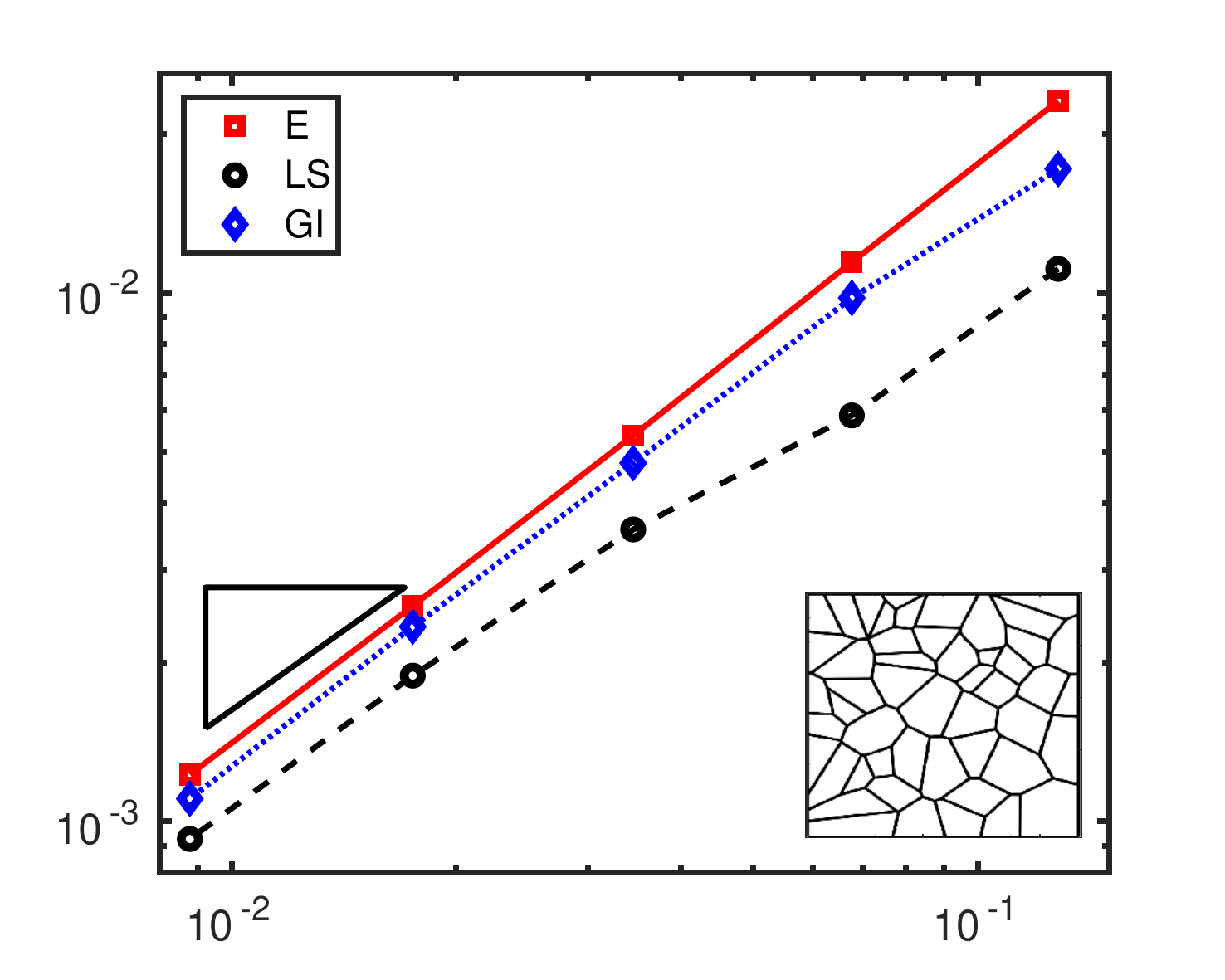} 
       \put( 0,7.5){\begin{sideways}\textbf{Magnetic field relative error}\end{sideways}}
        \put(40,-2) {\textbf{Mesh size $\mathbf{h}$}}
        \put(14,25){\textbf{1}}
        \put(23,33){\textbf{1}}
      \end{overpic}
    \end{tabular}
  \end{center}
  
  \medskip
 \caption{ Convergence plots of the VEM developed for the
   system~\eqref{eq:NumericsStrong}.
   The type of mesh used in the simulation is portrayed in the lower
   right corner of each plot.
   The three convergence curves shown in each plot show the different
   performance between the three possibilities of the inner product in
   the space $\Vh$, see subsubsection~\ref{subsubsec:ObliqueProj}.
   Each of these inner products is associated with a projector, they
   are the elliptic projector (E), the least squares projector (LS)
   and the Galerkin interpolator (GI).}
  \label{fig:test1:convergence-curves}
\end{figure}

An important feature of the VEM is that the divergence of the magnetic
field should remain zero throughout the simulation.
In Figure~\ref{fig:DivPlots} we show plots of the evolution of the
$\LTWO$-norm of the divergence of the magnetic field.
These show that this quantity remains very close, in norm, to the
machine epsilon.
\begin{figure}[htb]
  \begin{center}
    \begin{tabular}{ccc}
      \begin{overpic}[width=.345\textwidth]{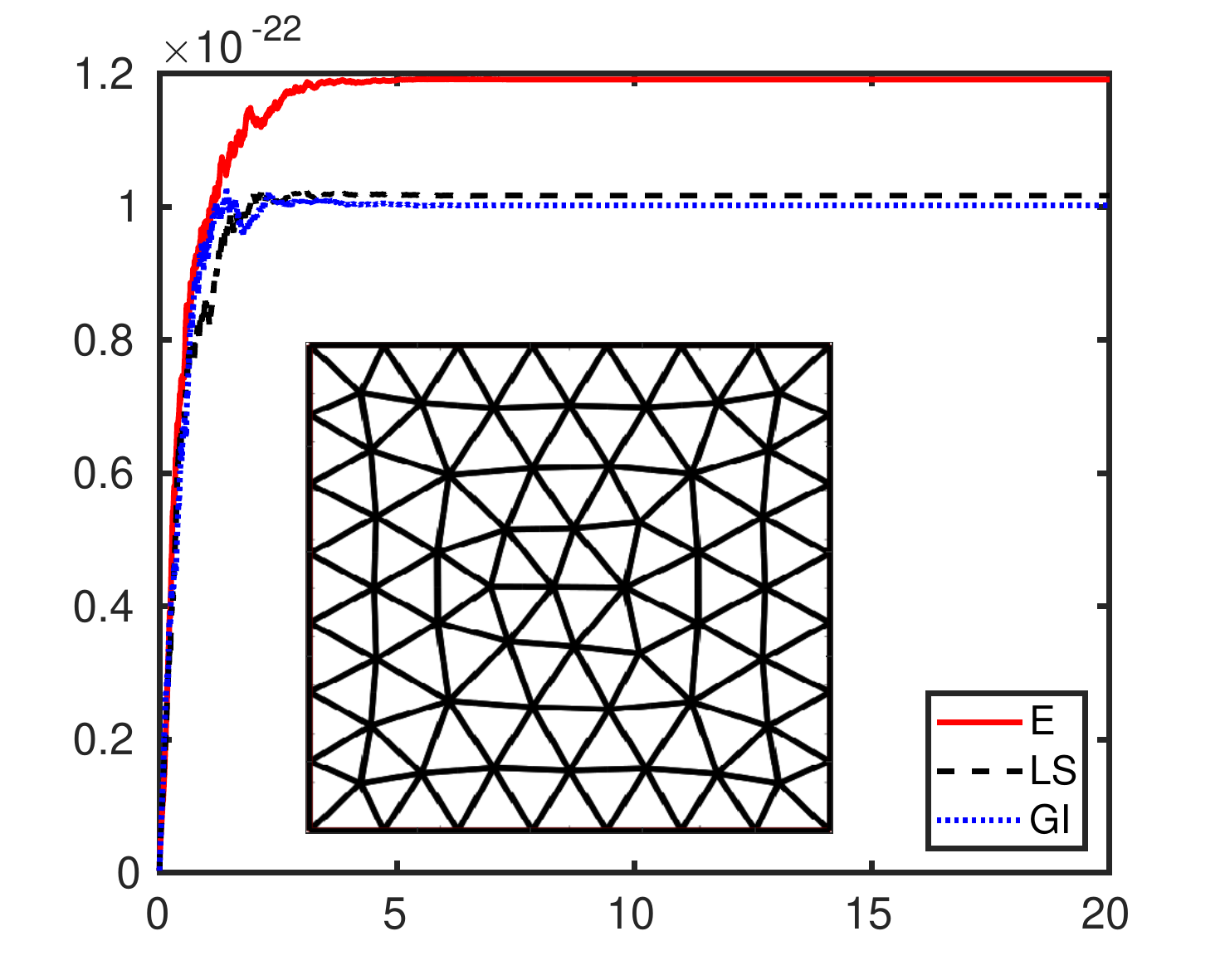} 
        \put(45,-2){{\small\textbf{Time}}}
        \put(-6,25){\begin{sideways}{\small $\TNORM{\DIV\Bvh}{\Ph}^2$}\end{sideways}}
      \end{overpic}
      &
      \hspace{-5mm}
      \begin{overpic}[width=.345\textwidth]{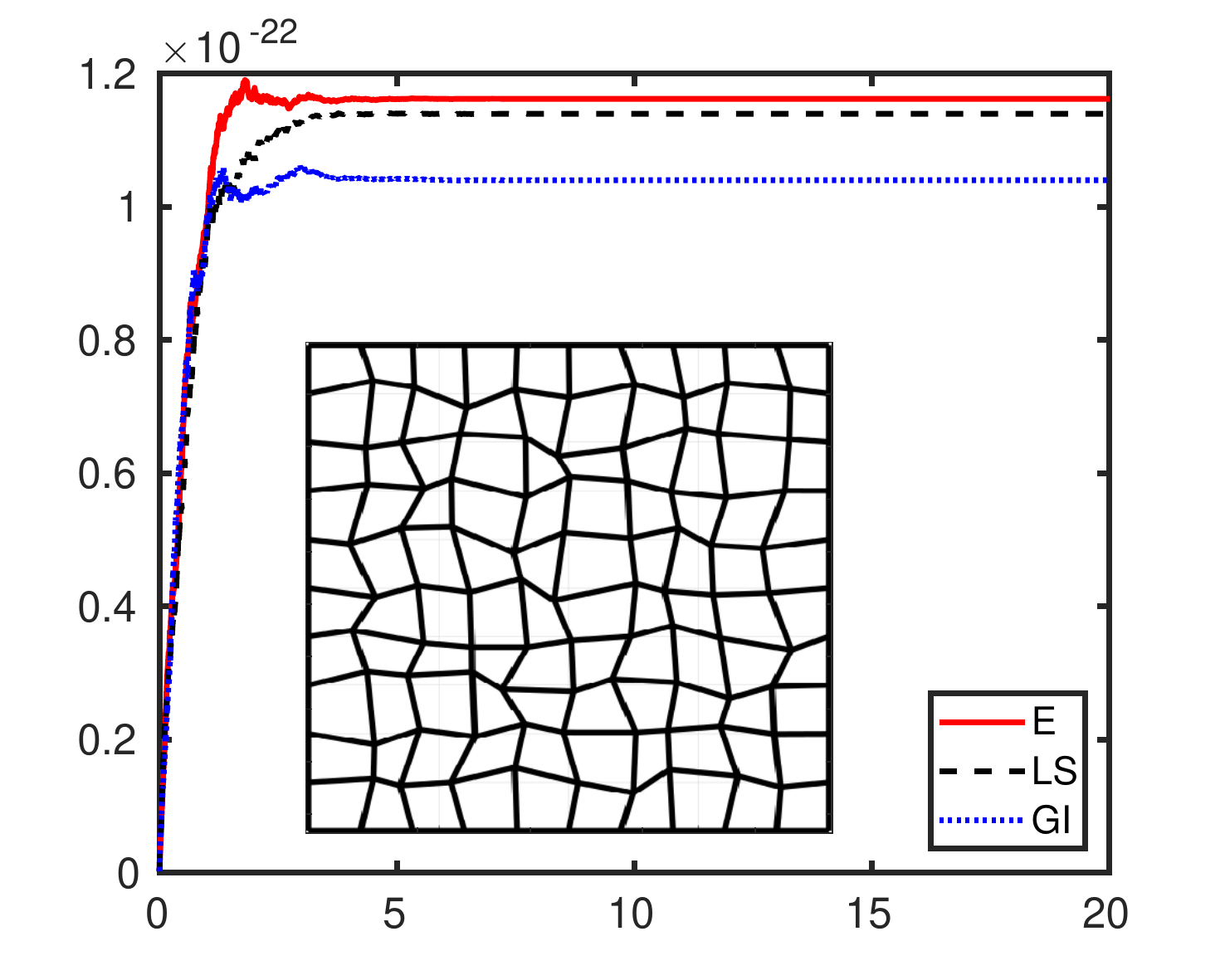} 
        \put(45,-2){{\small\textbf{Time}}}
      \end{overpic}
      &
      \hspace{-5mm}
      \begin{overpic}[width=.345\textwidth]{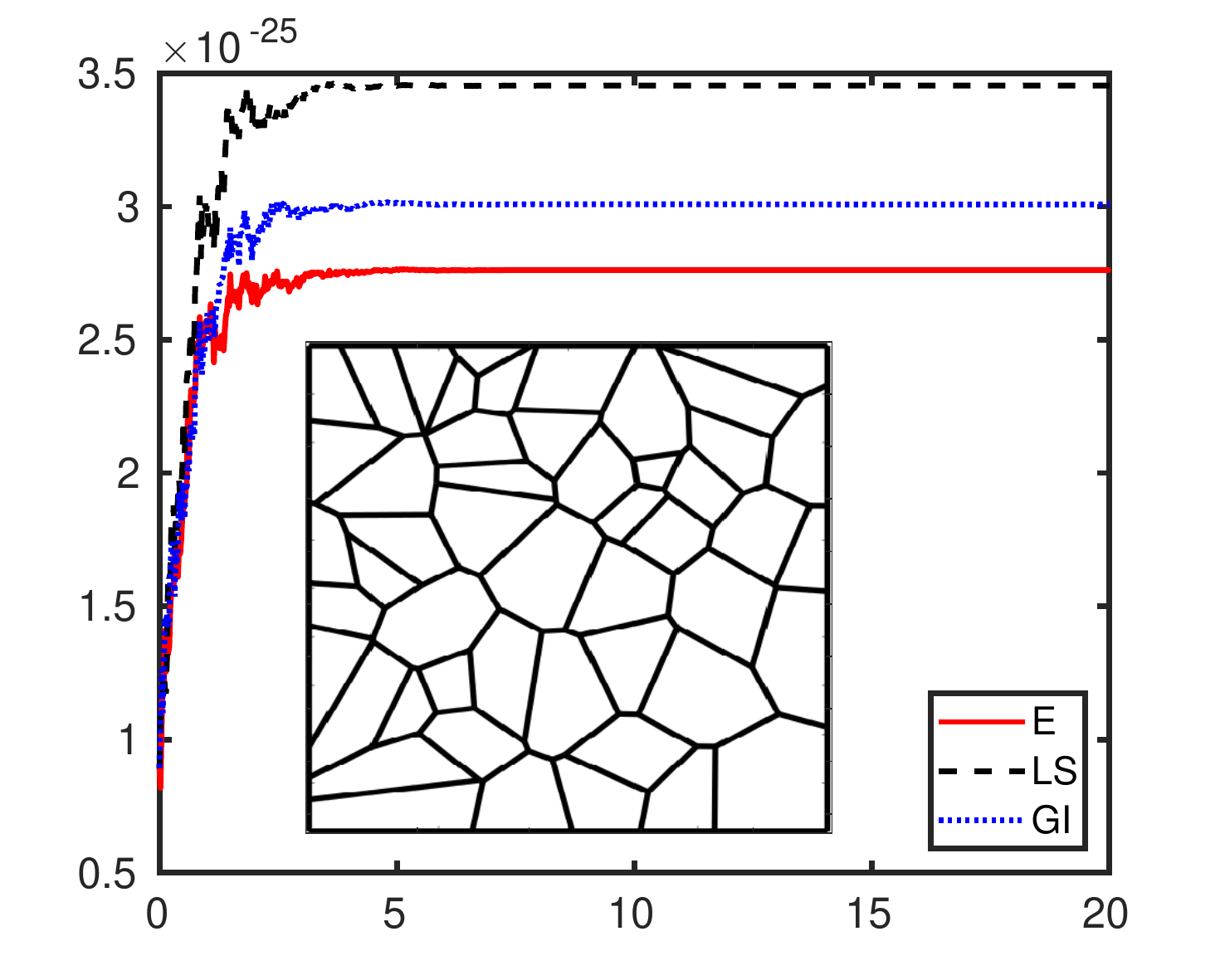} 
        \put(45,-2){{\small\textbf{Time}}}
      \end{overpic}
     \end{tabular}
  \end{center}
  \caption{Plots of the time evolution of the square of the $L^2$ norm
    of the divergence of the numerical magnetic field. We present
    three different types of meshes, these are displayed in the lower
    right hand corner of each plot.
    }
\label{fig:DivPlots}
\end{figure}
\subsection{Magnetic Reconnection}
\label{subsec:MagReconnect}

The next numerical experiment involves a characteristic feature of
resistive MHD -- the phenomenon of magnetic reconnection.
At very large scales, usually in space physics, the behavior of
plasmas can be well-approximated using ideal MHD.
In this case, according to the Alfven's Theorem, the magnetic field
lines will advect with the fluid, see Section 4.3 in
\cite{Davidson:2002} for a full discussion.
This feature is often referred to as the ``frozen-in'' condition on
the magnetic field.
In certain regions of the Earth's magnetosphere, namely the
magnetopause and magnetotail, the magnetic reconnection will lead to
very thin current sheets that separate regions across which the
magnetic field changes substantially.

In this numerical experiment, we consider one Harris sheet 
constrained to the computational domain $\Omega = [-1,1]^2$ and given by 
the following profile of the magnetic field,
first introduced in \cite{Harris:1962},
\begin{equation}\label{eq:Harris sheet}
  \Bv_0(x,y) = (\tanh{y},0).
\end{equation}
The simplicity of this condition has made it a popular choice in
modeling magnetic reconnection.  We will use the expression
\eqref{eq:Harris sheet} as the initial condition for the magnetic
field.
We will further assume that the particles in this sheet are subjected
by some external agent to a flow described by
\begin{equation}
  \uv(x,y,t) =(-x,y).
\end{equation}
This flow will force the magnetic field lines to come together at a
single point making the current density grow.
The system becomes highly unstable and the magnetic field lines begin
to tear apart, this phenomenon is called a tearing instability and
magnetic reconnection happens as a response.
This process is described in detail in
\cite{%
Kivelson-Russell:1995,%
Schindler:2006%
}.
We close this model by imposing the boundary conditions
\begin{equation}
  \forall t>0:\quad\Es_b(t)\in\PS{0}(\partial\Omega),
  \quad\mbox{and}
  \quad \int_{\partial\Omega}\Bv_b(t)\cdot\nv d\ell = 0
\end{equation}
The mesh we are using is refined near the center of the domain
$\Omega$, to guarantee higher resolution in the region of space where
the phenomenon of magnetic reconnection occurs.
The downside of using such a mesh is that a series of hanging nodes
are introduced.
This numerical experiment demonstrates the versatility of the VEM, and
the advantage of this method over more classical methods like the FEM
or FDM.

\medskip
In Fig.~\ref{fig:MagReconnectSet} we display the mesh used along with
a set of frames displaying the evolution in time of the magnetic
field. The phenomenon of magnetic reconnection begins at $T = 0$ and
by $T= 0.450$ a steady state is achieved.

\begin{figure}[htb]
  \begin{center}
    \begin{tabular}{cc}
      \begin{overpic}[width=.33\textwidth]{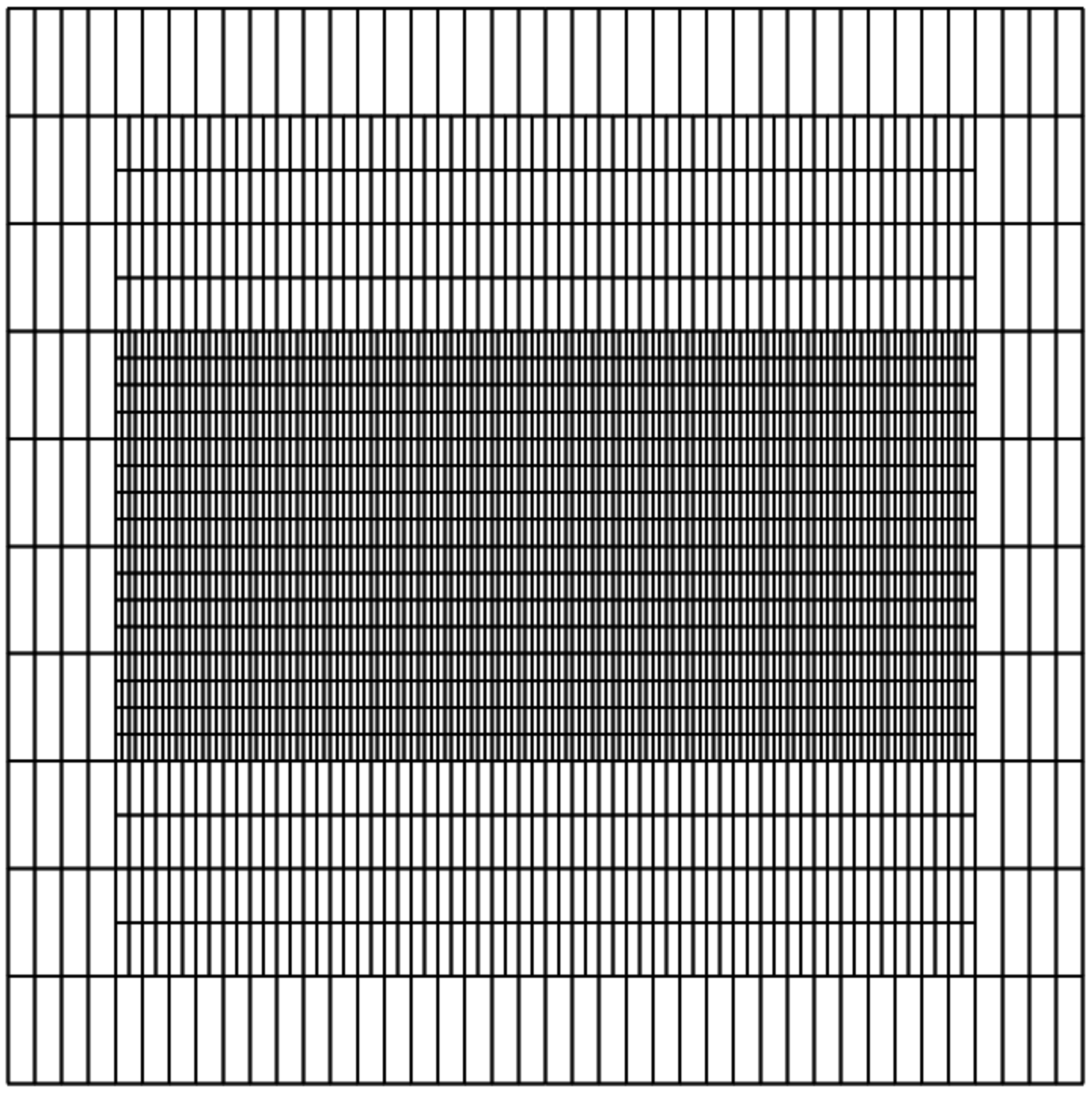} 
        \put(35,105){\textbf{Mesh Type}}
      \end{overpic}
      & \quad
      \begin{overpic}[width=.45\textwidth]{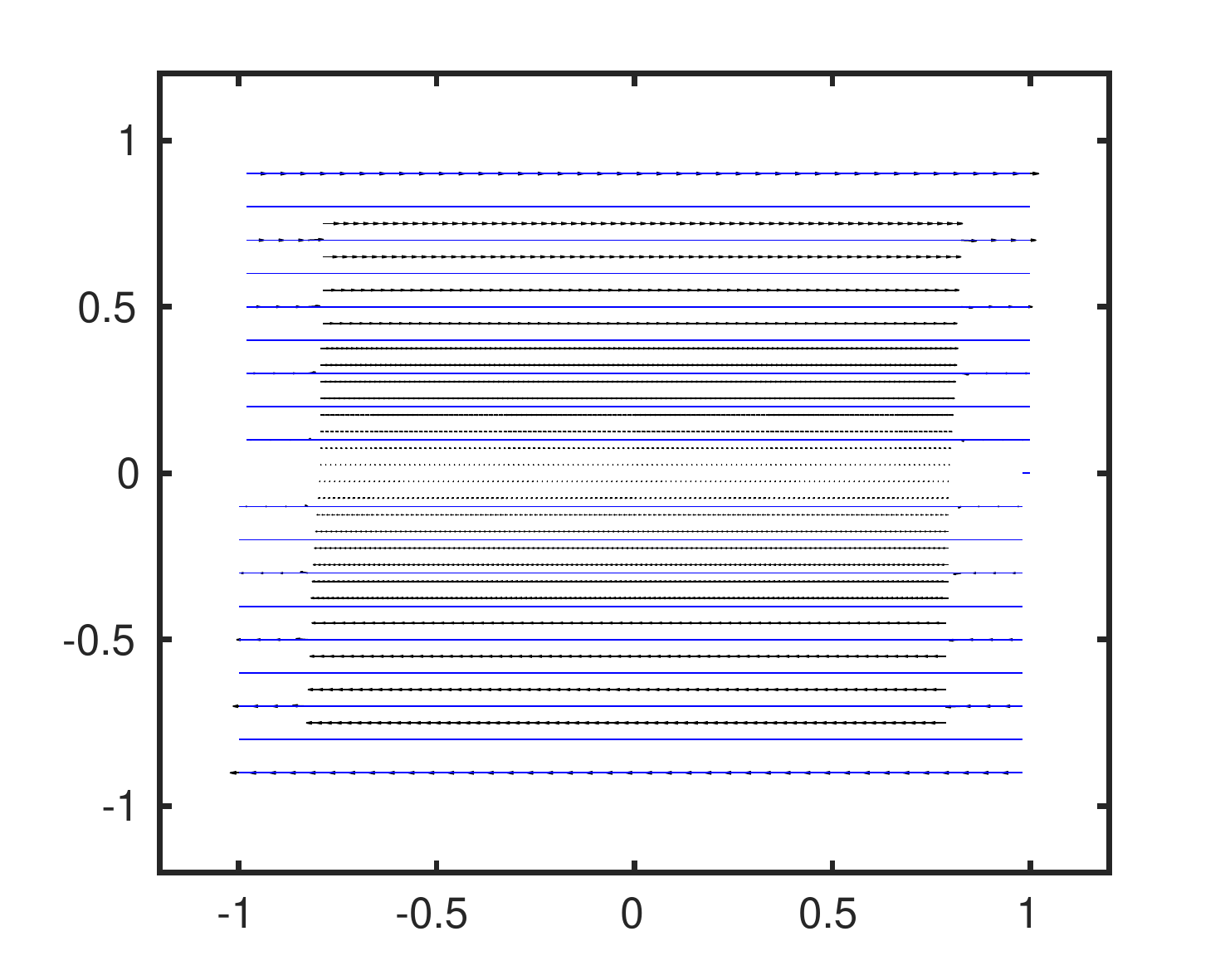} 
        \put(41,75){\textbf{T = 0}}
        \put(45,-3){\textbf{x-axis}}
        \put(-2,34){\begin{sideways}
            \textbf{y-axis}
        \end{sideways}}
      \end{overpic}\\
      \vspace{-5mm} \\[1.5em] 
      \begin{overpic}[width=.45\textwidth]{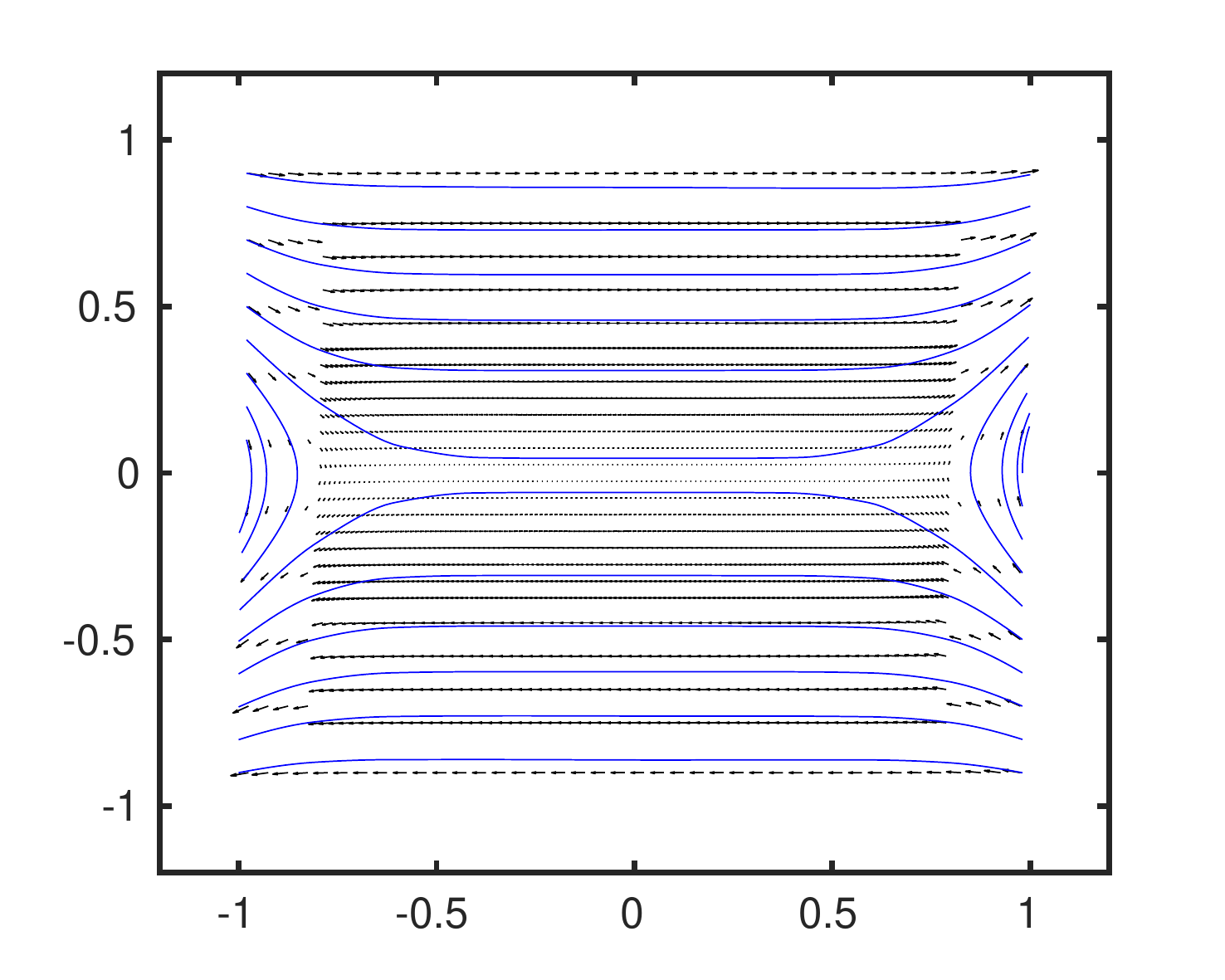} 
        \put(41,75){\textbf{T = 0.021}}
        \put(45,-3){\textbf{x-axis}}
        \put(-2,34){\begin{sideways}
            \textbf{y-axis}
        \end{sideways}}
      \end{overpic}
      & \quad
      \begin{overpic}[width=.45\textwidth]{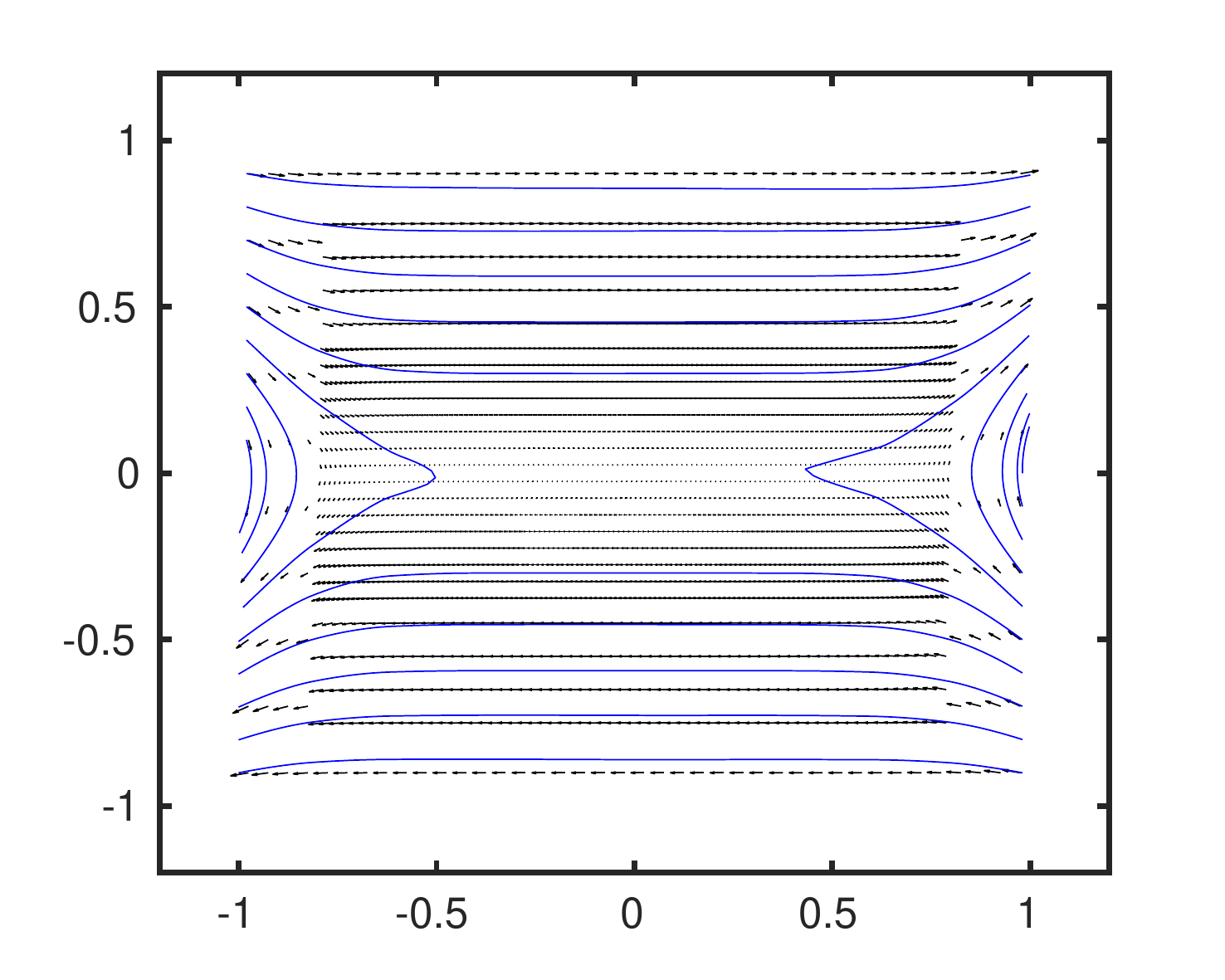} 
        \put(41,75){\textbf{T = 0.022}}
        \put(45,-3){\textbf{x-axis}}
        \put(-2,34){\begin{sideways}
            \textbf{y-axis}
        \end{sideways}}
      \end{overpic}\\
      \vspace{-5mm} \\[1.5em] 
      \begin{overpic}[width=.45\textwidth]{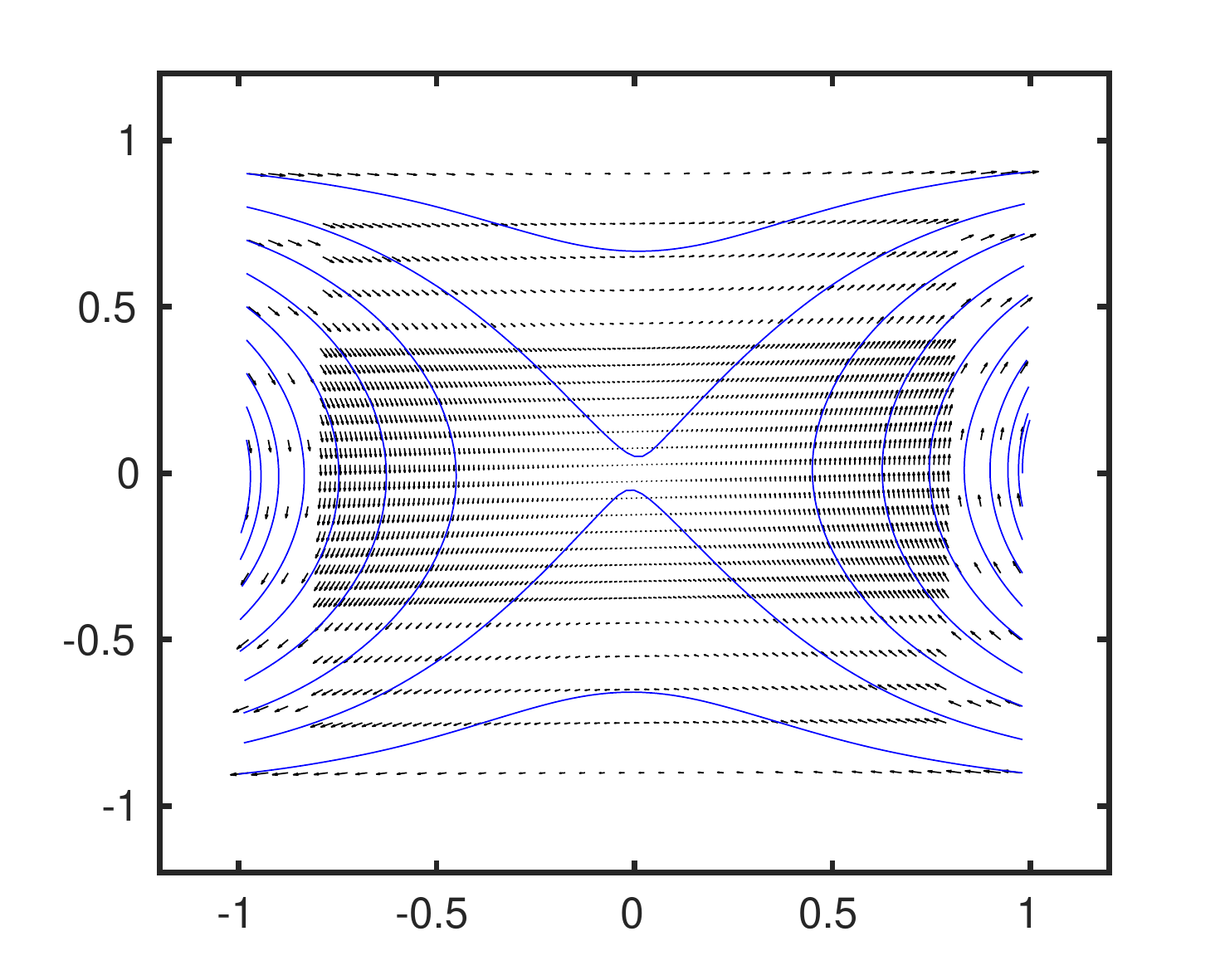} 
        \put(41,75){\textbf{T = 0.410}}
        \put(45,-3){\textbf{x-axis}}
        \put(-2,34){\begin{sideways}
            \textbf{y-axis}
        \end{sideways}}
      \end{overpic}
      & \quad
      \begin{overpic}[width=.45\textwidth]{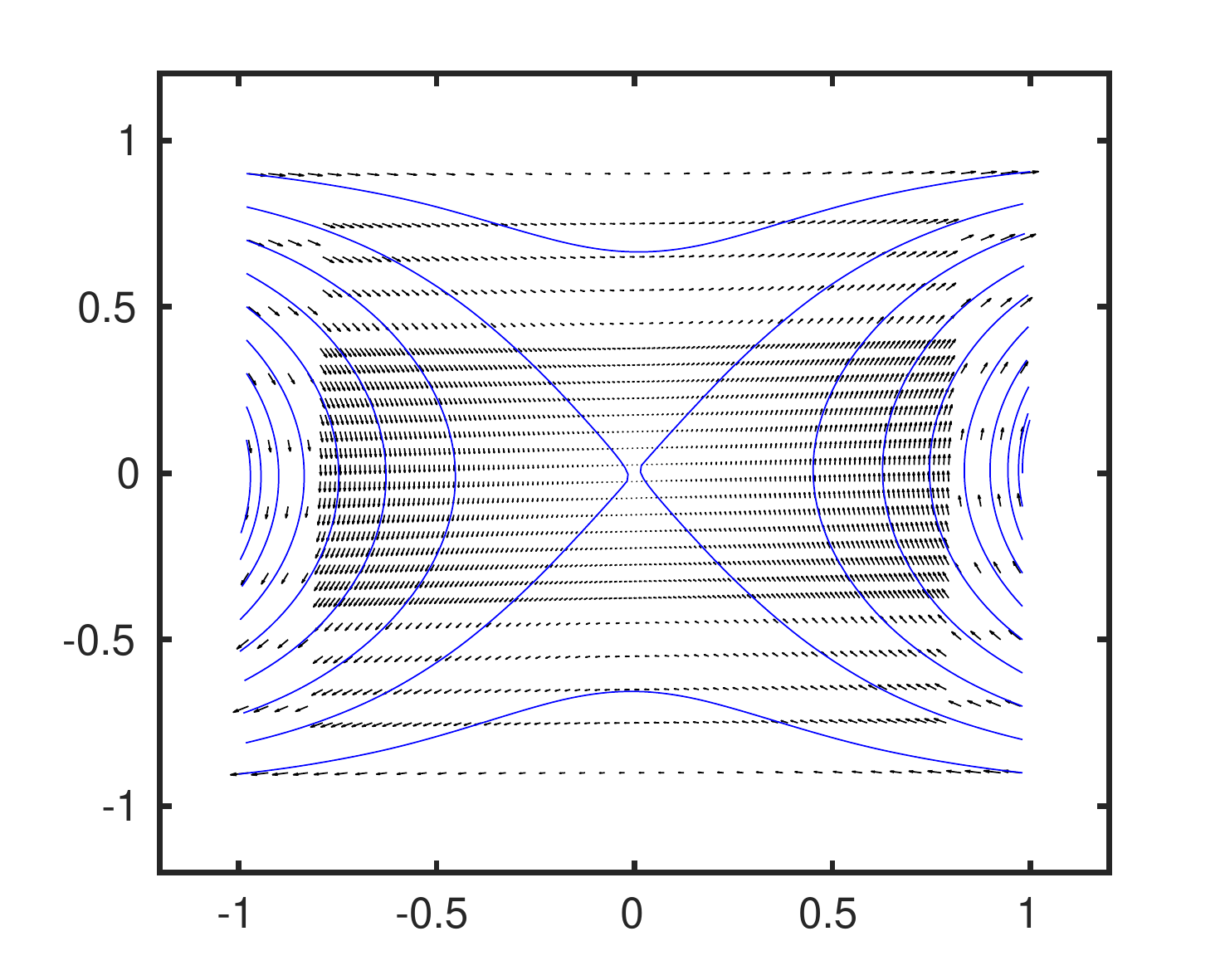} 
        \put(41,75){\textbf{T = 0.450}}
        \put(45,-3){\textbf{x-axis}}
        \put(-2,34){\begin{sideways}
            \textbf{y-axis}
        \end{sideways}}
      \end{overpic}
    \end{tabular}
  \end{center}
  \caption{Frames displaying the evolution, in time, of the magnetic
    field. The phenomenon of magnetic reconnection begins right away
    and by T= 0.450 a steady state is achieved.}
  \label{fig:MagReconnectSet}
\end{figure}
\section{Conclusions}\label{sec:Conclusions}

In this chapter, we developed a VEM for the PDE system of resistive MHD.
In developing this chapter we have introduced two chains of spaces,
see subsections~\ref{subsec:ElectroMagDeRham} and the final
results in \ref{subsec:ElementsFluidFlow}.
One chain of spaces is aimed at approximating the electromagnetics submodel, while the other applies to the  submodel for fluid flow.
There are terms in the MHD equations that couple the two submodels that require information about both the phenomena of electromagnetics and fluid flow. 
Special care needs to be taken in discretizing these terms so that the fully discretized MHD system satisfies discrete (stability) energy estimates.
These estimate are the main result of
section~\ref{sec:EnergyEstimate}.
They guarantee the stability of the method.

The VEM, when applied to MHD, yields a large system of non-linear
equations.
In order to arrive at approximate solutions a linearization strategy has to be developed.
In section~\ref{sec:Linearization} we developed a Newton iteration to
address this issue.
We were able to prove that this linearization strategy will preserve
the divergence of the magnetic field such that if the initial
conditions are divergence free then these non-linear iterations will preserve this divergence free property in the discrete mesh. 
The analysis of section~\ref{sec:WellPosedness} shows that the set of
linear systems that need to be solved are, in fact, well-posed saddle
point problems.
This well-posedness result serves as a first step into developing
robust preconditioners following the framework presented in
\cite{Loghin-Wathen:2004}.

In section~\ref{sec:NumericalExperiments} we also presented a series
of numerical experiments exploring the subsystem that describes the
electromagnetics.
These experiments show that the (lowest-order) VEM is convergent, the speed of
convergence is quadratic for the electric field and linear for the
magnetic field.
Further experimentation shows that the divergence of the magnetic
field remains well below machine epsilon.
We also present a model for a phenomenon characteristic of resistive MHD, that of magnetic reconnection.
This model accurately describes the behavior of plasmas in tokamaks as well as many electromagnetic interactions in the magnetosphere of a planet.
Our compatible discretization closely mimics the behavior of the exact solution to our model.
This model was discretized on a mesh that was refined near the center of the computational domain where reconnection will happen.
Using this mesh will provide higher resolution in the parts of the
domain where such resolution is required and aids in saving computational resources.
Other numerical methods struggle with this type of mesh because it
introduces a series of hanging nodes.
However, the VEM performs just as well in this type of mesh.

\section*{Acknowledgement}
Dr. S. Naranjo Alvarez's work was supported by the National Science
Foundation (NSF) grant \#1545188, ``\emph{NRT-DESE: Risk and
uncertainty quantification in marine science and policy}'', which
provided a one year fellowship and internship support at Los Alamos
National Laboratory.
Dr. S. Naranjo Alvarez also received graduate research funding from
Professor V. A. Bokil's DMS grant \#1720116 and \# 2012882, an INTERN
supplemental award to Professor Bokil's DMS grant \# 1720116 for a
second internship at Los Alamos National Laboratory, and teaching
support from the Department of Mathematics at Oregon State University.
In addition, S. Naranjo Alvarez was also supported by the DOE-ASCR AM
(Applied Math) base program grant for a summer internship.
  
Professor V. A. Bokil was partially supported by NSF funding from the
DMS grants \# 1720116 and \# 2012882.
  
Dr. V. Gyrya and Dr. G. Manzini were supported by the LDRD-ER program
of Los Alamos National Laboratory under project number 20180428ER.

The authors would like to thank Dr. K. Lipnikov and Dr. L. Chacon, T-5
Group, Theoretical Division, Los Alamos National Laboratory, for their
advice during the writing of this article.
  
Los Alamos National Laboratory is operated by Triad National Security,
LLC, for the National Nuclear Security Administration of
U.S. Department of Energy (Contract No. 89233218CNA000001).


\begin{thebibliography}{10}

\bibitem{Ahmad-Alsaedi-Brezzi-Marini-Russo:2013}
B.~Ahmad, A.~Alsaedi, F.~Brezzi, L.~D. Marini, and A.~Russo.
\newblock Equivalent projectors for virtual element methods.
\newblock {\em Computers \& Mathematics with Applications}, 66(3):376--391,
  2013.

\bibitem{Antonietti-BeiraodaVeiga-Scacchi-Verani:2016}
P.~F. Antonietti, L.~{Beir\~{a}o~da~Veiga}, S.~Scacchi, and M.~Verani.
\newblock A {$C^1$} virtual element method for the {C}ahn-{H}illiard equation
  with polygonal meshes.
\newblock {\em SIAM J. Numer. Anal.}, 54(1):34--56, 2016.

\bibitem{Antonietti-Manzini-Verani:2019:CAMWA:journal}
P.~F. Antonietti, G.~Manzini, and M.~Verani.
\newblock The conforming virtual element method for polyharmonic problems.
\newblock {\em Comput. Math. Appl.}, 79(7):2021--2034, 2020.

\bibitem{BeiraodaVeiga-Brezzi-Cangiani-Manzini-Marini-Russo:2013}
L.~{Beir\~ao~da~Veiga}, F.~Brezzi, A.~Cangiani, G.~Manzini, L.~D. Marini, and
  A.~Russo.
\newblock Basic principles of virtual element methods.
\newblock {\em Mathematical Models \& Methods in Applied Sciences},
  23(01):199--214, 2013.

\bibitem{BeiraodaVeiga-Brezzi-Dassi-Marini-Russo:2017}
L.~{Beir\~ao~da~Veiga}, F.~Brezzi, F.~Dassi, L.~D. Marini, and A.~Russo.
\newblock Virtual element approximation of {2D} magnetostatic problems.
\newblock {\em Computer Methods in Applied Mechanics and Engineering},
  327:173--195, 2017.

\bibitem{BeiraodaVeiga-Brezzi-Dassi-Marini-Russo:2018-CMAME}
L.~{Beir{\~a}o~da~Veiga}, F.~Brezzi, F.~Dassi, L.~D. Marini, and A.~Russo.
\newblock Lowest order virtual element approximation of magnetostatic problems.
\newblock {\em Computer Methods in Applied Mechanics and Engineering},
  332:343--362, 2018.

\bibitem{BeiraodaVeiga-Brezzi-Marini-Russo:2016}
L.~{Beir\~ao~da~Veiga}, F.~Brezzi, D.~Marini, and A.~Russo.
\newblock Mixed virtual element methods for general second order elliptic
  problems on polygonal meshes.
\newblock {\em ESAIM: Mathematical Modelling and Numerical Analysis},
  50(3):727--747, 2016.

\bibitem{BeiraodaVeiga-Brezzi-Marini-Russo:2014}
L.~{Beir\~ao~da~Veiga}, F.~Brezzi, L.~D. Marini, and A.~Russo.
\newblock The hitchhiker's guide to the virtual element method.
\newblock {\em Mathematical Models \& Methods in Applied Sciences},
  24(08):1541--1573, 2014.

\bibitem{BeiraodaVeiga-Brezzi-Marini-Russo:2016a}
L.~{Beir\~ao~da~Veiga}, F.~Brezzi, L.~D. Marini, and A.~Russo.
\newblock {H}(div) and {H}(curl)-conforming {VEM}.
\newblock {\em Numer. Math.}, 133(2):303--332, 2016.

\bibitem{BeiraodaVeiga-Brezzi-Marini-Russo:2016b}
L.~{Beir\~ao~da~Veiga}, F.~Brezzi, L.~D. Marini, and A.~Russo.
\newblock Virtual element methods for general second order elliptic problems on
  polygonal meshes.
\newblock {\em Math. Models Methods Appl. Sci.}, 26(4):729--750, 2016.

\bibitem{BeiraodaVeiga-Chernov-Mascotto-Russo:2016}
L.~{Beir\~{a}o~da~Veiga}, A.~Chernov, L.~Mascotto, and A.~Russo.
\newblock Basic principles of {$hp$} virtual elements on quasiuniform meshes.
\newblock {\em Math. Models Methods Appl. Sci.}, 26(8):1567--1598, 2016.

\bibitem{BeiraodaVeiga-Lipnikov-Manzini:2011}
L.~Beir\~ao~da Veiga, K.~Lipnikov, and G.~Manzini.
\newblock Arbitrary order nodal mimetic discretizations of elliptic problems on
  polygonal meshes.
\newblock {\em SIAM Journal on Numerical Analysis}, 49(5):1737--1760, 2011.

\bibitem{BeiraodaVeiga-Lipnikov-Manzini:2014}
L.~{Beir\~ao~da~Veiga}, K.~Lipnikov, and G.~Manzini.
\newblock {\em The Mimetic Finite Difference Method}, volume~11 of {\em MS\&A.
  Modeling, Simulations and Applications}.
\newblock Springer, {I} edition, 2014.

\bibitem{BeiraodaVeiga-Lovadina-Russo:2017}
L.~{Beir\~ao~da~Veiga}, C.~Lovadina, and A.~Russo.
\newblock Stability analysis for the virtual element method.
\newblock {\em Mathematical Models and Methods in Applied Sciences},
  27(13):2557--2594, 2017.

\bibitem{BeiraodaVeiga-Lovadina-Vacca:2017}
L.~{Beir{\~a}o~da~Veiga}, C.~Lovadina, and G.~Vacca.
\newblock Divergence free virtual elements for the {S}tokes problem on
  polygonal meshes.
\newblock {\em ESAIM: Mathematical Modelling and Numerical Analysis},
  51(2):509--535, 2017.

\bibitem{BeiraodaVeiga-Lovadina-Vacca:2018}
L.~{Beir{\~a}o~da~Veiga}, C.~Lovadina, and G.~Vacca.
\newblock Virtual elements for the {N}avier--{S}tokes problem on polygonal
  meshes.
\newblock {\em SIAM Journal on Numerical Analysis}, 56(3):1210--1242, 2018.

\bibitem{BeiraodaVeiga-Manzini:2015}
L.~{Beir\~{a}o~da~Veiga} and G.~Manzini.
\newblock Residual {\it a posteriori} error estimation for the virtual element
  method for elliptic problems.
\newblock {\em ESAIM Math. Model. Numer. Anal.}, 49(2):577--599, 2015.

\bibitem{BeiraodaVeiga-Brezzi-Dassi-Marini-Russo:2018-SINUM}
L.~Beir{\~a}o~da Veiga, F.~Brezzi, F.~Dassi, L.~D. Marini, and A.~Russo.
\newblock A family of three-dimensional virtual elements with applications to
  magnetostatics.
\newblock {\em SIAM J. Numer. Anal.}, 56(5):2940--2962, 2018.

\bibitem{Benvenuti-Chiozzi-Manzini-Sukumar:2019:CMAME:journal}
E.~Benvenuti, A.~Chiozzi, G.~Manzini, and N.~Sukumar.
\newblock Extended virtual element method for the {L}aplace problem with
  singularities and discontinuities.
\newblock {\em Comput. Methods Appl. Mech. Engrg.}, 356:571 -- 597, 2019.

\bibitem{Berrone-Pieraccini-Scialo-Vicini:2015}
S.~Berrone, S.~Pieraccini, S.~Scial{\`o}, and F.~Vicini.
\newblock A parallel solver for large scale {DFN} flow simulations.
\newblock {\em SIAM J. Sci. Comput.}, 37(3):C285--C306, 2015.

\bibitem{Boffi-Brezzi-Fortin:2013}
D.~Boffi, F.~Brezzi, and M.~Fortin.
\newblock {\em Mixed finite element methods and applications}, volume~44.
\newblock Springer, 2013.

\bibitem{Brackbill:1985}
J.~Brackbill.
\newblock Fluid modeling of magnetized plasmas.
\newblock {\em Space Science Reviews}, 42(1-2):153--167, 1985.

\bibitem{Brackbill-Barnes:1980}
J.~U. Brackbill and D.~C. Barnes.
\newblock The effect of nonzero {Div B} on the numerical solution of the
  magnetohydrodynamic equations.
\newblock {\em Journal of Computational Physics}, 35(3):426--430, 1980.

\bibitem{Brenner-Scott:2008}
S.~C. Brenner and R.~Scott.
\newblock {\em The mathematical theory of finite element methods}, volume~15.
\newblock Springer Science \& Business Media, 2008.

\bibitem{Brenner-Sung:2018}
S.~C. Brenner and L.-Y. Sung.
\newblock Virtual element methods on meshes with small edges or faces.
\newblock {\em Mathematical Models \& Methods in Applied Sciences},
  28(07):1291--1336, 2018.

\bibitem{Brezzi:1974}
F.~Brezzi.
\newblock On the existence, uniqueness and approximation of saddle-point
  problems arising from {L}agrangian multipliers.
\newblock {\em Rev. Fran\c{c}aise Automat. Informat. Recherche
  Op\'{e}rationnelle S\'{e}r. Rouge}, 8({\rm R}-2):129--151, 1974.

\bibitem{Brezzi-Buffa-Lipnikov:2009}
F.~Brezzi, A.~Buffa, and K.~Lipnikov.
\newblock Mimetic finite differences for elliptic problems.
\newblock {\em M2AN Math. Model. Numer. Anal.}, 43:277--295, 2009.

\bibitem{Cangiani-Georgoulis-Pryer-Sutton:2016}
A.~Cangiani, E.~H. Georgoulis, T.~Pryer, and O.~J. Sutton.
\newblock A posteriori error estimates for the virtual element method.
\newblock {\em Numer. Math.}, 137:857--893, 2017.

\bibitem{Cangiani-Gyrya-Manzini-Sutton:2017:GBC:chbook}
A.~Cangiani, V.~Gyya, G.~Manzini, and Sutton. O.
\newblock Chapter 14: Virtual element methods for elliptic problems on
  polygonal meshes.
\newblock In K.~Hormann and N.~Sukumar, editors, {\em Generalized Barycentric
  Coordinates in Computer Graphics and Computational Mechanics}, pages 1--20.
  CRC Press, Taylor \& Francis Group, 2017.

\bibitem{Cangiani-Manzini-Russo-Sukumar:2015}
A.~Cangiani, G.~Manzini, A.~Russo, and N.~Sukumar.
\newblock Hourglass stabilization of the virtual element method.
\newblock {\em Internat. J. Numer. Methods Engrg.}, 102(3-4):404--436, 2015.

\bibitem{Certik-Gardini-Manzini-Mascotto-Vacca:2020}
O.~Certik, F.~Gardini, G.~Manzini, L.~Mascotto, and G.~Vacca.
\newblock The p- and hp-versions of the virtual element method for elliptic
  eigenvalue problems.
\newblock {\em Comput. Math. Appl.}, 79(7):2035--2056, 2020.

\bibitem{Certik-Gardini-Manzini-Vacca:2018:ApplMath:journal}
O.~Certik, F.~Gardini, G.~Manzini, and G.~Vacca.
\newblock The virtual element method for eigenvalue problems with potential
  terms on polytopic meshes.
\newblock {\em Applications of Mathematics}, 63(3):333--365, 2018.

\bibitem{Chacon:2008}
L.~Chac{\'o}n.
\newblock An optimal, parallel, fully implicit {N}ewton--{K}rylov solver for
  three-dimensional viscoresistive magnetohydrodynamics.
\newblock {\em Physics of Plasmas}, 15(5):056103, 2008.

\bibitem{da2018lowest}
L~Beir{\~a}o da~Veiga, F~Brezzi, F~Dassi, LD~Marini, and A~Russo.
\newblock Lowest order virtual element approximation of magnetostatic problems.
\newblock {\em Computer Methods in Applied Mechanics and Engineering},
  332:343--362, 2018.

\bibitem{da2017high}
L~Beir{\~a}o Da~Veiga, Franco Dassi, and Alessandro Russo.
\newblock High-order virtual element method on polyhedral meshes.
\newblock {\em Computers \& Mathematics with Applications}, 74(5):1110--1122,
  2017.

\bibitem{Dai-Woodward:1998}
W.~Dai and P.~R. Woodward.
\newblock On the divergence-free condition and conservation laws in numerical
  simulations for supersonic magnetohydrodynamical flows.
\newblock {\em The Astrophysical Journal}, 494(1):317, 1998.

\bibitem{Dassi-Mascotto:2018}
F.~Dassi and L.~Mascotto.
\newblock Exploring high-order three dimensional virtual elements: bases and
  stabilizations.
\newblock {\em Comput. Math. Appl.}, 75(9):3379--3401, 2018.

\bibitem{Davidson:2002}
P.~A. Davidson.
\newblock An introduction to magnetohydrodynamics, 2002.

\bibitem{Dedner-Kemm-Kroner-Munz-Schnitzer-Wesenberg:2002}
A.~Dedner, F.~Kemm, D.~Kr{\"o}ner, C.~D. Munz, T.~Schnitzer, and M.~Wesenberg.
\newblock Hyperbolic divergence cleaning for the {MHD} equations.
\newblock {\em Journal of Computational Physics}, 175(2):645--673, 2002.

\bibitem{DiPietro-Droniou-Manzini:2018}
D.~A. Di~Pietro, J.~Droniou, and G.~Manzini.
\newblock Discontinuous skeletal gradient discretisation methods on polytopal
  meshes.
\newblock {\em J. Comput. Phys.}, 355:397--425, 2018.

\bibitem{Ding-Yang:2017}
J.~Ding and Y.~Yang.
\newblock Low-dispersive {FDTD} on hexagon revisited.
\newblock {\em Electronics Letters}, 53(13):834--835, 2017.

\bibitem{Eisenstat-Walker:1996}
S.~C. Eisenstat and H.~F. Walker.
\newblock Choosing the forcing terms in an inexact {N}ewton method.
\newblock {\em SIAM Journal on Scientific Computing}, 17(1):16--32, 1996.

\bibitem{Emmrich:1999}
E.~Emmrich.
\newblock {\em Discrete versions of {G}ronwall's lemma and their application to
  the numerical analysis of parabolic problems}.
\newblock Techn. Univ., 1999.

\bibitem{Harris:1962}
E.~G. Harris.
\newblock On a plasma sheath separating regions of oppositely directed magnetic
  field.
\newblock {\em Il Nuovo Cimento (1955-1965)}, 23(1):115--121, 1962.

\bibitem{Hu-Ma-Xu:2017}
K.~Hu, Y.~Ma, and J.~Xu.
\newblock Stable finite element methods preserving $div{B}=0$ exactly for {MHD}
  models.
\newblock {\em Numerische Mathematik}, 135(2):371--396, 2017.

\bibitem{Jiang:1998}
B.~Jiang.
\newblock {\em The least-squares finite element method: theory and applications
  in computational fluid dynamics and electromagnetics}.
\newblock Springer Science \& Business Media, 1998.

\bibitem{Joaquim-Scheer:2016}
M.~Joaquim and S.~Scheer.
\newblock Finite-difference time-domain method for three-dimensional grid of
  hexagonal prisms.
\newblock {\em Wave Motion}, 63:32--54, 2016.

\bibitem{Kelley:1995}
C.~T. Kelley.
\newblock {\em Iterative methods for linear and nonlinear equations},
  volume~16.
\newblock Siam, 1995.

\bibitem{Kivelson-Russell:1995}
M.~G. Kivelson and C.~T. Russell.
\newblock {\em Introduction to space physics}.
\newblock Cambridge university press, 1995.

\bibitem{Kuzmin-Klyushnev:2020}
D.~Kuzmin and N.~Klyushnev.
\newblock Limiting and divergence cleaning for continuous finite element
  discretizations of the {MHD} equations.
\newblock {\em Journal of Computational Physics}, 407:109230, 2020.

\bibitem{Lipnikov-Manzini-Shashkov:2014}
K.~Lipnikov, G.~Manzini, and M.~Shashkov.
\newblock Mimetic finite difference method.
\newblock {\em J. Comput. Phys.}, 257 -- Part B:1163--1227, 2014.
\newblock Review paper.

\bibitem{Liu-Wang:2001}
J.~G. Liu and W.~C. Wang.
\newblock An energy-preserving {MAC}--{Y}ee scheme for the incompressible {MHD}
  equation.
\newblock {\em Journal of Computational Physics}, 174(1):12--37, 2001.

\bibitem{Liu-Wang:2004}
J.~G. Liu and W.~C. Wang.
\newblock Energy and helicity preserving schemes for hydro-and
  magnetohydro-dynamics flows with symmetry.
\newblock {\em Journal of Computational Physics}, 200(1):8--33, 2004.

\bibitem{Loghin-Wathen:2004}
D.~Loghin and A.~J. Wathen.
\newblock Analysis of preconditioners for saddle-point problems.
\newblock {\em SIAM Journal on Scientific Computing}, 25(6):2029--2049, 2004.

\bibitem{Manzini-Lipnikov-Moulton-Shashkov:2017}
G.~Manzini, K.~Lipnikov, J.~D. Moulton, and M.~Shashkov.
\newblock Convergence analysis of the mimetic finite difference method for
  elliptic problems with staggered discretizations of diffusion coefficients.
\newblock {\em SIAM J. Numer. Anal.}, 55(6):2956--2981, 2017.

\bibitem{Manzini-Russo-Sukumar:2014}
G.~Manzini, A.~Russo, and N.~Sukumar.
\newblock New perspectives on polygonal and polyhedral finite element methods.
\newblock {\em Math. Models Methods Appl. Sci}, 24(8):1621--1663, 2014.

\bibitem{Mascotto:2018}
L.~Mascotto.
\newblock Ill-conditioning in the virtual element method: stabilizations and
  bases.
\newblock {\em Numer. Methods Partial Differential Equations},
  34(4):1258--1281, 2018.

\bibitem{monk2003finite}
Peter Monk et~al.
\newblock {\em Finite element methods for Maxwell's equations}.
\newblock Oxford University Press, 2003.

\bibitem{Mora-Rivera-Rodriguez:2015}
D.~Mora, G.~Rivera, and R.~Rodr\'{i}guez.
\newblock A virtual element method for the {S}teklov eigenvalue problem.
\newblock {\em Math. Methods Appl. Sci.}, 25(08):1421--1445, 2015.

\bibitem{Moreau:2013}
R.~J. Moreau.
\newblock {\em Magnetohydrodynamics}, volume~3.
\newblock Springer Science \& Business Media, 2013.

\bibitem{Munkres:2018}
J.~R. Munkres.
\newblock {\em Analysis on manifolds}.
\newblock CRC Press, 2018.

\bibitem{paper1}
S.~Naranjo-Alvarez, V.~A. Bokil, V.~Gyrya, and Manzini. G.
\newblock A virtual element method for magnetohydrodynamics.
\newblock {\em arXiv preprint arXiv:2004.11467}, 2020.

\bibitem{Sebas2021}
Sebasti\'an Naranjo~\'Alvarez.
\newblock {\em Virtual Element Methods for Magnetohydrodynamics on General
  Polygonal and Polyhedral Meshes}.
\newblock PhD thesis, Oregon State University, 2021.

\bibitem{Paulino-Gain:2015}
G.~H. Paulino and A.~L. Gain.
\newblock Bridging art and engineering using {E}scher-based virtual elements.
\newblock {\em Struct. and Multidisciplinary Optim.}, 51(4):867--883, 2015.

\bibitem{Perugia-Pietra-Russo:2016}
I.~Perugia, P.~Pietra, and A.~Russo.
\newblock A plane wave virtual element method for the {H}elmholtz problem.
\newblock {\em ESAIM Math. Model. Num.}, 50(3):783--808, 2016.

\bibitem{Schindler:2006}
K.~Schindler.
\newblock {\em Physics of space plasma activity}.
\newblock Cambridge University Press, 2006.

\bibitem{Toth:2000}
G.~T{\'o}th.
\newblock The $div {B}= 0$ constraint in shock-capturing magnetohydrodynamics
  codes.
\newblock {\em Journal of Computational Physics}, 161(2):605--652, 2000.

\bibitem{Vacca:2018}
G.~Vacca.
\newblock An ${H}^1$-conforming virtual element for {D}arcy and {B}rinkman
  equations.
\newblock {\em Mathematical Models \& Methods in Applied Sciences},
  28(01):159--194, 2018.

\bibitem{Wriggers-Rust-Reddy:2016}
P.~Wriggers, W.~T. Rust, and B.~D. Reddy.
\newblock A virtual element method for contact.
\newblock {\em Comput. Mech.}, 58(6):1039--1050, 2016.

\end{thebibliography}


\end{document}